\documentclass[final,3p,times]{article}

\usepackage[a4paper,left=22mm,top=20mm,right=22mm,bottom=35mm]{geometry}

\usepackage{authblk}
\usepackage[numbers, sort&compress]{natbib}
\usepackage[margin=30pt,font=small,labelfont=bf]{caption}
\usepackage{amsmath,amsopn,amsthm,amssymb}

\usepackage{graphicx}
\usepackage{mathptmx} 
\usepackage[mathscr]{euscript}
\usepackage{natbib}
\usepackage{mathtools}
\usepackage{amssymb}
\usepackage{wasysym}
\usepackage{float}
\usepackage{enumitem}
\usepackage{xspace}
\usepackage{url}

\usepackage{algorithm}
\usepackage[noend]{algpseudocode}
\algrenewcommand\algorithmicrequire{\textbf{Input:}}
\algrenewcommand\algorithmicensure{\textbf{Output:}}

\newtheorem{theorem}{Theorem}[section]
\newtheorem{lemma}[theorem]{Lemma} 
\newtheorem{corollary}[theorem]{Corollary}
\newtheorem{proposition}[theorem]{Proposition}
\newtheorem{example}[theorem]{Example}
\newtheorem{definition}[theorem]{Definition}
\newtheorem{observation}[theorem]{Observation}

\newtheorem*{conjecture}{Conjecture}

\usepackage{color}
\usepackage[dvipsnames]{xcolor}

\definecolor{petrol}{RGB}{0, 95, 105}
\newcommand{\mh}[1]{\begingroup #1\endgroup}
\newcommand{\NEW}[1]{\begingroup#1\endgroup}

\makeatletter
\def\moverlay{\mathpalette\mov@rlay}
\def\mov@rlay#1#2{\leavevmode\vtop{%
   \baselineskip\z@skip \lineskiplimit-\maxdimen
   \ialign{\hfil$\m@th#1##$\hfil\cr#2\crcr}}}
\newcommand{\charfusion}[3][\mathord]{
    #1{\ifx#1\mathop\vphantom{#2}\fi
        \mathpalette\mov@rlay{#2\cr#3}
      }
    \ifx#1\mathop\expandafter\displaylimits\fi}
\makeatother

\newcommand{\cupdot}{\charfusion[\mathbin]{\cup}{\cdot}}

\DeclareMathOperator{\join}{\otimes}
\DeclareMathOperator{\union}{\cupdot}

\newcommand{\PrimeCat}{\ensuremath{\vcenter{\hbox{\includegraphics[scale=0.01]{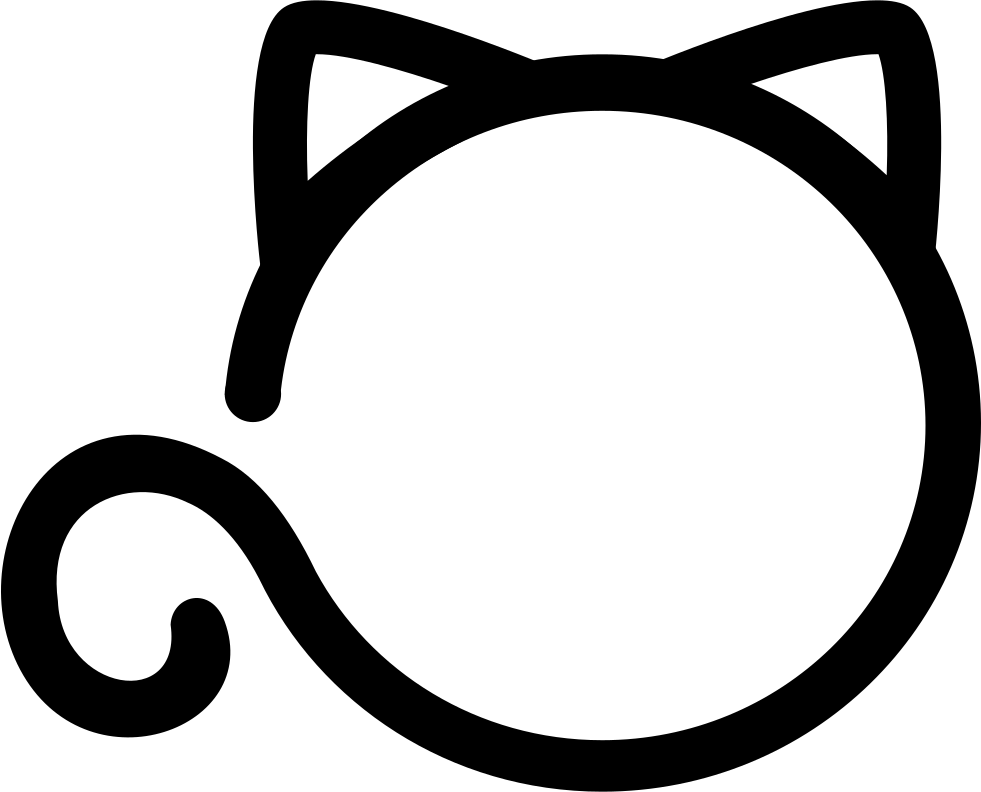}}}^{\textit{prime}}}}
\newcommand{\PolarCat}{\ensuremath{\vcenter{\hbox{\includegraphics[scale=0.01]{./cat.png}}}}}

\DeclareMathOperator{\child}{child}
\newcommand{\lca}{\ensuremath{\operatorname{lca}}}

\newcommand{\parent}{\ensuremath{\operatorname{par}}}
\newcommand{\gatex}{\textsc{GaTEx}\xspace}
\newcommand{\tww}{\ensuremath{\mathrm{tww}}}

\newcommand{\symdif}{\ensuremath{\mathop \triangle}}

\newcommand{\MD}{\ensuremath{\mathbb{M}}}
\newcommand{\MDstrong}{\ensuremath{\mathbb{M}_{\mathrm{str}}}}
\newcommand{\Mmax}{\ensuremath{\mathbb{M}_{\max}}}
\newcommand{\MDT}{\ensuremath{\mathscr{T}}}
\newcommand{\forbGT}{\ensuremath{\mathfrak{F}_{\mathrm{GT}}}}

\usepackage{algorithm}
\usepackage[noend]{algpseudocode}
\usepackage{algpseudocodex}
\algrenewcommand\algorithmicrequire{\textbf{Input:}}
\algrenewcommand\algorithmicensure{\textbf{Output:}}
\usepackage{eqparbox,array}

\newcommand{\omegaExclHyb}{\ensuremath{\omega_{\neg\eta}}}

\providecommand{\keywords}[1]{\textbf{\textit{Keywords: }} #1}

\title{Resolving Prime Modules: The Structure of Pseudo-cographs and Galled-Tree Explainable Graphs}

\author[1,*]{Marc Hellmuth} 

\author[2]{Guillaume E. Scholz} 

\affil[1]{Department of Mathematics, Faculty of Science,
  Stockholm University, SE-10691 Stockholm, Sweden 
  \newline \texttt{marc.hellmuth@math.su.se}}

\affil[2]{Bioinformatics Group, Department of Computer Science \&
    Interdisciplinary Center for Bioinformatics, Universit{\"a}t Leipzig,
    H{\"a}rtelstra{\ss}e~16--18, D-04107 Leipzig, Germany.}

\affil[*]{corresponding author}

\date{\ }

\begin{document}
\sloppy

\maketitle

\abstract{ 
The modular decomposition of a graph $G$ is a natural construction to capture
key features of $G$ in terms of a labeled tree $(T,t)$ whose vertices are
labeled as ``series'' ($1$), ``parallel'' ($0$) or ``prime''. However, full
information of $G$ is provided by its modular decomposition tree $(T,t)$ only,
if $G$ is a cograph, i.e., $G$ does not contain prime modules. In this case,
$(T,t)$ explains $G$, i.e., $\{x,y\}\in E(G)$ if and only if the lowest common
ancestor $\mathrm{lca}_T(x,y)$ of $x$ and $y$ has label ``$1$''.
Pseudo-cographs, or, more general, \textsc{GaTEx} graphs $G$ are graphs that can
be explained by labeled galled-trees, i.e., labeled networks $(N,t)$ that are
obtained from the modular decomposition tree $(T,t)$ of $G$ by replacing the
prime vertices in $T$ by simple labeled cycles. \textsc{GaTEx} graphs can be
recognized and labeled galled-trees that explain these graphs can be constructed
in linear time.

In this contribution, we provide a novel characterization of \textsc{GaTEx}
graphs in terms of a set $\forbGT$ of 25 forbidden induced subgraphs. This
characterization, in turn, allows us to show that \gatex graphs are closely
related to many other well-known graph classes such as $P_4$-sparse and $P_4$-reducible
graphs, weakly-chordal graphs, perfect graphs with perfect order, comparability
and permutation graphs, murky graphs as well as interval graphs, Meyniel graphs
or very strongly-perfect and brittle graphs. \NEW{Moreover, we show that every
\gatex graph as twin-width at most 1} and provide linear-time algorithms 
to solve several NP-hard problems (clique, coloring, independent set)
on \gatex graphs by utilizing the structure of the underlying galled-trees
they explain.
}

\smallskip
\noindent
\keywords{gatex graph, pseudo-cograph, galled-tree, graph classes, forbidden subgraphs, \NEW{twin-width}}

\section{Introduction}

Cographs can be generated from the single-vertex graph $K_1$ by complementation
and disjoint union \cite{Corneil:81,Sumner74,Seinsche:74}. This definition
associates a vertex labeled tree, the cotree, with each cograph, where a vertex
label ``0'' denotes a disjoint union, while ``1'' indicates a complementation.
Each cograph $G$ is \emph{explained} by its cotree $(T,t)$, i.e., $\{x,y\}$ is
an edge in $G$ precisely if the lowest common ancestor $\lca_T(x,y)$ of $x$ and
$y$ has label $t(\lca_T(x,y))=1$. Cographs are exactly the graphs containing no
induced $P_4$ (a chordless path with four vertices) and many NP-complete
problems become polynomial-time solvable, when the input is a cograph
\cite{Corneil:81,Brandstadt1999,GAO20132763}.

Cotrees are a special case of the much more general modular decomposition tree
(MDT) \cite{HP:10,CS:99,DGC:01,EHMS:94}. A \emph{module} $M$ in a graph $G$ is a
subset of the vertices of $G$ such that every $z\in V(G)\setminus M$ is either
adjacent to all or none of the vertices in $M$. A prime module $M$ is a module
that is characterized by the property that both, the induced subgraph $G[M]$ and
its complement $\overline{G[M]}$, are connected subgraphs of $G$. Cographs are
characterized by the absence of prime modules. Hence, for non-cographs, the MDT
has, in addition to the labels ``$0$'' and ``$1$'', vertices that are labeled as
$\mathit{prime}$. As a consequence, not all structural information of
non-cographs $G$ is provided by its MDT, i.e., $G$ cannot be recovered just by
the information of its MDT alone. It is natural to ask, therefore, whether the
MDT can be modified in a such a way that vertices $v$ with label
$\mathit{prime}$ are resolved by relabeling the vertices of the MDT or by
replacing $v$ by other labeled graphs to obtain a labeled tree or network that
explains $G$. The latter idea has been, in particular, fruitful for graphs with
a ``bounded number'' of $P_4$s \cite{JO:91,JO:89,JO:92}, cumulating in a
framework for trees with four labels to explain general graphs
\cite{JO:93,JO:95}. 

An independent motivation to manipulate the MDT of a given graph recently arose
in biology, more precisely in molecular phylogenetics
\cite{HSW:16,LDEM:16,Hellmuth:15a,Geiss:18a,Lafond2014,HHH+13,geiss2020best}. In
particular, orthology, a key concept in evolutionary biology and phylogenetics,
is intimately tied to cographs \cite{HHH+13}. Two genes in a pair of related
species are said to be orthologous if their last common ancestor was a
speciation event. The orthology relation on a set of genes forms a cograph
\cite{HHH+13,HW:16}. This relation can be estimated directly from biological
sequence data, albeit in a necessarily noisy form. In a simple evolutionary
scenario, such estimates are biologically feasible if one can find a tree that
supports these estimates, i.e., the inferred relations can be explained by a
0/1-labeled tree \cite{HW:16,Hellmuth:17,LDEM:16,LDEM:16}. However, in practice,
such estimates can often not be explained by any such tree since they violate
the cograph property. There are two main reasons why such estimates are not
cographs: (1) Noise in the data and measurement errors or (2) the history of the
underlying genes is simply not a tree but a phylogenetic network. In fact,
phylogenetic networks are a common framework to accomodate further events such
as horizontal gene transfer or hybridization. Assuming Item (2) to be one of the
main source of error, it makes sense, therefore, to ask if there is a
0/1-labeled phylogenetic network that can explain estimates of homology
relations. A first attempt to answer this question is provided in
\cite{HS18}, where the authors assumed to have, in addition to the orthology
relation, information about the ``behavior'' of 3-elementary subsets of genes.

In a recent work \cite{HS:22}, graphs $G$ have been characterized that can
    be explained by 0/1-labeled networks that are obtained from the MDT of $G$
    by replacing the prime vertices by a simple 0/1-labeled cycle which leads to
    the concept of pseudo-cographs and, more general, galled-tree explainable
    (\gatex) graphs. It has been shown, that such graphs can be recognized and
    0/1-labeled networks to explains these graphs can be constructed in linear
    time. In this contribution, we show that \gatex graphs are characterized by
    the absence of 25 induced forbidden subgraphs. This, in turn, allows us to
    show the rich connection of \gatex graphs to other graph classes. While
    $P_4$-reducible graphs form a proper subclass of \gatex graphs, all \gatex
    graphs are weakly-chordal, perfect graphs with perfect order, comparability
    and permutation graphs, as well as murky. Further connections to
    distance-hereditary graphs, interval and circular-arc graphs as well as
    Meyniel, very strongly-perfect and brittle graphs are established.

\begin{figure}[t]
	\begin{center}
			\includegraphics[width=0.75\textwidth]{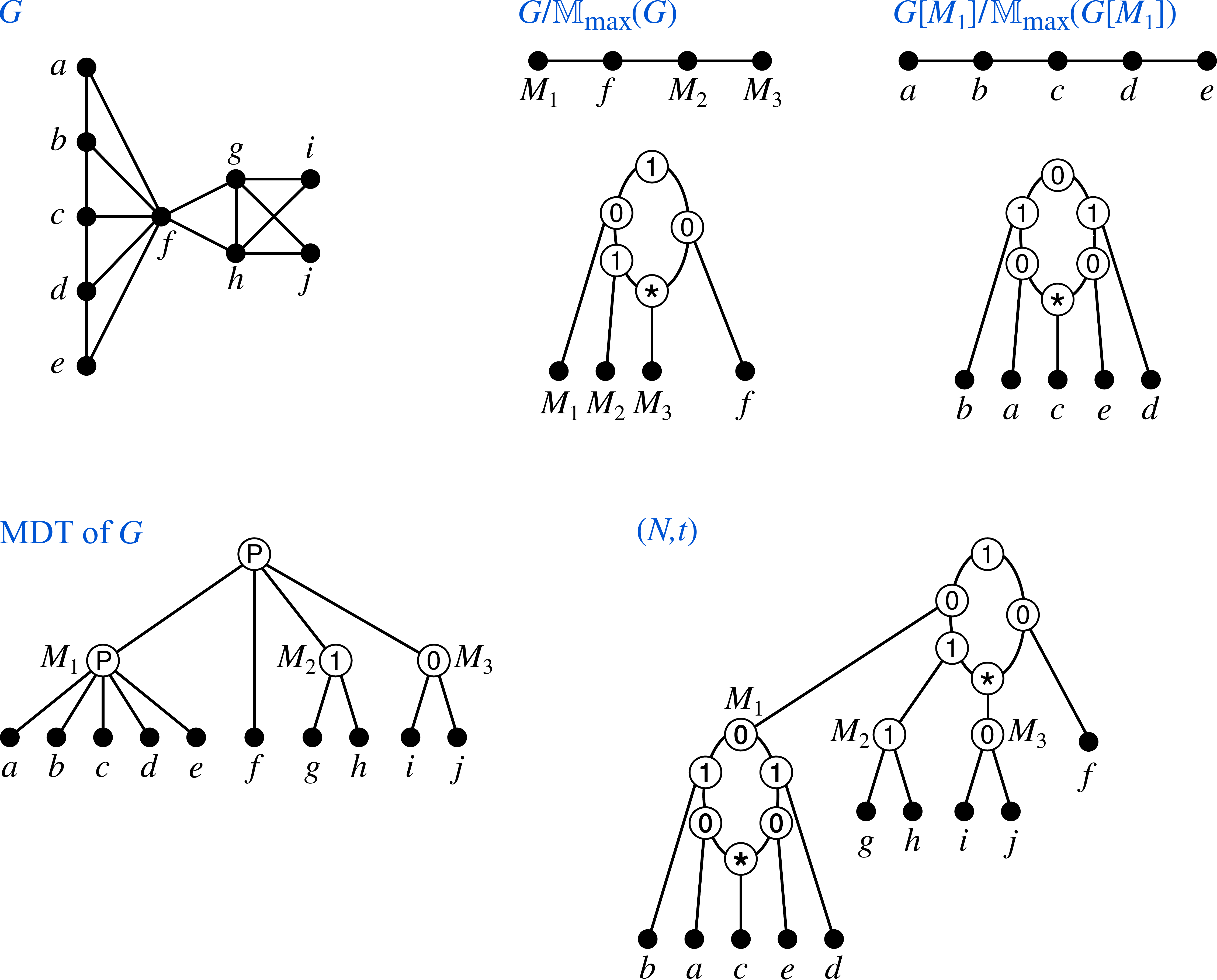}
	\end{center}
	\caption{
		Shown is a galled-tree explainable (\gatex) graph $G$. The modular
		decomposition tree (MDT) of $G$ shows that $G$ contains two prime
		modules, namely $V(G)$ and $M_1 = \{a,b,c,d,e\}$. To keep the drawing simple, we have
		written all singleton modules $\{x\}$ just by $x$. The
		respective quotient graphs $H\coloneqq G/\Mmax(G)$ and $H'\coloneqq
		G[M_1]/\Mmax(G[M_1])$ are shown in the upper part together with the
		respective galled-trees that explain them (directly below the
		respective graphs). One easily verifies that the 0/1-labeled galled-tree
		$(N,t)$ that explains $G$ is obtained from the MDT of $G$, by locally
		replacing all vertices in the MDT with label ``$\mathit{prime}$ (P)'' by
		the respective cycles of the networks explaining $H$ and $H'$ and by
		rewiring the edges accordingly. \newline
		Note that $H$ and $H'$ are polar-cats and thus,
		primitive pseudo-cographs. In particular, $H$ is a
		$(M_3,H[M_1,M_2,M_3],H[f,M_3])$-pseudo-cograph while $H'$ is a
		$(c,G[\{a,b,c\}],G[\{c,d,e\}])$ pseudo-cograph. Moreover, $H'$ is
		well-proportioned, while $H$ is not. Hence, there are further ways to
		write $H$ as a $(v,H_1,H_2)$-pseudo-cograph. Consequently, there are further
		galled-trees that explain $H$ (cf.\ Prop.\ 6.13 and Fig.\ 8 in
		\cite{HS:22}). 
			 }
\label{fig:gatex}
\end{figure}

\section{Preliminaries}
\label{sec:prelim}

\paragraph{\bf Graphs}

We consider graphs $G=(V,E)$ with vertex set $V(G)\coloneqq V\neq \emptyset$ and
edge set $E(G)\coloneqq E$. A graph $G$ is \emph{undirected} if $E$ is a subset
of the set of two-element subsets of $V$ and $G$ is \emph{directed} if
$E\subseteq V\times V$. Thus, edges $e\in E$ in an undirected graph $G$ are of
the form $e=\{x,y\}$ and in directed graphs of the form $e=(x,y)$ with $x,y\in
V$ being distinct. An undirected graph is \emph{connected} if, for every two
vertices $u,v\in V$, there is a path connecting $u$ and $v$. A directed graph is
\emph{connected} if its underlying undirected graph is connected. We write $H \subseteq G$ if $H$ is a
subgraph of $G$ and $G[W]$ for the subgraph in $G$ that is induced by some
subset $W \subseteq V$. If $H$ is an induced subgraph of $G$, we also write
$H\sqsubseteq G$. Moreover $G-v$ denotes the induced subgraph $G[V\setminus
\{v\}]$. In addition, if $H\sqsubseteq G$ and $x\in V(G)$, we define $H+x
\coloneqq G[V(H)\cup \{x\}]$. 
\NEW{Let $G$ be a directed or undirected graph. Then, $G$ 
is \emph{biconnected} if it contains no vertex
whose removal disconnects $G$. A \emph{biconnected component} of a $G$
is a maximal biconnected subgraph. If such a biconnected component is not a
single vertex or an edge, then it is called \emph{non-trivial}.}

\smallskip

\noindent
\emph{From here on, we will call an undirected graph simply \emph{graph}.}\smallskip

Complete graphs $G=(V,E)$ are denoted by $K_{|V|}$, a path with $n$ vertices by
$P_n$ and a cycle with $n$ vertices by $C_n$. We sometimes write, for
simplicity, $x_1-x_2-\cdots - x_{n}$ for a path $P_n$ with vertices
$x_1,\dots,x_n$ and edges $\{x_i,x_{i+1}\}$, $1\leq i<n$. The join of two
vertex-disjoint graphs $H=(V,E)$ and $H'=(V',E')$ is defined by $H \join
H'\coloneqq (V\cup V',E\cup E' \cup \{\{x,y\}\mid x\in V, y\in V'\})$, whereas
their disjoint union is given by $H \union H'\coloneqq (V \cup V',E \cup E')$.

For later reference, we provide the following simple result. Although we assume that it 
is folklore, we provide a short proof for completeness.

\begin{lemma}\label{lem:heri-forb}
	Let $\Pi$ be a graph property that is closed under isomorphism. 
	Then, the following statements are equivalent.
	\begin{enumerate}
	\item  $\Pi$ is heritable, i.e., $G$ satisfies $\Pi$ if and only if every induced
			subgraph of $G$ satisfied $\Pi$. 
	\item  The class of graphs satisfying $\Pi$ can be characterized in terms of a
			set	$\mathfrak F$ of forbidden subgraphs, i.e., 
			$G$ satisfies $\Pi$ if and only if $G$ does not contain an induced 
			subgraph $H\simeq F\in \mathfrak F$.
	\end{enumerate}
\end{lemma}
\begin{proof}
	Suppose first that $\Pi$ is heritable. Now put $\mathfrak F$ as the set 
	of all graphs that do not satisfy $\Pi$. 
	Conversely, suppose that there is a set $\mathfrak F$ of forbidden subgraphs
	that characterizes graphs that  satisfy property $\Pi$. Since $G$ does not contain
	any graph in $\mathfrak F$ as an induced subgraph, none of the induced subgraphs
	of $G$ can contain a graph in $\mathfrak F$ as an induced subgraph. 
	Consequently, $\Pi$ is heritable. 
\end{proof}

Although for every hereditable property $\Pi$ there is a set $\mathfrak F$ of forbidden subgraphs
that characterizes $\Pi$, this set  $\mathfrak F$ is not necessarily finite.
For example, the property ``$G$ is acyclic'' is heritable, but any
set of forbidden subgraphs characterizing this property is 
of infinite size as it contains all cycles $C_n$, $n\geq 3$.

A graph is \emph{$\mathfrak F$-free}, if it does not contain any of the graphs
in $\mathfrak F$ as an induced subgraph. If $G$ is $\mathfrak F$-free and
$\mathfrak F = \{F\}$ consists of just a single graph, we also say that $G$ is
\emph{$F$-free}.

\paragraph{\bf Trees and Galled-trees}

(Phylogenetic) trees and galled-trees
are particular directed acyclic graphs (DAGs).  
To be more precise, a {\em galled-tree} $N=(V,E)$ on $X$ is a
such that either 
\begin{enumerate}
	\item[(N0)] $V=X = \{x\}$ and, thus, $E=\emptyset$.
\end{enumerate}
or $N$ satisfies  the following three properties 
\begin{enumerate}[noitemsep]
	\item[(N1)] There is a unique \emph{root} $\rho_N$ with indegree 0 and outdegree at least 2; and\smallskip
	\item[(N2)] $x\in X$ if and only if $x$ has outdegree 0 and indegree 1; and \smallskip
	\item[(N3)] Every vertex $v\in V^0 \coloneqq V \setminus X$ with $v\neq \rho_N$ has 
		\begin{enumerate}[noitemsep]
			\item indegree 1 and 	outdegree at least 2 (\emph{tree-vertex}) or
			\item indegree 2 and 	outdegree at least 1 (\emph{hybrid-vertex}).
		\end{enumerate}	
	\NEW{\item[(N4)] Each biconnected component $C$ contains at most one hybrid-vertex $v$ 
					for which the two vertices $v_1,v_2$ with $(v_1,v), (v_2,v)\in E$
					belong to $C$.  }		
\end{enumerate}
\NEW{We note that in \cite{HS:22}, galled-trees have been called level-1 networks.}
If a galled-tree $N=(V,E)$ does not contain hybrid-vertices, then $N$ is a
\emph{tree}. The set $L(N)=X$ is the \emph{leaf set} of $N$ and $V^0$ the set of
\emph{inner} vertices. 
By definition,
every non-trivial biconnected component in a galled-tree $N$ forms an ``undirected''
\emph{cycle} in $N$ \cite{HS18,CRV07} and is also known as \emph{block} \cite{GBP12}
or \emph{gall} \cite{gambette2017challenge}. Hence, a cycle $C$ of a galled-tree
$N$ is composed of two directed paths $P^1$ and $P^2$ in $N$ with the same
start-vertex $\rho_C$ and end-vertex $\eta_C$, and whose vertices distinct from
$\rho_C$ and $\eta_C$ are pairwise distinct. A cycle $C\subseteq N$ is
\emph{weak} if either (i) $P^1$ or $P^2$ consist of $\rho_C$ and $\eta_C$ only
or (ii) both $P^1$ and $P^2$ contain only one vertex distinct from $\rho_C$ and
$\eta_C$.

In \cite{Gusfield:03}, galled-trees were introduced as phylogenetic networks in which 
all cycles are vertex disjoint. Here we consider a more general version, where cycles are 
allowed to share a cutvertex, see also \cite{HSS:22cluster} for more details.

If $N=(V,E)$ is a galled-tree and $(u,v)\in E$, then the vertex $v$
is a \emph{child of $u$}. The set of children of a vertex $u$ in $N$ is denoted
by $\child_N(u)$.

A \emph{caterpillar} is a tree $T$ such that the subgraph induced by the inner vertices is a path with the root
$\rho_T$ at one end of this path and each inner vertex has exactly two
children. A subset $\{x,y\}\subseteq X$ is a \emph{cherry} if
the two leaves $x$ and $y$ are adjacent to the same vertex in $T$. In this case, we
also say that $x$, resp., $y$ is \emph{part of a cherry}. 
Note that a caterpillar
on $X$ with $|X|\geq 2$ contains precisely two vertices that are part of a cherry.

A \emph{caterpillar} is a tree $T$ such that the subgraph induced by the inner vertices is a path with the root
$\rho_T$ at one end of this path and each inner vertex has exactly two
children. A subset $\{x,y\}\subseteq X$ is a \emph{cherry} if
the two leaves $x$ and $y$ are adjacent to the same vertex in $T$. In this case, we
also say that $x$, resp., $y$ is \emph{part of a cherry}. 
Note that a caterpillar
on $X$ with $|X|\geq 2$ contains precisely two vertices that are part of a cherry.

\paragraph{\bf Labeled Galled-trees and Galled-Tree Explainable (\gatex) Graphs}

We consider galled-trees $N=(V,E)$ on $X$ that are equipped with a
\emph{(vertex-)labeling} $t$ i.e., a map $t\colon V\to\{0,1,\odot\}$ such that
$t(x)=\odot$ if and only if $x\in X$. Hence, every inner vertex $v\in V^0$ must have
some \emph{binary} label $t(v)\in \{0,1\}$. The labeling of the leaves $x\in X$ with
$t(x)=\odot$ is just a technical subtlety to cover the special case $V=X=\{x\}$. A
galled-tree $N$ with such a ``binary'' labeling $t$ is called \emph{labeled}
and denoted by $(N,t)$.

Let $N=(V,E)$ be a galled-tree on $X$. A vertex $u\in V$ is called an
\emph{ancestor} of $v\in V$ and $v$ a \emph{descendant} of $u$, in symbols $v
\preceq_N u$, if there is a directed path (possibly reduced to a single vertex)
in $N$ from $u$ to $v$. We write $v \prec_N u$ if $v \preceq_N u$ and $u\neq v$. For
a non-empty subset of leaves $A\subseteq V$ of $N$, we define $\lca_N(A)$, or a
\emph{lowest common ancestor of $A$}, to be a $\preceq_N$-minimal vertex of $N$
that is an ancestor of every vertex in $A$. 
\NEW{Note that (N4) allows that a hybrid vertex $w$ in a galled-tree $N$ might 
satisfy $w = \lca_N(V(C))$ for some non-trivial biconnected component $C$ (in which
case $w$ is the unique hybrid according to (N4) in a biconnected component different
from $C$ in  $N$).}
For simplicity we write
$\lca_N(x,y)$ instead of $\lca_N(\{x,y\})$. Note that in trees and galled-trees
the $\lca_N$ is uniquely determined \cite{HS18,HSS:22cluster}. This allows us to
establish the following 
\begin{definition}
Let $(N,t)$ be a labeled galled-tree on $X$. We denote with $\mathscr{G}(N,t) =
(X,E)$ the graph with vertex set $X$ and edges $\{x,y\}\in E$
precisely if $t(\lca_N(x,y))=1$ and say that $\mathscr{G}(N,t)$ is
\emph{explained} by $(N,t)$. A graph $G = (X,E)$ is
\emph{\textbf{\textsc{Ga}}lled-\textbf{\textsc{T}}ree
\textbf{\textsc{Ex}}plainable (\gatex)}) if there is a labeled galled-tree
$(N,t)$ such that $G\simeq \mathscr{G}(N,t)$.
\end{definition}

If $G$ is explained by a labeled galled-tree $(N,t)$ and one
converts the vertex-labels by replacing every `1' by `0' and \emph{vice versa}, 
one obtains a labeled galled-tree $(N,\overline t)$ that explains 
the complement $\overline G$ of $G$. For later reference, we summarize
the latter in 
\begin{observation}\label{obs:complement-galled-tree}
$G$ is \gatex if and only if 
$\overline G$ is \gatex.
\end{observation}

Note that a galled-tree is a special case of a \emph{tree-child} network
\cite{CRV07}, that is, a network in which every inner vertex $v$ has a child $w$ that is a tree-vertex. In particular, if $N=(V,E)$ is a galled
tree with leaf set $X$ and the set of hybrid-vertices $H$, then it holds that $|H| \leq |X|-1$ and $|V| \leq 4|X|-3$ by \cite[Prop.\
1]{CRV07} and the fact the every hybrid-vertex has indegree two.
Hence, the information about the vertices and edges of \gatex
graphs can be stored efficiently in $O(|X|)$-space.

\paragraph{\bf Modular Decomposition (MD)}

A \emph{module} $M$ of a graph $G=(X,E)$ is a subset 
\NEW{$M\subseteq V(G)  = X$} such that
for all $x,y\in M$ it holds that $N_G(x)\setminus M = N_G(y)\setminus M$, where
$N_G(x)$ is the set of all vertices of $X$ that are adjacent to $x$ in $G$. The singletons and
$X$ are the \emph{trivial} modules of $G$ and thus, the set $\MD(G)$ of all
modules of $G$ is not empty. A graph $G$ is called \emph{primitive} if it has at
least four vertices and contains trivial modules only. The smallest primitive
graph is the path $P_4$ on four vertices. 

A module $M$ of $G$ is \emph{strong} if $M$ does not \emph{overlap} with any
other module of $G$, that is, $M\cap M' \in \{M, M', \emptyset\}$ for all
modules $M'$ of $G$. In particular, all trivial modules of $G$ are strong. The
set of strong modules $\MDstrong(G)\subseteq \MD(G)$ is uniquely determined
\cite{HSW:16,EHMS:94} and forms a hierarchy which gives rise to a unique tree
representation $\MDT_G$ of $G$, known as the \emph{modular decomposition tree}
(\emph{MDT}) of $G$. The vertices of $\MDT_G$ are (identified with) the elements
of $\MDstrong(G)$. Adjacency in $\MDT_G$ is defined by the maximal proper
inclusion relation, that is, there is an edge $(M',M)$ between $M,M'\in
\MDstrong(G)$ if and only if $M\subsetneq M'$ and there is no $M''\in
\MDstrong(G)$ such that $M\subsetneq M'' \subsetneq M'$. The root of $\MDT_G$ is
(identified with) $X$ and every leaf is uniquely identified with a singleton
$\{v\}$, $v\in X$. In summary, $\MDstrong(G) = \{L(\MDT_G(v)) \mid v\in
V(\MDT_G)\}$. Uniqueness and the hierarchical structure of $\MDstrong(G)$
implies that there is a unique partition $\Mmax(G) = \{M_1,\dots, M_k\}$ of $X$
into inclusion-maximal strong modules $M_j\ne X$ of $G$ \cite{ER1:90,ER2:90}.
Since $X\notin \Mmax(G)$ the set $\Mmax(G)$ consists of $k\ge 2$ strong modules,
whenever $|X|>1$.

Similar as for galled-trees, one can equip $\MDT_G$ with a vertex-labeling $t_G$
such that, for $M\in \MDstrong(G) = V(\MDT_G)$, we have $t_G(M)=\odot$ if
$|M|=1$; $t(M)=0$ if $|M|>1$ and $G[M]$ is disconnected; $t_G(M)=\mh{1}$ if $|M|>1$
and $G[M]$ is connected but $\overline G[M]$ is disconnected; $t_G(M) =
\mathit{prime}$ in all other cases. Strong modules of $G$ are called
\emph{series}, \emph{parallel} and \emph{prime} if $t_G(M)=1$, $t_G(M)=0$ and
$t_G(M) = \mathit{prime}$, respectively. Efficient linear algorithms to compute
$(\MDT_G,t)$ have been proposed e.g.\ in \cite{DGC:01,CS:99,TCHP:08}. In this
way, we obtain a labeled tree $(\MDT_G,\tau_G)$ that, at least to some extent,
``encodes'' the structure of $G$. In particular, $\tau_G(\lca(x,y))=1$ implies
$\{x,y\}\in E(G)$ and $\tau_G(\lca(x,y))=0$ implies $\{x,y\}\notin E(G)$. The
converse, however, is not satisfied in general, since
$\tau_G(\lca(x,y))=\mathit{prime}$ is possible.

Two disjoint modules $M, M'\in \MD(G)$ are \emph{adjacent} (in $G$) if each
vertex of $M$ is adjacent to all vertices of $M'$; the modules are
\emph{non-adjacent} if none of the vertices of $M$ is adjacent to a vertex of
$M'$. By definition of modules, every two disjoint modules $M, M'\in \MD(G)$ are
either adjacent or non-adjacent \cite[Lemma 4.11]{ER1:90}. One can therefore
define the quotient graph $G/\Mmax(G)$ which has $\Mmax(G)$ as its vertex set
and $\{M_i,M_j\}\in E(G/\Mmax(G))$ if and only if $M_i$ and $M_j$ are adjacent
in $G$. 
\begin{observation}[\cite{HP:10}]\label{obs:quotient} 
The quotient graph $G/\Mmax(G)$ with $\Mmax(G) = \{M_1 , \dots , M_k\}$ is
isomorphic to any subgraph induced by a set $W\subseteq V$ such that $|M_i \cap
W | = 1$ for all $i \in \{1, \dots,k\}$.
\end{observation}

For later reference, we provide the following simple result. 
\begin{lemma}\label{lem:Hprimitive-in-primeM}
 Let $H$ be an induced primitive subgraph of $G$ and $M\in \MDstrong(G)$ be a
 strong module that is inclusion-minimal w.r.t.\ $V(H)\subseteq M$. Then, $M$ is
 prime and $H$ is isomorphic to an induced subgraph of $G[M]/\Mmax(G[M])$.
\end{lemma}
\begin{proof}
 Let $H$ be an induced primitive subgraph of $G$ and $M\in \MDstrong(G)$ be a
 strong module that is inclusion-minimal w.r.t.\ $V(H)\subseteq M$. Assume, for
 contradiction, that $M$ is not prime. W.l.o.g.\ assume that $M$ is parallel
 (otherwise consider the complement $\overline G$ of $G$). Since $M$ is
 inclusion-minimal w.r.t.\ $V(H)\subseteq M$, it follows that there are two
 vertices $x,y\in V(H)$ such that $x\in M'$ and $y\in M''$ where $M',M'' \in
 \Mmax(G[M])$ are distinct. In this case, however, $x$ and $y$ must be in
 distinct connected components of $H$. Hence, $H$ is disconnected and,
 therefore, not primitive; a contradiction. 
 
 We continue with showing that, for any two distinct vertices $x,y\in V(H)$ we
 have $x\in M'$ and $y\in M''$ with $M',M'' \in \Mmax(G[M])$ being distinct.
 Assume, for contradiction, that $x,y\in V(H)$ are distinct but $x,y\in M'\in
 \Mmax(G[M])$. Since $M$ is inclusion-minimal w.r.t.\ $V(H)\subseteq M$, it
 follows that there are is a vertex $z\in V(H)$ such that $z\in M'' \in
 \Mmax(G[M]) \setminus \{M'\}$. Thus, we can partition $V(H)$ into two non-empty
 sets $V_1$ and $V_2$ such that $V_1$ contains all vertices of $M'\cap V(H)$ and
 we put $V_2 = V(H)\setminus V_1$. By definition of modules, a vertex in $V_2$
 must be either adjacent to all or to none of the vertices in $M'$ and, thus, in
 particular, the vertices in $V_1$. This and $|V(H)|>|V_1|\geq 2$ implies that
 $V_1$ is a non-trivial module of $H$; a contradiction. In summary, for any
 two distinct vertices $x,y\in V(H)$ we have $x\in M'$ and $y\in M''$ with
 $M',M'' \in \Mmax(G[M])$ being distinct. Hence, there is a subset $\mathcal
 M\subseteq \Mmax(G[M])$ of size $|V(H)|$ such that there is a 1-to-1
 correspondence between the vertices in $|V(H)|$ and the modules in $\mathcal M$
 which together with Obs.\ \ref{obs:quotient} implies that $H$ is isomorphic to
 an induced subgraph of $G/\Mmax(G)$. 
\end{proof}

In \cite{SCHMERL1993191} the structure of primitive graphs (called indecomposable in \cite{SCHMERL1993191})
has been studied. A graph $G$ is \emph{critical-primitive} if $G-x$ is not primitive for all $x\in V(G)$. 

\begin{proposition}[{\cite[Cor.\ 5.8]{SCHMERL1993191}}]\label{prop:critical-primitive}
A graph $G$ is critical-primitive if and only if $G$ is isomorphic to the graph
$\mathcal{G}_r$ or its complement $\overline{\mathcal{G}}_r$ where, for $r\geq
2$, $\mathcal{G}_r$ has vertex set $\{a_1,\dots,a_r,b_1,\dots,b_r\}$ and
$\{a_i,b_j\}$ is an edge in $\mathcal{G}_r$ precisely if $i\geq j$. 
\end{proposition}

\section{Cographs,  Pseudo-cographs and \gatex Graphs}

\emph{Cographs} are defined recursively \cite{Corneil:81}: the single vertex
graph $K_1$ is a cograph and the join $H \join H'$ and disjoint union $H \union
H'$ of two cographs $H$ and $H'$ is a cograph. In particular, cographs are
precisely those graphs that are explained by their MDT $(\MDT_G,t_G)$. Many
characterizations of cographs have been established in the last decades
\cite{Corneil:81,Sumner74,Seinsche:74,HS:22} from which we summarize a few of
them in the following

\begin{theorem}\label{thm:prop_cograph}
	The following statements are equivalent.
 \begin{enumerate}[noitemsep]
 	\item $G$ is a cograph.
 	\item The complement $\overline G$ of $G$ is a cograph.
 	\item Every induced subgraph of $G$ is a cograph.
 	\item $G$ does not contain a path on four vertices as an induced subgraph.
 	\item The  MDT $(\MDT_G,t_G)$ of $G$ does not contain vertices $v$ such that $t_G(v)=\mathit{prime}$.
 	\item $G$ can be explained by a labeled galled-tree $(N,t)$
 		  for which all cycles are weak. \cite[Thm.\ 5.5]{HS:22}

 		 Thus, cographs form a proper subclass of \gatex graphs. 
 \end{enumerate} 
\end{theorem}

A cograph is \emph{caterpillar-explained} if its MDT is a caterpillar. We note
in passing that caterpillar-explained cographs form a proper subclass of
threshold graphs \cite{CHVATAL1977145}, i.e., graphs that can be obtained from a
single vertex graph by repeatedly adding either an isolated vertex or a vertex
that is adjacent to all pre-existing vertices.

A proper super-class of cographs are pseudo-cographs. 

\begin{definition}[Pseudo-Cographs]\label{def:pseudo-cograph}
	A graph $G$ is a \emph{pseudo-cograph} if $|V(G)|\leq 2$ or
	if there are induced subgraphs $G_1,G_2\subseteq G$ and a vertex $v\in V(G)$
 such that 
 \begin{enumerate}[noitemsep]
 	\item[(P1)] $V(G)=V(G_1)\cup V(G_2)$, $V(G_1)\cap V(G_2) = \{v\}$,  $|V(G_1)|>1$ and 
 	$|V(G_2)|>1$; and \smallskip
 	\item[(P2)] $G_1$ and $G_2$ are cographs; and \smallskip
 	\item[(P3)] $G-v$ is either 
 	            the join or the disjoint union of $G_1-v$ and $G_2-v$. 
 \end{enumerate} 
 In this case, we also say that $G$ is a  \emph{$(v,G_1,G_2)$-pseudo-cograph}. 
  A $(v,G_1,G_2)$-pseudo-cograph $G$ is \emph{slim} (resp., \emph{fat})  
  if $G-v$ is the disjoint union (resp., join) of $G_1-v$ and $G_2-v$. 
\end{definition}

Note that the choice of $G_1$, $G_2$ and $v$ may not be unique. In particular,
$G$ can be a slim $(v,G_1,G_2)$-pseudo-cograph and a fat
$(v',G'_1,G'_2)$-pseudo-cograph for different choices of the triplets
$(v,G_1,G_2)$ and $(v',G'_1,G'_2)$. By way of example, consider a path
$G=1-2-3-4$ on four vertices. Here, $G$ is a fat $(1,G_1,G_2)$-pseudo-cograph
for $G_1=G[\{1,3\}]$ and $G_2=G[\{1,2,4\}]$ and a slim
$(2,G'_1,G'_2)$-pseudo-cograph for $G'_1 =G[\{1,2\}]$ and $G'_2 =G[\{2,3,4\}]$.

By definition of pseudo-cographs and, in particular (P3), we easily derive
\begin{observation}\label{obs:G-v-Cograph}
	If $G$ is a \emph{$(v,G_1,G_2)$-pseudo-cograph}, then $G-v$ is a cograph
	and either $G-v$ or its complement $\overline{G-v}$ is disconnected.
\end{observation}

Several characterizations of and algorithmic results for pseudo-cographs have
been established in \cite{HS:22} and are summarized in the following
\begin{theorem}\label{thm:prop_pc}
	The following statements are equivalent.
 \begin{enumerate}[noitemsep]
 	\item $G$ is a pseudo-cograph.
 	\item The complement $\overline G$ of $G$ is a pseudo-cograph.
 	\item Every induced subgraph of $G$ is a pseudo-cograph.
 	\item $|V(G)|\leq 2$ or $G$ can be explained by a galled-tree $(N,t)$ that contains precisely
 		 one cycle $C$, and $C$ is such that $\rho_C = \rho_N$ and $\child_N(\eta_C)=\{x\}$ for some $x\in  L(N)$.
 		 
 		 Thus, pseudo-cographs form a proper subclass of \gatex graphs. 
 \end{enumerate} 
 In particular, $G$ is a $(v,G_1,G_2)$-pseudo-cograph if and only if 
 $\overline G$ is a $(v,\overline G_1,\overline G_2)$-pseudo-cograph. 
 Moreover, Pseudo-cographs $G = (V , E)$ can be recognized in $O( | V | + | E |
 )$ time. In the affirmative case, one can determine a vertex $v$ and subgraphs
 $G_1 , G_2 \subseteq G$ such that $G$ is a $( v, G_1 , G_2 )$-pseudo-cograph
 and a labeled galled-tree that explains $G$ in $O( | V | + | E | )$ time.
\end{theorem}

Note that, by \cite[Lemma 4.2]{HS:22}, every cograph is a pseudo-cograph.
A further subclass of pseudo-cographs are polar-cats. 

\begin{definition}[Polar-Cat]
A  $(v,G_1,G_2)$-pseudo-cograph $G$ is a polar-cat if $G$ satisfies
the following conditions:
\begin{enumerate}
	\item[(C1)] $|V(G)|\geq 4$
	\item[(C2)] \emph{polar}izing: $G_1$ and $G_2$ are both connected (resp., both disconnected) and 
			 $G-v$ is the disjoint union (resp., join) of $G_1-v$ and $G_2-v$.
	\item[(C3)] \emph{cat}: $G_1$ and $G_2$ are \emph{cat}erpillar-explained such that $v$ is part of a cherry in
		        both caterpillars that explain $G_1$, resp, $G_2$. 
\end{enumerate}
In this case, we also say that $G$ is a \emph{$(v,G_1,G_2)$-polar-cat}. A
$(v,G_1,G_2)$-polar-cat is \emph{well-proportioned} if $|V(G_1)| \geq 3$ and
$|V(G_2)| \geq 3$, or $|V(G_i)| = 2$ and $|V(G_j)|\geq 5$ with $i,j\in\{1,2\}$
distinct.
\end{definition}

If $G$ is a $(v,G_1,G_2)$-polar-cat and thus, a pseudo-cograph, 
Def.\ \ref{def:pseudo-cograph}(P1) implies that 
$|V(G_1)|\geq 2$ and $|V(G_2)|\geq 2$. This, in particular, implies 
\begin{observation}\label{obs:well-prop}
Every polar-cat $G$ with $|V(G)|\geq 7$ is well-proportioned. 
\end{observation}

Well-proportioned polar-cats $G$ are of particular interest, since the galled-tree
   $(N,t)$ that explains $G$ is uniquely determined (cf.\ \cite[Prop.~6.13]{HS:22}).
We give now a mild extension of a characterization of polar-cats established in \cite{HS:22}.
\begin{theorem}\label{thm:CharPolCat}
	The following statements are equivalent.
	 \begin{enumerate}[noitemsep]
	 	\item $G$ is a polar-cat.
	 	\item $G$ is a primitive graph that can be explained by a labeled galled-tree. 
	 	\item $G$ is a primitive pseudo-cograph. 
	 \end{enumerate} 
	 Moreover, if $G$ is a well-proportioned $(v,G_1,G_2)$-polar-cat, 
	 then $v$ as well as $G_1$ and $G_2$ are uniquely determined, i.e., 
	 there is no other vertex $w\neq v$ in $G$ and no other subgraphs $G'_1$ and $G'_2$
	 than $G_1$ and $G_2$ such that $G$ is a $(w,G'_1,G'_2)$-polar-cat.

Furthermore, polar-cats $G = (V , E)$ can be recognized in $O( | V | + | E | )$
time. In the affirmative case, one can determine a vertex $v$ and subgraphs $G_1
, G_2 \subseteq G$ such that $G$ is a $( v, G_1 , G_2 )$-polar-cat and a labeled
galled-tree that explains $G$ in $O( | V | + | E | )$ time.
\end{theorem}
\begin{proof}
The equivalence between Condition (1) and (2) have been shown in \cite[Thm.\
6.9]{HS:22}. Suppose that (3) $G$ is a primitive pseudo-cograph. By Thm.\
\ref{thm:prop_pc}, $G$ can be explained by a labeled galled-tree and Condition
(2) follows. Finally, Condition (1) and its equivalent expression Condition (2)
together with the definition of polar-cats imply that $G$ is a primitive
pseudo-cograph. The second statement is equivalent to \cite[Thm.\ 6.12]{HS:22},
and the last part is stated in \cite[Thm.\ 9.3]{HS:22}. 
\end{proof}

A convenient characterization of the edges in primitive pseudo-cographs
is provided by
\begin{proposition}[{\cite[Prop.~6.2]{HS:22}}]\label{prop:polcat-edges}
A graph $G =(V,E)$ is a primitive pseudo-cograph if and only if
$|V|\geq 4$ and there exists a vertex $v \in V$ and two ordered sets $Y=\{y_1, \ldots, y_{\ell-1},
y_{\ell}=v\}$ and $Z=\{z_1, \ldots, z_{m-1}, z_{m}=v\}$, $\ell,m \geq 2$ such that $Y \cap
Z=\{v\}$, $Y \cup Z=V$, and one of the following conditions hold:
\begin{itemize}
	\item[(a)] $G[Y]$ and $G[Z]$ are connected and the edges of $G$ are 
		\begin{itemize}[noitemsep, nolistsep]
			\item[(I)] $\{y_i,y_j\}$,		where $1 \leq i<j \leq \ell$ and $i$ is odd; and
			\item[(II)] $\{z_i,z_j\}$, where $1 \leq i<j \leq		m$ and $i$ is odd.
		\end{itemize}
		
		In this case, $G$ is a slim $(v,G[Y],G[Z])$-polar-cat.
	\item[(b)] $G[Y]$ and $G[Z]$ are disconnected and the edges of $G$ are 
		\begin{itemize}[noitemsep, nolistsep]
			\item[(I)] $\{y_i,y_j\}$, where $1 \leq i<j \leq \ell$ and $i$ is even; and
			\item[(II)] $\{z_i,z_j\}$, where $1 \leq i<j \leq m$ and $i$ is even; and 
			\item[(III)] $\{y,z\}$ for all $y\in Y\setminus\{v\}$ and $z\in Z\setminus\{v\}$ 
	  \end{itemize}
	  
	  		In this case, $G$ is a fat $(v,G[Y],G[Z])$-polar-cat.
	\end{itemize}
\end{proposition}
 
As for general pseudo-cographs, the choice of $G_1$, $G_2$ and $v$ may not be
unique even for polar-cats (as example, consider again an induced $P_4$).
However, by Thm.\ \ref{thm:CharPolCat}, for all well-proportioned
$(v,G_1,G_2)$-polar-cats, $G_1$, $G_2$ and $v$ are uniquely determined. This and
the fact that $G_1$ and $G_2$ are induced subgraphs implies that the sets $Y$,
$Z$ and vertex $v$ as specified in Prop.\ \ref{prop:polcat-edges} are uniquely
determined for well-proportioned polar-cats. Thm.\ \ref{thm:CharPolCat},
can be used to obtain a slightly rephrased and simplified version of Thm.\ 7.5
in \cite{HS:22}. 

\begin{theorem}\label{thm:GalledTreeExplainable}
A graph $G$ is \gatex  if and only if 
$G[M]/\Mmax(G[M])$ is a (primitive) pseudo-cograph for all 
prime modules $M$ of $G$.

Moreover, it can be verified in $O( | V | + | E | )$ time if a given graph $G =
(V , E)$ is \gatex and, in the affirmative case, a labeled galled-tree $(N , t)$
that explains $G$ can be constructed within the same time complexity.
\end{theorem}

The latter statement of Theorem \ref{thm:GalledTreeExplainable} is equivalent to
\cite[Thm.\ 9.4]{HS:22}. In \cite{HS:22}, the class of polar-cats has been
denoted by \PolarCat\ and the class of \gatex graphs by \PrimeCat. The latter is
due to Theorem \ref{thm:GalledTreeExplainable} and Theorem \ref{thm:CharPolCat}
which imply that $G[M]/\Mmax(G[M])$ is a polar-cat \PolarCat\ for every prime
module $M$ of a \gatex graph $G$. Moreover, it has been shown in \cite{HS:22}
that every labeled galled-tree $(N,t)$ that explains $G$ can be obtained from
the MDT of $G$ by replacing each vertex labeled as $\mathit{prime}$ by a
suitable choice of a single 0/1-labeled cycle. The latter is justified by
\cite[Prop.\ 3.8]{HS:22} which shows that for any labeled galled-tree $(N,t)$
that explains $G$ and that does not contain weak cycles, there is a 1-to-1
correspondence between the cycles in $N$ and the prime modules of $G$. Suitable
choices to replace respective prime modules $M$ of $G$ can be obtained as
follows: Construct first $H\coloneqq G[M]/\Mmax(G[M])$. Check if $H$ is a
polar-cat and determine a vertex $v\in V(H_1)$ as well as induced subgraphs
$H_1, H_2\sqsubseteq H$ such that $H$ is a $(v,H_1,H_2)$-polar-cat. By Theorem
\ref{thm:CharPolCat}, the latter task can be done in linear-time. By definition,
$H-v$ is either the join or the disjoint union of the cographs $H_1-v$ and
$H_2-v$. Note that the two cotrees $(T_1,t_1)$ and $(T_2,t_2)$ that explain
$H_1-v$ and $H_2-v$, respectively, must be caterpillars and $v$ is part of a cherry in $(T_1,t_1)$ as
well as in $(T_2,t_2)$. To construct $(N,t)$ that explains $H$, (1) add a root
$\rho_N$ and $(T_1,t_1)$ and $(T_2,t_2)$ to $N$, (2) add edges $\{\rho_N,
\rho_{T_1}\}$ and $\{\rho_N, \rho_{T_2}\}$ and (3) identify $v$ in $(T_1,t_1)$
and $(T_2,t_2)$ to obtain a vertex $v'$ and (4) add an edge $\{v',v\}$. Finally
keep the labels of all vertices in $T_1$ and $T_2$, label $v'$ arbitrarily and
$\rho_N$ by either ``1'' or ``0'' depending on whether $H$ is fat or slim. Now
replace the prime module $M$ in $G$ by the unique cycle in $N$; see \cite{HS:22}
for further details.

\section{Forbidden Subgraph Characterization}

In this section, we show that $\gatex$ graphs are characterized by a set of 
25 forbidden subgraphs. To this end, we  first provide the following

\begin{lemma}[{\cite[Lemma 7.1]{HS:22}}]\label{lem:gt-heri}
 A graph $G$ is \gatex  if and only if every induced subgraph of $G$ is \gatex.
\end{lemma}

By Lemma \ref{lem:gt-heri}, the property of a graph being galled-tree
explainable is heritable and Lemma \ref{lem:heri-forb} implies that the class of
\gatex graphs can be characterized by means of a not necessarily finite set of
forbidden induced subgraphs. In the following, let $\forbGT$ denote the
(currently unknown) set of minimal forbidden induced subgraphs that
characterizes \gatex graphs. With \emph{minimal} we mean that if $F\in \forbGT$,
then $F$ is not \gatex, while $F-x$ is \gatex for all $x\in V(F)$. In what
follows, we will define the set $\forbGT$ in full detail. To this end, we
provide first a new characterization of \gatex graphs. 
\begin{theorem}	\label{thm:GatexIFFallPrimitive=PsC}
 $G$ is \gatex if and only if all primitive induced subgraphs of $G$ are pseudo-cographs.
\end{theorem}
\begin{proof}
	By contraposition, assume that $G$ contains a primitive
	induced subgraph $H=(W,E)$ that is not a pseudo-cograph
	and let $M$ be an inclusion-minimal module that contains $V(H)$. 
	By Lemma \ref{lem:Hprimitive-in-primeM}, $M$ is prime
	and $H$ is isomorphic to an induced subgraph of $G[M]/\Mmax(G[M])$.
	By Thm.\ \ref{thm:prop_pc}, the property of being a pseudo-cograph is heritable.
	The latter arguments imply that $G[M]/\Mmax(G[M])$
	cannot be a pseudo-cograph. 
	By Thm.\ \ref{thm:GalledTreeExplainable}, 
	$G$ is not \gatex. 
	This establishes the \emph{only-if} direction. 
	
	Assume now that all primitive induced subgraphs of $G$ (if there are any) 
	are pseudo-cographs. Note that if $G$ does not contain prime modules at all, 
	then Thm.\ \ref{thm:GalledTreeExplainable} implies that $G$
	can be explained by a labeled galled-tree. Hence, assume that $G$
	contains a prime module $M$. By Obs.\ \ref{obs:quotient}, 
	$G[M]/\Mmax(G[M])\simeq H$ where $H$ is an induced subgraph of $G$. 
	In particular, $H$ is primitive. By assumption, $H$ is a pseudo-cograph.
	Since the latter argument holds for all prime modules of $G$, 
	Thm.\ \ref{thm:GalledTreeExplainable} implies that $G$ is \gatex. 
	This establishes the \emph{if} direction. 
\end{proof}	

The following lemma shows that the elements in $\forbGT$ are precisely the 
   minimal primitive non-pseudo-cographs.

\begin{lemma}\label{lem:forb->primitiveNOTPC}
	It holds that $F\in \forbGT$ if and only if $F$ is primitive and not a pseudo-cograph
	and there is no $F'\in \forbGT$ such that $|V(F')|<|V(F)|$ and $F'\sqsubseteq F$.
\end{lemma}
\begin{proof}
    Let $F\in \forbGT$. By definition, $F$ is not \gatex. However, since every
     pseudo-cograph is \gatex, it follows that $F$ cannot be a pseudo-cograph.
     Assume now, for contradiction, that $F$ is not primitive. Since $F$ is not
     \gatex, Thm.\ \ref{thm:GalledTreeExplainable} implies that there is a
     prime module $M$ of $F$ such that $ F[M]/\Mmax(F[M])$ is not a
     pseudo-cograph. By
     Obs.\ \ref{obs:quotient}, $ F[M]/\Mmax(F[M])$ is isomorphic to an induced
     subgraph $H$ of $F$. In particular, $V(H)$ is a prime module of $H$ and,
     since $H$ is primitive but not a pseudo-cograph, Thm.\
     \ref{thm:GalledTreeExplainable} implies that $H$ is not \gatex. This
     and the fact that $\forbGT$ characterizes \gatex graphs implies that there is an induced subgraph $F'\in \forbGT$ such that
     $F'\subseteq H$. Note since $F$ is not primitive but $H\subseteq F$ is, it
     follows that $H\subsetneq F$. Therefore, $F'\subseteq H\subsetneq F$
     implies that $F$ is not minimal and, hence, $F\notin \forbGT$; a
     contradiction. In summary, all $F\in \forbGT$ are primitive and not
     pseudo-cographs. Assume now, for contradiction, that there is an $F'\in
     \forbGT$ such that $|V(F')|<|V(F)|$ and $F'\sqsubseteq F$. By the latter
     arguments, $F'$ is primitive and not a pseudo-cograph. In particular,
     $F'\sqsubseteq F-x$ for some $x\in V(F)$. By Thm.\
     \ref{thm:GatexIFFallPrimitive=PsC}, $F-x$ is not \gatex; a contradiction to
     minimality of $F$. 
	
	For the converse, assume that $F$ is primitive and not a pseudo-cograph and
	that there is no $F'\in \forbGT$ such that $|V(F')|<|V(F)|$ and
	$F'\sqsubseteq F$. Hence, for all $x\in V(F)$, it holds that $F-x$ does not
	contain any of the graphs in $\forbGT$ as an induced subgraph. Since
	$\forbGT$ is the set of (minimal) forbidden induced subgraphs that
	characterizes \gatex graphs, it follows that $F-x$ is \gatex for all $x\in
	V(F)$. By Thm.\ \ref{thm:GatexIFFallPrimitive=PsC}, $F$ is not \gatex. By
	the latter arguments, $F$ is minimal and thus, $F\in \forbGT$.
\end{proof}

\begin{algorithm}[tb] 
\small 
  \caption{\texttt{Construct $\forbGT$}}
\label{alg:construct-F}
\begin{algorithmic}[1]
  \Require  The sets $\mathfrak{G}_i$ of all graphs on $i\in\{5,\dots,8\}$ vertices 
  \Ensure 	Set $\forbGT$ of forbidden subgraphs on $5$ to $8$ vertices 
  \State $\forbGT\gets \emptyset$ 
  \While{$i\leq 8$}
	  \ForAll{$G\in \mathfrak{G}_i$} 
		\State forbidden=\texttt{true}
	  	\If{$G$ is primitive and not a pseudo-cograph} \Comment{Pseudo-cograph recognition with Alg.\ 2 in \cite{HS:22}} \label{l:primNotPC}
				\ForAll{$H\in \forbGT$ with $|V(H)|<|V(G)|$} \label{for1}
					\If{$H\sqsubseteq G$} \label{if1}
						\State forbidden=\texttt{false} 
					\EndIf
				\EndFor
				\If{forbidden} \label{if2}
					\State 	$\forbGT\gets \forbGT\cup\{G\}$. \label{add}
	 	 		\EndIf	 \label{if22}
	 	\EndIf	
	  \EndFor
	\State $i\gets i+1$
  \EndWhile
  \State \Return  $\forbGT$		 
\end{algorithmic}
\end{algorithm}

By Lemma \ref{lem:forb->primitiveNOTPC}, any forbidden subgraph $F\in \forbGT$
of \gatex graphs must be primitive. Hence, $|V(F)|\geq 4$. Since, however, an
induced $P_4$ is a pseudo-cograph and no other graph on four vertices is
primitive, we can conclude that $|V(F)|\geq 5$. Furthermore, a
$(v,G_1,G_2)$-pseudo-cograph $G$ has, in particular, the property that $v$ must
be located on every induced $P_4$. By Lemma \ref{lem:forb->primitiveNOTPC},
$F\in \forbGT$ cannot be a pseudo-cograph. By the latter two arguments, a
putative worst-case for the number of vertices of a forbidden subgraph in
$\forbGT$ are graphs that contain two vertex-disjoint $P_4$s. The latter leads
to the conjecture that $|V(F)|\leq 8$. We summarize the latter discussion into
the following
\begin{conjecture}
For all $F\in \forbGT$, we have $5\leq |V(F)|\leq 8$.  
\end{conjecture}

\begin{figure}[t]
	\begin{center}
			\includegraphics[width=0.9\textwidth]{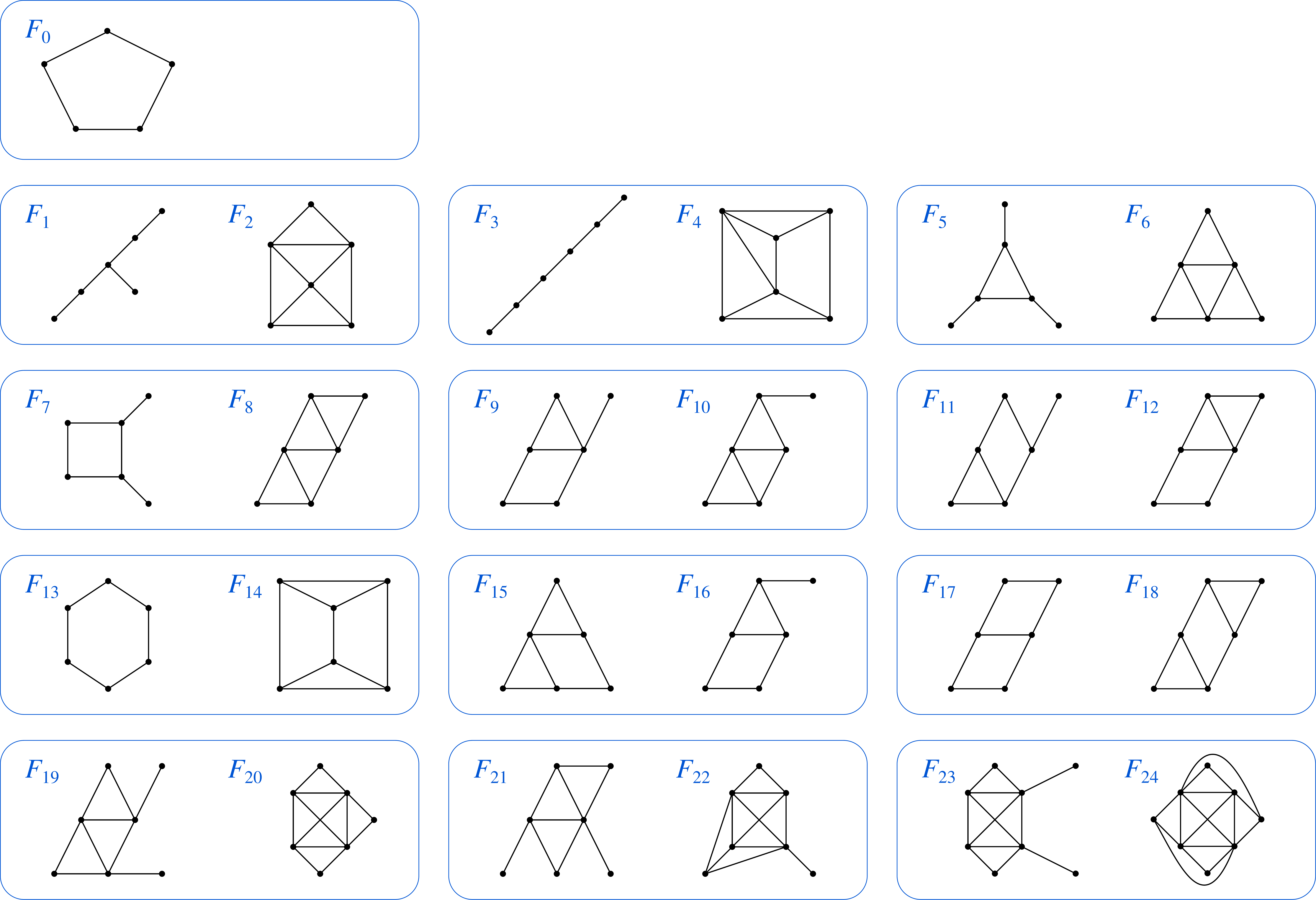}
	\end{center}
	\caption{All graphs on 5 to 8 vertices that are primitive and not
	         pseudo-cographs and for which every induced subgraph with fewer
	         vertices is \gatex. Note, $F_0$ is self-complementary. Every two
	         graphs $F_i$ and $F_{i+1}$, $i \in \{1,3,5, \ldots, 23\}$
	         (represented in a common framebox) are complements of each other.
	         According to Theorem \ref{thm:F-free}, a graph $G$ is \gatex if and
	         only if $G$ does not contain any of the graphs in $\forbGT =
	         \{F_0,F_1,\dots,F_{24}\}$ as an induced subgraph. 
			 }
\label{fig:forb}
\end{figure}

To get evidence for this conjecture, we established Alg.~\ref{alg:construct-F}
to compute all minimal forbidden subgraphs of size $5$ to $8$ which we
downloaded from \cite{graph-list}. In total, we checked $13\ 580$ graphs.
Alg.~\ref{alg:construct-F} and the choice of graphs is based on the previous
results. To be more precise, by the arguments above, it suffices to look for
graphs on at least $5$ vertices. By Lemma \ref{lem:forb->primitiveNOTPC}, a
necessary condition for $F$ being an element in $\forbGT$ is that $F$ is
primitive and not a pseudo-cograph (Alg.\ \ref{alg:construct-F}, Line
\eqref{l:primNotPC}). To check if $F$ is a minimal, we can apply Lemma
\ref{lem:forb->primitiveNOTPC} and verify if a formerly computed forbidden
subgraph $H$ with $|V(H)|<|V(F)|$ (Alg.\ \ref{alg:construct-F}, Line
\eqref{for1}) is an induced subgraph of $F$ (Alg.\ \ref{alg:construct-F}, Line
\eqref{if1}). If this is not the case, $F$ adds a new forbidden subgraph
$\forbGT$ (Alg.\ \ref{alg:construct-F}, Line \eqref{add}). This algorithm is
written in Python and requires Python version 3.7 or higher and relies on the
packages NetworkX and tralda \cite{tralda}. It is hosted at GitHub
(\url{https://github.com/marc-hellmuth/ForbiddenSubgraphs-GaTEx}
\cite{github-MH}). The set $\forbGT$ of forbidden subgraphs computed with Alg.\
\ref{alg:construct-F} is shown in Fig.\ \ref{fig:forb} and we put, \emph{from
here on}, 
\[\forbGT \coloneqq \{F_0,\dots, F_{24}\}.\] 
So-far, we have constructed all elements of $\forbGT$ on $5$ to $8$ vertices. As
we shall see later, these are indeed all elements of $\forbGT$. We are aware of
the fact, that the correctness of the computed graphs in $\forbGT$ relies on the
correct implementation of Alg.~\ref{alg:construct-F}. In order to make it as
easy as possible for the reader to verify this implementation,
Alg.~\ref{alg:construct-F} is implemented to be readable as easy as possible
and, thus, mainly brute-force routines are used, although such a simplified
algorithm goes hand in hand with a possibly lower runtime.

Since we have checked all graphs on $5$ to $8$ vertices, we obtain
\begin{observation}\label{obs:exhaustive}
 The set  $\forbGT$ of forbidden subgraphs is exhaustive for graphs up to $8$ vertices.
 In other words,  every  primitive graph $G$ that is not a pseudo-cograph and where $|V(G)| \leq 8$
 must contain at least one $F\in \forbGT$ as an induced subgraph.
\end{observation}
	
In addition, we applied Alg.\ \ref{alg:construct-F} to all
$9$ vertex graphs. As it turned out, all such graphs $G$ that are not 
\gatex already contain one of the existing 
$F\in \forbGT$. Consequently, there are no forbidden subgraphs
with 9 vertices which gives further evidence to the aforestated
conjecture. 
Closer inspection of the graphs in $\forbGT$ shows 
\begin{observation}
$F\in \forbGT$ if and only $\overline F\in \forbGT$ and, therefore, 
$G$ is $\forbGT$-free if and only if $\overline G$ is $\forbGT$-free
\end{observation}

Now consider the graph $\mathcal{G}_r$ as specified in Prop.\
\ref{prop:critical-primitive}. If $r\geq 3$ and thus, $|V(\mathcal{G}_r)|\geq
6$, then $\mathcal{G}_r[\{a_1,a_2,a_3,b_1,b_2,b_3] \simeq F_7\in \forbGT$. This
implies together with Prop.\ \ref{prop:critical-primitive} 

\begin{corollary}\label{cor:critical-primitive}
If $G$ is critical-primitive and $|V(G)|\geq 6$, then $G$ contains the forbidden
subgraph $F_7\in \forbGT$ or its complement $F_8 \in \forbGT$.
\end{corollary}

The next lemma is crucial for the forbidden subgraph characterization of graphs
that can be explained by galled-trees.
\begin{lemma}\label{lem:PrimWellPC->F}
	Every primitive graph that is not a pseudo-cograph is not $\forbGT$-free.
\end{lemma}
\begin{proof}
In essence, the proof is based on induction on the number of vertices. 
	The base case is provided by the previously computed set $\forbGT$ that
	already contains all forbidden subgraphs on $5$ to $8$ vertices.
	The straightforward but tedious case-by-case analysis is given in the appendix.
\end{proof}
We are now in the position to provide a further novel characterization of \gatex graphs.
\begin{theorem}\label{thm:F-free}
	$G$ is \gatex if and only if $G$ is $\forbGT$-free.
\end{theorem}
\begin{proof}
	The \emph{only-if} direction readily follows from the fact that every
	induced subgraph of a \gatex graph $G$ is \gatex as well (cf.\ Lemma
	\ref{lem:gt-heri}). Since none of the graphs in $\forbGT$ are \gatex, it
	follows that $G$ is $\forbGT$-free. 

	By contraposition, assume that $G$ is not \gatex. By Thm.\
	\ref{thm:GalledTreeExplainable}, $G$ contains a prime module $M$ such that
	$G[M]/\Mmax(G[M])$ is not a pseudo-cograph. Obs.\ \ref{obs:quotient} implies
	that there is an induced subgraph $H\subseteq G$ such that $H\simeq
	G[M]/\Mmax(G[M])$. Note that $H$ is primitive and, in particular, not a
	pseudo-cograph. This together with Lemma \ref{lem:PrimWellPC->F} implies
	that $G$ is not $\forbGT$-free.
\end{proof}

While the set $\forbGT$ characterizes \gatex graphs, pseudo-cographs are not
characterized by $\forbGT$. By way of example, consider the graph $G$ that is
the disjoint union of two induced $P_4$s. Since $G$ is $\forbGT$-free, $G$ is
\gatex. However, $G$ is not a pseudo-cograph as, for all choices $v\in V(G)$,
the graph $G-v$ contains an induced $P_4$ and is, therefore not a cograph (cf.\
Obs.\ \ref{obs:G-v-Cograph}). Hence, we leave it as an open problem to find a
set of forbidden subgraphs that characterizes pseudo-cographs.

\section{\gatex Graphs and Other Graph Classes}

In this section, we show the rich connection of \gatex graphs  to other graph classes.

We start with the following definitions. A \emph{hole} is an induced cycle $C_n$
on $n\geq 5$ five vertices. The complement of a hole is an \emph{anti-hole}. A
graph is \emph{weakly-chordal} if and only if it does not contain holes or
anti-holes \cite{hayward1985weakly}. A graph $G$ is \emph{perfect}, if the
chromatic number of every induced subgraph equals the size of the largest clique
of that subgraph. As shown by \cite{hayward1985weakly}, weakly-chordal graphs
are perfect.

\begin{lemma}\label{lem:odd-hole-free}
	\gatex graphs are hole-free.
\end{lemma}
\begin{proof}
 By contraposition, assume that $G$ contains an induced cycle $C_\ell$ on $\ell
 \geq 5$ vertices. Hence, if $\ell=5$ or $\ell=6$, then $G$ contains $F_0$ or
 $F_{13}$ as an induced subgraph and, otherwise, if $\ell>6$, then $G$ contains
 $F_3$ as an induced subgraph. In all cases, Theorem \ref{thm:F-free} implies
 that $G$ is not \gatex. 
\end{proof}

\begin{theorem}
	\gatex graphs are weakly-chordal and thus, perfect. 	
\end{theorem}
\begin{proof}
 	By Lemma \ref{lem:odd-hole-free}, \gatex graphs are hole free. In addition,
 by Obs.\ \ref{obs:complement-galled-tree}, \gatex graphs are closed under
 complementation which immediately implies that $G$ must be anti-hole free. By
 definition, \gatex graph are therefore weakly-chordal and thus, by
 \cite{hayward1985weakly}, perfect. 
\end{proof}

In all perfect graphs and thus, in all \gatex graphs, the graph coloring
problem, maximum clique problem, and maximum independent set problem can all be
solved in polynomial time \cite{grotschel2012geometric}. 

A graph is \emph{murky}, if the graph and its complement is hole- and $P_6$-free
\cite{HAYWARD:90}. Theorem \ref{thm:F-free} and Lemma \ref{lem:odd-hole-free}
imply
\begin{theorem}
	Every \gatex graph is murky.
\end{theorem}

A directed graph $G' = (V, E')$ is an \emph{orientation} of a graph $G=(V,E)$ if
for all $\{x,y\}\in E$, either $(x, y) \in E'$ or $(y, x) \in E'$ (but not both)
and for all $(x,y) \in E'$, $\{x,y\} \in E$ holds. An orientation $G'=(V',E')$
of a graph $G=(V,E)$ is \emph{transitive} if $(x,y),(y,z)\in E'$ implies
$(x,z)\in E'$ \cite{gallai1967transitiv}. Graphs for which a transitive
orientation exists are known as \emph{comparability graphs}. In what follows, we
show that \gatex graphs are comparability graphs. To this end, we start with

\begin{theorem}\label{thm:gatex-comparabilty}
	Every \gatex graph is a comparability graph.
\end{theorem}
\begin{proof}
 We start with showing that every primitive pseudo-cograph is a comparability
 graph. To this end, let $G=(V,E)$ be a primitive pseudo-cograph. 

 Suppose first 
 that $G$ is slim. By Prop.\ \ref{prop:polcat-edges}, we can construct two
 ordered sets $Y=\{y_1, \ldots, y_{\ell-1}, y_{\ell}=v\}$ and $Z=\{z_1, \ldots,
 z_{m-1}, z_{m}=v\}$, $\ell,m \geq 2$ such that $Y \cap Z=\{v\}$, $Y \cup Z=V$;
 and $\{y_i,y_j\}\in E$ precisely if $1 \leq i<j \leq \ell$ and $i$ is
 odd; and $\{z_i,z_j\}\in E$ precisely if $1 \leq i<j \leq m$ and $i$ is odd.
 Note that, since $G$ is slim, there are no edges connecting vertices in
 $Y\setminus \{v\}$ with vertices in $Z\setminus\{v\}$. We now define a directed
 graph $G'=(V,E')$ by putting $(y_i,y_j)\in E'$ for all $\{y_i,y_j\}\in E$ and
 $(z_i,z_j)\in E'$ for all $\{z_i,z_j\}\in E$ where, in all cases, $i<j$ and $i$
 is odd. Note, that $v=y_\ell=y_m$ where $\ell$ and $m$ are the largest indices
 of the vertices in $Y$ and $Z$, respectively. By definition, there are is no
 arc $(v,x)$ for all $x \in Y\cup Z$ in $G'$.
 
 Consider now two arcs $(a,b)$ and $(b,c)$ in $G'$. We show that $(a,c)\in E'$.
 We consider the following cases: 
 (i) either $a,b,c\in Y$ or $a,b,c\in Z$ and
 (ii) at most two of $a,b,c$ are contained in $Y$ and 
       at most two of $a,b,c$ are contained in $Z$. 
 Suppose first that $a,b,c\in Y$ and let $a=y_p$, $b=y_q$ and $c=y_r$ and thus,
 $p<q<r$ and $p,q$ must be odd. Since $p<r$ and $p$ is odd, there is an edge
 $\{y_p,y_r\}\in E$ and, by construction, an arc $(y_p,y_r)\in E'$. By similar
 arguments, the case $a,b,c\in Z$ is shown. Consider now Case (ii). Since there
 are no arcs in $G'$ connecting vertices in $Y\setminus \{v\}$ with vertices in
 $Z\setminus\{v\}$, one of the vertices $a,b$ and $c$ must be the vertex $v$.
 Since, $v=y_\ell=y_m$ where $\ell$ and $m$ are the largest indices of the
 vertices in $Y$ and $Z$, respectively, 
 there are no arcs $(v,x)$ for all $x \in Y\cup Z$ in $G'$. Therefore, only $c=v$ is
 possible. Hence, if $b\in Y$ then $a\in Y$ which together with $v\in Y$ implies
 $a,b,c\in Y$. Thus, this case cannot occur when we assume that at most two of
 $a,b,c$ are contained in $Y$. By similar arguments, if $b\in Z$, then $a,b,c\in
 Z$ and thus, this case cannot occur. 
 
 Therefore, $G'$ is a transitive orientation of $G$ in case that $G$ is a slim
 primitive pseudo-cograph. 

 Assume, now that $G'$ is a fat primitive pseudo-cograph. By Prop.\
 \ref{prop:polcat-edges}, we can construct two ordered sets $Y=\{y_1, \ldots,
 y_{\ell-1}, y_{\ell}=v\}$ and $Z=\{z_1, \ldots, z_{m-1}, z_{m}=v\}$, $\ell,m \geq
 2$ such that $Y \cap Z=\{v\}$, $Y \cup Z=V$; and
 $\{y_i,y_j\} \in E$ precisely if $1 \leq i<j \leq \ell$ and $i$ is even; and
 $\{z_i,z_j\} \in E$ precisely if $1 \leq i<j \leq m$ and $i$ is even; and 
 $\{y,z\} \in E$ for all $y\in Y\setminus\{v\}$ and $z\in Z\setminus\{v\}$. 
 We now define a directed graph $G'=(V,E')$ by putting 
 $(y_i,y_j)\in E'$ for all $\{y_i,y_j\}\in E$ with $i<j$ and $i$ is even;
 $(z_j,z_i)\in E'$ for all $\{z_i,z_j\}\in E$ with $i<j$ and $i$ is even; 
 $(y,z)\in E'$ for all $y\in Y\setminus\{v\}$ and $z\in Z\setminus\{v\}$. 
 We emphasize that the ``order'' of the arcs $(z_j,z_i)\in E'$ is opposite
 to the ``order'' chosen for slim pseudo-cographs.   
 Note, that $v=y_\ell=y_m$ where $\ell$ and $m$ are the largest indices of the
 vertices in $Y$ and $Z$, respectively. By definition, there are is no arc
 $(v,y)$ for all $y \in Y$, no arc  $(z,v)$ for all $z \in Z$ 
 and no arc $(z,y)$ for all $y\in Y, z\in Z$ in $G'$. 
 
 Consider now two arcs $(a,b)$ and $(b,c)$ in $G'$. We show that $(a,c)\in E'$.
 We consider the following cases: 
 (i') either $a,b,c\in Y$ or $a,b,c\in Z$ and
 (ii') at most two of $a,b,c$ are contained in $Y$ and 
       at most two of $a,b,c$ are contained in $Z$. 
  The existence of the arc $(a,c)$ in $G'$ in Case (i') can be shown  
 by similar arguments as used for Case (i) for slim pseudo-cographs. 
 Hence, assume that Case (ii') holds.
 
 Suppose first that $v\in \{a,b,c\}$. Assume that $a=v$. Since there are no arcs
 $(v,y)$ for all $y\in Y$ in $G'$, we have $b\in Z$. Since there is no arc
 $(b,y)$ for all $y\in Y\setminus \{v\}$ it follows that $c\in Z$. Thus,
 $a,b,c\in Z$ which implies that this case cannot occur when we assume that at
 most two of $a,b,c$ are contained in $Z$. Assume that $b=v$. Since there is
 no arc $(v,y)$ for all $y \in Y$ and no arc $(z,v)$ for all $z \in Z$ in $G'$,
 $a\in Y\setminus \{v\}$ and $c\in Z\setminus\{v\}$ must hold. By construction,
 the arc $(a,c)$ exists in $G'$. Suppose now that $c=v$. By similar arguments as
 used for the case $a=v$, we can conclude that $a,b,c\in Y$; a case that cannot
 occur when we assume that at most two of $a,b,c$ are contained in $Y$.

 Hence, suppose that $v\notin \{a,b,c\}$. If $c\in Y$, then $(z,c)\notin E'$ for
 all $z\in Z$ implies that $b\in Y$ and, subsequently, $a\in Y$; a case that
 cannot occur when we assume that at most two of $a,b,c$ are contained in $Y$.
 Hence, $c\in Z\setminus \{v\}$ must hold. 
 By assumption, there are three cases we need to consider:
 $a\in Z\setminus \{v\}$ and thus,  $b\notin Z\setminus \{v\}$;
 $b\in Z\setminus \{v\}$ and thus,  $a\notin Z\setminus \{v\}$;
 and $a,b\notin Z\setminus \{v\}$. 
 Since there are no arcs $(z,y)$ for all $y\in Y, z\in Z$ in $G'$, 
 the case  $a\in Z\setminus \{v\}$ and, therefore, 
 $b\in Y\setminus \{v\}$ cannot occur. 
 Hence, suppose that  $b\in Z\setminus \{v\}$ and thus, $a\in Y\setminus \{v\}$. 
 By definition and since $c\in Z\setminus \{v\}$, the arc $(a,c)$ exists in $G'$. 
 Now assume that $a,b\notin Z\setminus \{v\}$ and thus, $a,b \in Y\setminus \{v\}$. 
 Again, since since $c\in Z\setminus \{v\}$, the arc $(a,c)$ exists in $G'$.

 Therefore, $G'$ is a transitive orientation of $G$ in case that $G$ is a fat
 primitive pseudo-cograph. 
 
 In summary, every primitive pseudo-cograph is a comparability graph. We note in
 passing that a primitive comparability graph and thus, a primitive
 pseudo-cograph, has only two transitive orientations, where one is obtained
 from the other by reversing the directions of all the edges
 \cite{CS:99,Golumbic-book:04}
 
 We continue with showing that every \gatex graph is a comparability graph. As
 argued in \cite{Spinrad:85} and \cite[Section 4]{CS:99}, a graph is a comparability graph if
 $G[M]/\Mmax(G[M])$ is a comparability graph for all prime modules $M$ of $G$.
 Since for \gatex graphs, Theorem\ \ref{thm:GalledTreeExplainable} implies that
 $G[M]/\Mmax(G[M])$ is a primitive pseudo-cograph and thus, a comparability
 graph, it follows that every \gatex graph is a comparability graph.
\end{proof}

By Obs.\ \ref{obs:complement-galled-tree}, the complement $\overline G$ of a
\gatex graph $G$ is \gatex. This and Thm.\ \ref{thm:gatex-comparabilty} imply 
\begin{corollary}\label{cor:gatex->co-comp}
  The complement of every \gatex graph is a comparability graph and thus, every
  \gatex graph is the complement of a comparability graph.
\end{corollary}

A graph $G=(V,E)$ is a \emph{permutation graph} if there exists a labeling
$\ell$ of the vertices of $G$ and a permutation $\pi=(\pi(1),\dots,\pi(|V|))$
such that for all $u,v\in V$ it holds that $\{u,v\} \in E$ if and only if
$\ell(u)>\ell(v)$ and $\pi^{-1}(\ell(u))<\pi^{-1}(\ell(v))$
\cite{pnueli_lempel_even_1971}.

\begin{theorem}
Every \gatex graph is a permutation graph.
\end{theorem}
\begin{proof}
By Thm.\ \ref{thm:gatex-comparabilty} and Cor.\ \ref{cor:gatex->co-comp}, if $G$
is \gatex, both $G$ and $\overline G$ are comparability graphs. By \cite[Thm.\
3]{pnueli_lempel_even_1971}, $G$ is a permutation graph.
\end{proof}

A graph is \emph{perfectly orderable} if there is an ordering of the vertices of
$G$ such that a greedy coloring algorithm with that ordering optimally colors
every induced subgraph of the given graph. Such an ordering is called perfect.

\begin{theorem}\label{thm:perfect-order-linT}
 \gatex graphs are perfectly orderable and this ordering can be found in linear-time. 
\end{theorem}
\begin{proof}
Let $G$ be \gatex. By Thm.\ \ref{thm:gatex-comparabilty}, $G$ is a comparability
graph and thus, $G$ is transitive orientable. A transitive orientation of
comparability graphs can be constructed in linear-time \cite{CS:97}. A topological
ordering of a transitive orientation then yields a perfect order of $G$ and can
be constructed in linear-time, cf.\ \cite{CS:94,CS:99,Maffray2003}.
\end{proof}

\mh{
For a given graph $G=(V,E)$, a greedy coloring algorithm can be implemented
to run in $O(|V|+|E|)$ time, see \cite[Sec.~6.4]{TurauWeyer:2015}. This together
with Theorem \ref{thm:perfect-order-linT}
yields
\begin{corollary}
 The minimum number $\chi(G)$ of colors needed to color a \gatex graph $G$
 can be determined in linear-time. 
\end{corollary}
}

A graph $G$ is \emph{distance-hereditary}, if the distances in any connected induced subgraph of $G$
are the same as they are in $G$.
In general, \gatex graphs are not distance-hereditary (as an example consider the pseudo-cograph
$\overline P_5$). It has been shown in \cite{HM:90} (see also \cite{DiS:12}) that 
a graph $G$ is distance-hereditary if and only if $G$ does not
contain the \emph{house} $\overline P_5$, the  \emph{gem} $K_1\join P_4$, a hole
and the domino $F_{17}\in \forbGT$ as an induced subgraph. 
By Thm.\ \ref{thm:F-free} and Lemma \ref{lem:odd-hole-free}, 
\gatex graphs are $F_{17}$- and hole-free. We summarize the
latter arguments into the following
\begin{proposition}
Let $G$ be a \gatex graph. Then, $G$ is distance-hereditary
if and only if $G$ is house- and gem-free.
\end{proposition}

The intersection graph of family $\mathcal S$ of nonempty sets is obtained by
representing each set in $\mathcal S$ by a vertex and connecting two vertices by
an edge if and only if their corresponding sets intersect. The intersection
graph of a family of intervals on a linearly ordered set is called an
\emph{interval graph} \cite{Lekkeikerker1962,Golumbic-book:04}. A
\emph{circular-arc graph} is the intersection graph of a set of arcs on the
circle, see \cite{LS:09} for further details. Every interval graph is a
circular-arc graph \cite{LS:09}. 

\begin{theorem}\label{thm:C4freeGatex->intervall}
Every $C_4$-free \gatex graph is an interval graph and, thus a circular-arc graph.
\end{theorem}
\begin{proof}
 Suppose that $G$ is a $C_4$-free \gatex graph. By Cor.\
 \ref{cor:gatex->co-comp}, $G$ is the complement of a comparability graph. By
 \cite[Thm.\ 2]{GH:64}, $H$ is an interval graph if and only if $H$ is
 $C_4$-free and $H$ is the complement of an comparability graph. Hence, $G$ is
 an interval graph and, by \cite{LS:09}, a circular-arc graph.
\end{proof}

A graph $G$ is \emph{very strongly perfect} if for every induced subgraph $H$ of
$G$, every vertex of $H$ belongs to a stable set of $H$ meeting all maximal
cliques in $H$, see \cite{Brandstadt1999} for further details. A \emph{Meyniel}
graph is a graph in which every odd cycle of length five or more has at least
two chords, i.e., edges connecting non-consecutive vertices of the cycle
\cite{MEYNIEL:84}. It has been shown in \cite{HOANG1987302} that a graph $G$ is
Meyniel if and only if $G$ is very strongly perfect. An \emph{odd-building
$B_n$} is a graph isomorphic to an odd cycle $C_n$ on $n\geq 5$ vertices with
just one chord where the chord makes a triangle with two consecutive edges on
the cycle. As an example, a \emph{house} is the odd-building $B_5$. As argued in
\cite{KSS:08}, a graph is Meyniel if and only if the graph contains neither a
hole on an odd number of vertices nor an odd-building as an induced subgraph.
Moreover, a vertex $x$ of a graph $G$ is \emph{soft} if $x$ is either not a
midpoint or not an endpoint of any $P_4$ of $G$. A graph $G$ is \emph{brittle},
if each induced subgraph of $G$ contains a soft vertex
\cite{CHVATAL1987349,Brandstadt1999}. 
\begin{theorem}
	Every house-free \gatex graph is Meyniel,  very strongly perfect and brittle. 
\end{theorem}
\begin{proof}
	Suppose that $G$ is a house-free \gatex graph. By Lemma
	\ref{lem:odd-hole-free}, $G$ is hole-free. Hence, $G$ does, in particular,
	not contain a hole on an odd number of vertices. Suppose, for contradiction,
	that $G$ is not Meyniel. By the latter argument and the arguments preceding
	the statement of this theorem, $G$ must contain an odd building $B_n$ as an
	induced subgraph. Since $G$ is house-free, we have $n\geq 7$. It is easy to
	see that every odd building $B_n$, $n\geq 7$ contains an even cycle
	$C_{n-1}$ on at least six vertices as an induced subgraph. Thus, $G$
	contains holes; a contradiction. Therefore, $G$ is Meyniel. As shown in
	\cite{HOANG1987302}, $G$ is Meyniel if and only if $G$ is very strongly
	perfect. By \cite{HK:88}, if a graph does not contain a hole, a domino
	$F_{17}\in \forbGT$ or a house as an induced subgraph, then this graph is
	brittle \cite{HK:88}. Since $G$ is hole-free and house-free and, in addition
	\gatex, and thus $F_{17}\not\sqsubseteq G$, we can conclude that $G$ must be
	brittle.
\end{proof}

\begin{figure}[t]
	\begin{center}
			\includegraphics[width=0.6\textwidth]{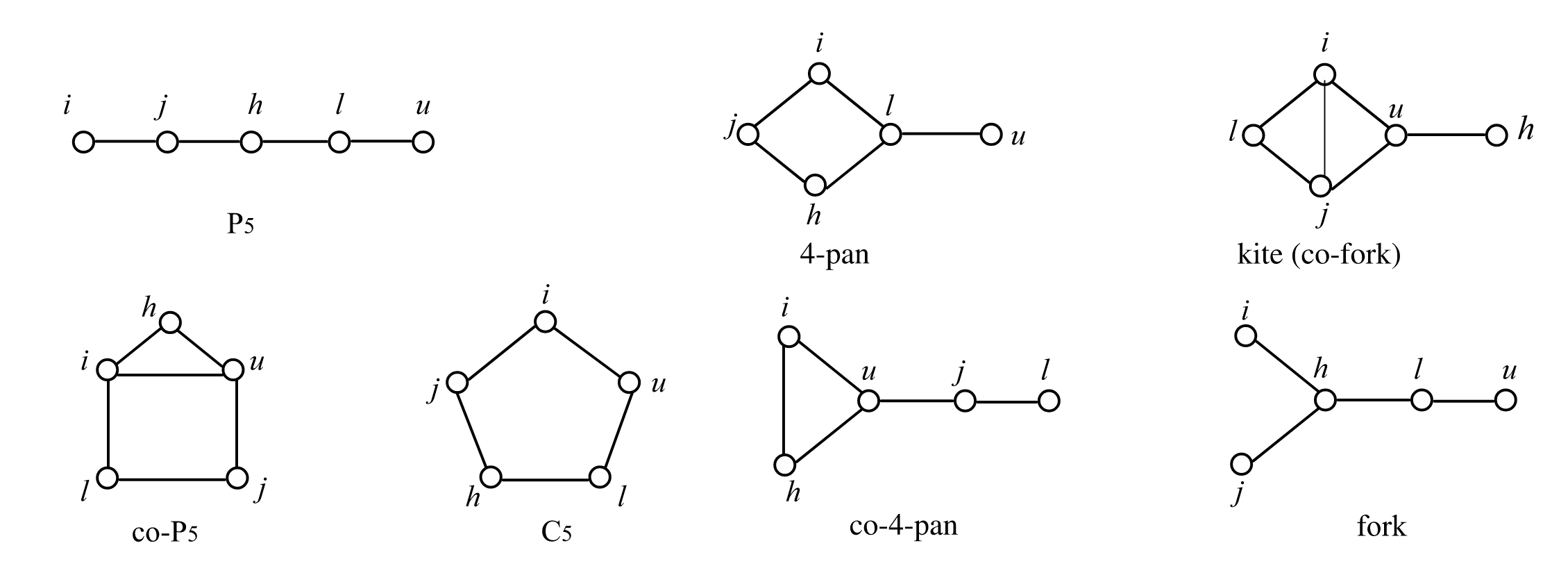}
	\end{center}
	\caption{
			The set $\mathfrak{F}_{\mathrm{P4sparse}}$ of forbidden subgraphs 
			that characterizes $P_4$-sparse graphs \cite[Fig.\ 3]{LIU201245}.
			 }
\label{fig:forbP4-sparse}
\end{figure}

Several generalizations of cographs are obtained by bounding the number of
$P_4$s in different ways. A graph $G$ is \emph{$P_4$-reducible} if every vertex
belongs to at most one induced $P_4$ of the graph \cite{JO:89}. A graph is
\emph{$P_4$-sparse} if every induced subgraph with exactly five vertices
contains at most one induced $P_4$ \cite{hoang1985perfect,Brandstadt1999}.
$P_4$-sparse graphs are characterized by the set of forbidden subgraph
$\mathfrak{F}_{\mathrm{P4sparse}}$ as shown in Fig.\ \ref{fig:forbP4-sparse}
and, moreover, $P_4$-reducible graphs are precisely the $P_4$-sparse graph that
do not contain $F_5$ and $F_6$ as an induced subgraph
\cite{JO:92,GV:97,Brandstadt1999}. Closer inspection of the forbidden subgraphs
in Fig.\ \ref{fig:forbP4-sparse} and the forbidden subgraphs in $\forbGT$ shows
that each $F_i\in \forbGT\setminus\{F_5,F_6\}$ contains at least one of the
graphs in $\mathfrak{F}_{\mathrm{P4sparse}}$ as an induced subgraph.
Consequently, $P_4$-sparse graphs cannot contain any of the graphs $F_i\in
\forbGT\setminus\{F_5,F_6\}$ as an induced subgraph.
In summary we obtain
\begin{theorem}
	Every $\{F_5,F_6\}$-free $P_4$-sparse graph or, equivalently, 
	every $P_4$-reducible graph is \gatex. 
\end{theorem}

\NEW{
Bonnet et al.\ recently introduced a novel parameter called ``twin-width'' as a
non-negative integer measuring a graphs distance to being a cograph
\cite{BKTW:21} that is based one the following characterization of cographs: A
graph is a cograph if it contains two vertices with the same neighborhood
(called twins), identify them, and iterate this process until one ends in a
$K_1$ \cite{bonnet_et_al:LIPIcs.ICALP.2021.35}. Since \gatex graphs form a
generalization of cographs but are still closely linked to cographs, we study
here the twin-width of \gatex graphs.

To formally define the twin-width, we first consider \emph{tri-graphs $G =
(V,B\cupdot R)$} where the edge set $E(G)$ is partitioned into a set \emph{$B$
of black edges} and \emph{$R$ of red edges}. For all $v\in V$ the \emph{red
degree} of $v$ is the degree of $v$ in $(V, R)$. A trigraph $G$ is a
\emph{$d$-trigraph} iff $G$ has maximum red degree at most $d$. Given a trigraph
$G = (V, B\cupdot R)$ and two vertices $u, v \in V$, we define the trigraph
$G/\{u, v\} = (V', B'\cupdot R'$) obtained by ``contracting'' $u, v$ into a
vertex $w$ as follows: $V' = (V \setminus \{u, v\}) \cup \{w\}$ such that $G -
\{u, v\} = (G/\{u, v\}) - \{w\}$ and with the following edges incident to $w$:
\begin{itemize}
\item $\{w,x\} \in B'$ if and only if $\{u,x\} \in B$ and $\{v,x\} \in B$,
\item $\{w,x\} \notin B'\cupdot R'$ if and only if $\{u,x\} \notin B \cupdot R$ and  $\{v,x\} \notin B \cupdot R$, and
\item $\{w,x\} \in R'$ otherwise.
\end{itemize}
In other words, when contracting two vertices $u, v$,  red edges remain red, and red edges
\mh{$\{w,x\}$} are
created for vertex $x$, if $x$ is not adjacent to $u$ and $v$ at the same time. We say that
$G/\{u, v\}$ is a \emph{contraction of $G$}. 

If both $G$ and $G/\{u, v\}$ are $d$-trigraphs, $G/\{u, v\}$ is a $d$-contraction.
A (tri)graph $G$ on $n$ vertices is $d$-collapsible if there exists a sequence of $d$-contractions
which contracts $G$ to a single vertex. More precisely, there is a sequence of $d$-trigraphs
$G = G_n, G_{n-1}, \dots , G_2, G_1$ such that $G_{i-1}$ is a contraction of $G_i$ (hence $G_1$ is the singleton
graph). The minimum $d$ for which $G$ is $d$-collapsible is the \emph{twin-width} of $G$, denoted by $\tww(G)$.
}

\NEW{
\begin{proposition}\label{pr-twin}
Every primitive pseudo-cograph has twin-width $1$. 
\end{proposition}
\begin{proof}
Let $G$ be a primitive $(v,G[Y],G[Z])$-pseudo-cograph.
Since $G$ is not a cograph, $\tww(G)\geq 1$ (cf.\ \cite{BKTW:21}). We next show that $G$ has twin-width at most $1$.
Twin-width is invariant by complementation \cite{BKTW:21}. Hence, we may assume
that w.l.o.g.\ $G$ is slim (otherwise, take the complement of $G$; cf.\
Thm.\ \ref{thm:prop_pc}). Let $Y=\{y_1, \ldots, y_\ell=v\}$, $Z=\{z_1, \ldots,
z_m=v\}$, where the vertices are labelled according to
Proposition~\ref{prop:polcat-edges}. In what follows, $A\symdif B$ denotes
the symmetric difference of two set $A$ and $B$. We 
show now that $G$ is $1$-collapsible and provide
a sequence of $1$-contractions transforming $G$ into the singleton graph $K_1$.
in such a way that, after each contraction, the resulting graph contains exactly one red edge.
We do this via a two steps process.

\emph{Step 1:} We provide a sequence of $1$-contractions to transform $G$
into a slim $(v,G[\{y_1,v\}],G[Z])$-pseudo cograph $G'$ such that $G'$
either has no red edge or has $\{y_1,v\}$ as its unique red edge. If $|Y|=2$
we put $G'\coloneqq G$ and immediately proceed with \emph{Step 2}. Suppose for now that $|Y|>2$.

If $|Y|=3$, then $N(y_1) \symdif N(y_2)=\{v\}$. In this case, we contract $y_2$
into $y_1$ (i.e, we keep the vertex label $y_1$ for the new vertex). Since
$N(y_1) \symdif N(y_2)=\{v\}$, $G_1\coloneqq G/\{y_2, y_1\}$ is a $1$-contraction with one
resulting red edge $\{y_1,v\}$. Moreover, one can easily verify 
that $G_1$ is a slim $(v,G[\{y_1,v\}],G[Z])$-pseudo cograph. In this case, we put $G'\coloneqq G_1$, and continue with \emph{Step 2}.

If $|Y| \geq 4$, then by definition of slim pseudo-cographs, $N(y_3) \symdif
N(y_1) = \{y_2\}$. We contract $y_3$ into $y_1$ (i.e, we keep the vertex label
$y_1$ for the new vertex). Since $N_G(y_3) \symdif N_G(y_1) = \{y_2\}$,
$G_1\coloneqq G/\{y_3, y_1\}$ is a $1$-contraction with one resulting red edge
$\{y_1,y_2\}$ in $G_1$. If $|Y|=4$, then $G_1$ is a slim
$(v,G[\{y_1,y_2,v\}],G[Z])$-pseudo-cograph with a single red edge $\{y_1,y_2\}$.
In this case, we can apply to $G_1$ the construction as described for the
case $|Y|=3$ to obtain a sequence of $1$-contractions until we end in 
a slim $(v,G[\{y_1,v\}],G[Z])$-pseudo cograph $G'$ and continue with \emph{Step 2}.
If otherwise, $|Y|>4$, then
$N_{G_1}(y_4)=N_{G_1}(y_2)=\{y_1\}$. We now contract $y_4$ into $y_2$ (i.e, we
keep the vertex label $y_2$ for the new vertex). Since
$N_{G_1}(y_4)=N_{G_1}(y_2)=\emptyset$ and since $\{y_1,y_2\}$ is a red-edge in
$G_1$, the resulting graph $G_2 \coloneqq G_1/\{y_4,y_2\}$ has $\{y_1,y_2\}$ as
unique red edge. Moreover, one easily verifies that $G_2$ is a
slim $(v,G[Y'],G[Z])$-pseudo-cograph for $Y' = Y-\{y_2,y_3\}$. Hence, we can
repeat the latter operations on $G_2$ (all resulting in $1$-contraction)
until we end in a slim $(v,G[\{y_1,v\}],
G[Z])$-pseudo-cograph $G'$ with $\{y_1,v\}$ as its unique red edge.

\emph{Step 2:} We provide now a sequence of $1$-contractions to transform
$G'$ into the singleton graph $K_1$ by a sequence of $1$-contractions, in
such a way that, after each contraction, the resulting graph has $\{y_1,v\}$ as
unique red edge. We first contract $z_{m-1}$ into $v$ (i.e, we keep the
vertex label $v$ for the new vertex). Note that $N_{G'}(z_{m-1})\symdif
N_{G'}(v) = \{y_1\}$. Hence, $G'_1 \coloneqq G'/\{z_{m-1}, v\}$ is a
1-contraction with one resulting red edge $\{y_1,v\}$. In particular, $G'_1$ is
a slim $(v,G'[Y],G'[Z-\{z_{m_1}\}])$-pseudo-cograph. Thus, we can repeat the
latter process on $G'_1$, until we end in a graph $G''$ that contains only the
vertices $y_1$ and $v$ and the red edge $\{y_1,v\}$. Finally, contraction of
$y_1$ and $v$ yields then the single vertex graph $K_1$, as desired. Since
all steps result in a $1$-contraction and since we end in a $K_1$, the initial
primitive pseudo-cograph $G$ is 1-collapsible.
\end{proof}
}

\begin{theorem}
Every \gatex graph $G$  has twin-width at most $1$. 
	In particular, $\tww(G)=1$ for a \gatex graph $G$ if and only if $G$ is not a cograph.
\end{theorem}
\begin{proof}
$G$ is a cograph precisely if $\tww(G) = 0$ (cf.\ \cite{BKTW:21}). Assume that
$G$ is not a cograph. 
By \cite[Thm.\ 3.1]{TWW:sat},  the twin-width of a graph coincides with the maximal
twin-width of its primitive induced subgraphs. 
By Theorem~\ref{thm:CharPolCat} and Theorem~\ref{thm:GatexIFFallPrimitive=PsC}, 
all primitive induced subgraphs of \gatex graph are pseudo-cographs. 
Moreover, a \gatex graph that is not a cograph admits at least one primitive induced 
subgraph. This together Proposition~\ref{pr-twin} implies that $\tww(G) =  1$. 
\end{proof}

\section{Linear-time algorithms for hard problems  in \gatex graphs}

A \emph{clique} of $G$ is an inclusion-maximal complete subgraph $G$. The size
of a maximum \emph{clique} of a graph $G$ is denoted by $\omega(G)$. For
simplicity, we denote with $\omega(w)$ the size of a maximum clique in
$G[L(N(w))]$.  A
\emph{coloring} of $G$ is a map $\sigma\colon V(G)\to S$, where $S$ denotes a
set of colors, such that $\sigma(u)\neq \sigma(v)$ for all $\{u,v\}\in E(G)$.
The minimum number of colors needed for a coloring of $G$ is called the
\emph{chromatic number} of $G$ and denoted by $\chi(G)$. A subset $W\subseteq
V(G)$ of pairwise non-adjacent vertices is called \emph{independent set}. The
size of a maximum independent set in $G$ is called the \emph{independence
number} of $G$ and denoted by $\alpha(G)$. In general, determining the
invariants $\omega(G)$, $\chi(G)$ and $\alpha(G)$ for arbitrary graphs is an
NP-hard task \cite{garey1979computers}. In contrast, we show here that $\omega(G)$,
$\chi(G)$ and $\alpha(G)$ can be computed in linear-time for \gatex graphs $G$.
The crucial idea for the linear-time algorithms  is
to avoid working on the \gatex graphs $G$ directly, but to use the the galled-trees that
explain $G$ as a guide for the algorithms to compute these invariants. 

Galled-trees that explain a given \gatex graph $G$ can be obtained from the 
modular decomposition trees $(\MDT_G,t_G)$ by replacing  prime vertices
locally by simple rooted cycles. We will formalize this concept in more
detail in Def.\ \ref{def:pvr}. To this end, we need additional definitions. 
We denote with $P^1(C),P^2(C)$ the \emph{sides of $C\subseteq N$}, 
i.e., the two directed paths $C$ with the same start-vertex $\rho_C$ and
end-vertex $\eta_C$, and whose vertices distinct from $\rho_C$ and $\eta_C$ are
pairwise distinct. Moreover
	 $G_1(M), G_2(M)\subseteq G[M]$ will denote the subgraphs induced by leaf-descendants of
	the vertices in $P^1(C_M)-\rho_{C_M}$ and $P^2(C_M)-\rho_{C_M}$, respectively.

A galled-tree $N$ with leaf-set $L$ is \emph{elementary} if it contains a single 
rooted cycle $C$ of length $|L|+1$ with root $\rho_C = \rho_N$
and single hybrid-vertex $\eta_C\in V(C)$ and additional edges
$\{v_i,x_i\}$ such that every vertex $v_i\in V(C)\setminus \{\rho_C\}$ is adjacent to
unique vertex $x_i\in L$. 
A labeling $t$ (or equivalently $(N,t)$) is \emph{quasi-discriminating} if $t(u)\neq t(v)$ for all $(u,v)\in$ with $v$ not
being a hybrid-vertex. 
A galled-tree is \emph{strong} if it \emph{does not} contain
cycles of the following form:  (i) $P^1(C)$ or $P^2(C)$ consist of $\rho_C$ and $\eta_C$ only or
(ii) both $P^1(C)$ and $P^2(C)$ contain only one vertex distinct from $\rho_C$ and
$\eta_C$. 

To obtain a galled-tree by locally replacing prime vertices $v$ in $(\MDT_G,t_G)$
with rooted cycles, we first compute for  $M = L(\MDT_G(v))$ 
the quotient  $H = G[M]/\Mmax(G[M])$. In particular, $H$ is primitive and
can, therefore, be explained by a strong elementary quasi-discriminating galled-tree $(N_v,t_v)$ (cf.\ \cite[Thm.\ 6.10]{HS:22}). 
We then use the rooted cycles in  $(N_v,t_v)$ to replace $v$ in $(\MDT_G,t_G)$. 
The latter is formalized as follows.

\begin{definition}[prime-vertex replacement (pvr) networks]
  \label{def:pvr}
    Let $G$ be a \gatex graph and $\mathcal{P}$ be the set of all prime vertices in
	$(\MDT_G,t_G)$. A \emph{prime-vertex replacement} (\emph{pvr}) networks $(N,	
	t)$ of $G$ (or equivalently, $(\MDT_G,t_G)$) is obtained by the following procedure:
\begin{enumerate}[noitemsep] 
\item For all $v\in \mathcal{P}$, let $(N_v,t_v)$ be a
	  strong quasi-discriminating elementary network with root $v$
	  that explains $G[M]/\Mmax(G[M])$ with $M = L(\MDT_G(v))$.			
  \label{step:Gv} \smallskip
\item For all $v\in \mathcal{P}$, remove all edges $(v,u)$ with
  		$u\in \child_{\MDT_G}(v)$ from $\MDT_G$ to obtain the forest
  		$(T',t_G)$ and \label{step:T'}  
		add $N_v$ to $T'$ by identifying the root
  	of $N_v$ with $v$ in $T'$ and each leaf $M'$ of $N_v$ with the
  	corresponding child $u\in \child_{\MDT_G}(v)$ for which $M' = L(\MDT_G(u))$.  \smallskip

\noindent
 \emph{This results in the pvr graph $N$.}\smallskip
\item \label{step:color} 
 Define the labeling $t\colon V(N)\to \{0,1,\odot\}$ by putting, for
  all $w\in V(N)$,
  \begin{equation*}
    t(w) =
    \begin{cases} 
      t_G(v) &\mbox{if } v\in V(\MDT_G)\setminus \mathcal P \\
      	t_v(w) &\mbox{if } w \in V(N_v)\setminus X \text{ for some } v\in \mathcal P
    \end{cases}
  \end{equation*}
\end{enumerate}
\end{definition}


We later reference, we provide

\begin{observation}\label{ref:properties_C_pvr}
Let $(N,t)$ be a pvr-network of a \gatex graph $G$. Then, 
\begin{itemize}
\item $(N,t)$ is a galled tree that explains $G$ \cite[Prop.\ 7.4]{HS:22}.
\item There is a 1:1 correspondence between the cycle $C$ in $N$ and 
	  prime modules $M$ of $G$ \cite[Prop.\ 8.3]{HS:22}.
	  
	  Hence, we can define 
      $C_M$ as the unique cycle in $N$ corresponding to prime module $M$. 
\end{itemize}
Moreover, let $v$ be a prime vertex associated with the prime module $M_v=L(\MDT_G(v))$ module 
and let $C \coloneqq C_{M_v}$. Since we used strong elementary networks for the replacement
of $v$, one easily verifies that:
\begin{itemize}
\item $C$ has a unique root $\rho_C$ and a unique hybrid-vertex $\eta_C$. 
\item $\eta_C$ has precisely one child and precisely two parents. 
\item All vertices $v\neq \eta_C$ have two children and one parent. 

	 In particular, 
	  all vertices $v\neq \eta_C,\rho_C$
	   have one child $u'$ located in $C$
	  and one child $u''$ that is not located in $C$
	  and these children satisfy  $L(N(u'))\cap L(N(u''))=\emptyset$
	  and it holds that $\lca_N(x,y)=w$ for all $x\in L(N(u'))$ and $y\in L(N(u''))$.

	  Both children $u'$ and $u''$ of 
	  $\rho_C$ are located in $C$ and satisfy  $L(N(u'))\cap L(N(u''))=L(N(\eta_C))$. 

\end{itemize}

\end{observation}

We start  first with the computation of the size $\omega(G)$ of maximum
cliques in \gatex graphs. 

In the upcoming proofs we may need to compute the join $H'\join H$ where $H$ is the empty graph
and we put, in this case, $H'\coloneqq H'\join H$.

\begin{lemma}\label{lem:max-clique-in-eta}
	Let $G$ be a \gatex graph that is explained by the pvr-network $(N,t)$ and
	suppose that $G$ contains a prime module $M$. Put $L_{\eta}\coloneqq
	L(N(\eta_{C_M}))$ and let $H\in \{G[M], G_1(M)\}$.
	If $H$ contains a maximum clique $K$ with
	vertices in $L_\eta$, then $V(K)\cap L_\eta$ induces a maximum clique in
	$G[L_\eta]$ and $(V(K)\setminus L_\eta)\cupdot V(K')$
	induces a maximum clique in $H$ for every   maximum clique $K'$ in $G[L_\eta]$. 
\end{lemma}
\begin{proof}
	Let $G$ be a \gatex graph that is explained by the pvr-network $(N,t)$ and
	suppose that $G$ contains a prime module $M$. Put $L_{\eta}\coloneqq
	L(N(\eta_{C_M}))$ and $p=\rho_{C_M}$. In the following, let $H\in
	\{G_1(M),G[M]\}$. To recall, $p$ has precisely two children where one them
	is located on $P^1(C)-p$ and the other on $P^2(C)-p$. Moreover, $\eta$ has
	precisely one child. Thus, one easily verifies that $\lca_N(x,z)=\lca_N(y,z)
	\in (V(P^1(C))\cup V(P^2(C))) -p)$ for all $x,y\in L_{\eta}$ and for all
	$z\in V(H)\setminus L_\eta$. 

	Suppose that $H$ contains a maximum clique $K$ that contains vertices in
	$L_\eta$. Since $K$ is a clique in $H$, it must hold that $t(\lca_N(x,z)) =
	1$ for all $x \in V(K)\cap L_\eta$ and $z\in V(K)\setminus L_\eta\subseteq
	V(H)\setminus L_\eta$. By definition of pvr-networks, $L_\eta$ is a module
	of $G$ and, therefore, $t(\lca_N(x,z)) = 1$ with $x \in V(K)\cap L_\eta$ and
	$z\in V(K)\setminus L_\eta$ implies that $t(\lca_N(x',z)) = 1$ for all $x'
	\in L_\eta$. By construction, we have $V(K) = (V(K)\setminus L_\eta)\cupdot
	(V(K) \cap L_\eta)$. Assume, for contradiction, that $V(K)\cap L_\eta$ does
	not induce a maximum clique in $G[L_\eta]$. In this case, there is a clique
	$K'$ in $G[L_\eta]$ such that $|V(K')|>|V(K)\cap L_\eta|$. By the previous
	arguments, $t(\lca_N(x',z)) = 1$ for all $x' \in V(K')$ and $z\in
	V(K)\setminus L_\eta$ which implies that $(V(K)\setminus L_\eta)\cupdot
	V(K')$ induces a complete graph in $H$. However, $|(V(K)\setminus
	L_\eta)\cupdot V(K')|> |(V(K)\setminus L_\eta)\cupdot (V(K) \cap L_\eta)| =
	|V(K)|$; a contradiction to $K$ being a maximum clique in $H$. 	
	Therefore, $V(K)\cap L_\eta$ induces a maximum clique in $G[L_\eta]$.
	
	Finally, let $K'$ be some maximum clique in $G[L_\eta]$ and thus, $|V(K)\cap L_\eta|=|V(K')|$. 
	As argued before, $t(\lca_N(x',z)) = 1$ for all $x' \in V(K')$ and $z\in
	V(K)\setminus L_\eta$ which implies that $(V(K)\setminus L_\eta)\cupdot
	V(K')$ induces a complete graph $K''$ in $H$ of size 
	$|V(K'')| = |V(K)\setminus L_\eta| + |V(K')| = |V(K)\setminus L_\eta| + |V(K)\cap L_\eta| = |V(K)|$.
	Hence, $K''$ is a maximum clique in $H$. 
\end{proof}

\begin{lemma}\label{lem:Leta1}
	Let $G$ be a \gatex graph that is explained by the pvr-network $(N,t)$ and
	suppose that $G$ contains a prime module $M$ such that $t(\rho_{C_M})=1$. If
	both, $G_1(M)$ and $G_2(M)$, contain a maximum clique  
	with vertices in
	$L(N(\eta_C))$, then $G[M]$ contains a maximum clique with vertices in
	$L(N(\eta_C))$.
	
	In this case, there are maximum cliques  $K'$ and $K''$ in $G_1(M)$
	and $G_2(M)$, respectively, such that $V(K')\cup V(K'')$ induce a
	maximum clique in $G[M]$.
\end{lemma}
\begin{proof}
	Let $G$ be a \gatex graph that is explained by the pvr-network $(N,t)$ and
	suppose that $G$ contains a prime module $M$. Put $p=\rho_{C_M}$ and assume
	that $t(p)=1$. Put $G_1\coloneqq G_1(M)$, $G_2\coloneqq G_2(M)$ and
	$L_{\eta}\coloneqq L(N(\eta_{C_M}))$. For a subgraph $H\subseteq G$ we put
	$|H|\coloneqq |V(H)|$.
		
	Suppose that $G_1$, resp., $G_2$ contains a maximum clique $K'$, resp., $K''$
	that contains vertices in $L(N(\eta_C))$. By Lemma
	\ref{lem:max-clique-in-eta}, we can assume w.l.o.g.\ that $V(K')\cap L_{\eta}
	= V(K'') \cap L_{\eta}$ induce a maximum clique $K^\eta$ in $G[L_\eta]$. We
	show that $V(K')\cup V(K'')$ induces a maximum clique in $G[M]$. Let $K^1$,
	resp., $K^2$ be the complete subgraph of $K'$, resp., $K''$ that is induced by
	$V(K')\setminus L_{\eta}$, resp., $V(K'')\setminus L_{\eta}$. Since $t(p)=1$,
	all vertices in $V(K^1)$ are adjacent to all vertices in $V(K^2)$ and thus,
	the subgraph induced by $V(K^1) \cupdot V(K^2)$ coincides with $K^1\join
	K^2$. Since $K^\eta$ is a complete graph, we have $K' = K^1\join K^\eta$ and
	$K'' = K^2\join K^\eta$, The latter two arguments imply that $K''' \coloneqq
	K^1\join K^\eta\join K^2$ is a complete graph in $G[M]$ that is induced by
	$V(K')\cup V(K'')$. Assume, for contradiction, that $K'''$ is not a maximum
	clique in $G[M]$. Let $\hat K$ be a maximum clique in $G[M]$ and thus,
	$|\hat K|>|K'''|$. Let $\hat K^i$ be the complete subgraph of $\hat K$
	induced by the vertices $(V(G_i)\cap V(\hat K)) \setminus L_{\eta}$, $i\in
	\{1,2\}$. There are two cases, either $\hat K$ contains vertices in $L_\eta$
	or not. 

	If $\hat K$ contains vertices in $L_\eta$, then we can apply Lemma
	\ref{lem:max-clique-in-eta} and assume that $V(\hat K)\cap L_\eta =
	V(K^\eta)$ induces the maximum clique $K^\eta$ in $G[L_\eta]$. By
	construction, $V(K''') = V(K^1)\cupdot V(K^\eta)\cupdot V(K^2)$ and,
	moreover, $V(\hat K) = V(\hat K^1)\cupdot V(K^\eta)\cupdot V(\hat K^2)$. If
	$|\hat K^1|\leq |K^1|$ and $|\hat K^2|\leq |K^2|$, then $|\hat K| = |\hat
	K^1| + |K^\eta| + |\hat K^2|\leq |K^1| + |K^\eta| + |K^2| = |K'''|$, which is
	impossible as, by assumption, $|\hat K|>|K'''|$. Thus, $|\hat K^1|>|K^1|$ or
	$|\hat K^2|>|K^2|$ must hold. W.l.o.g.\ we may assume that $|\hat
	K^1|>|K^1|$. But then, $|V(\hat K^1)\cupdot V(K^\eta)|> |V(K^1)\cupdot
	V(K^\eta)|$ which together with the fact that $V(\hat K^1)\cupdot
	V(K^\eta)\subseteq V(\hat K)$ induce a complete graph in $G_1$ implies that
	$V(K^1)\cupdot V(K^\eta) = V(K)$ cannot induce a maximum clique in $G_1$; a
	contradiction. 

	If $\hat K$ does not contain vertices in $L_\eta$, then $V(\hat K) = V(\hat
	K^1)\cupdot V(\hat K^2)$ and, in particular, $\hat K = \hat K^1\join \hat
	K^2$. If $|\hat K^1|\leq |K^1|$ and $|\hat K^2|\leq |K^2|$, then $|\hat K^1|
	+ |\hat K^2|\leq |K^1| + |K^2|$ which is not possible since $|\hat K| =
	|\hat K^1| + |\hat K^2|>|K'''| = |K^1| + |K^\eta| + |K^2|>|K^1| + |K^2|$
	holds. Thus, $|\hat K^1|>|K^1|$ or $|\hat K^2|>|K^2|$. We can now re-use
	similar arguments as in the previous case to obtain a contradiction. 
	
	In summary, $K'''$ is a maximum clique in $G[M]$ that contains vertices in 
	$L_\eta$ and is induced by $V(K')\cup V(K'')$.
\end{proof}

\begin{lemma}\label{lem:Leta2}
	Let $G$ be a \gatex graph that is explained by the pvr-network $(N,t)$
	and suppose that $G$ contains a prime module $M$ such that
	$t(\rho_{C_M})=1$. 
	If $G[M]$ contains a maximum clique that contains vertices
	in $L(N(\eta_{C_M}))$, then $G_1(M)$ and $G_2(M)$ have both a maximum clique
	that contains vertices in $L(N(\eta_{C_M}))$.
	%
\end{lemma}
\begin{proof}
	Let $G$ be a \gatex graph that is explained by the pvr-network $(N,t)$ and
	suppose that $G$ contains a prime module $M$. Put $p=\rho_{C_M}$ and assume
	that $t(p)=1$. Furthermore, put $G_1\coloneqq G_1(M)$, $G_2\coloneqq
	G_2(M)$, $L_{\eta}\coloneqq L(N(\eta_{C_M}))$ and $G_\eta\coloneqq G[L_\eta]$.
	For a subgraph $H\subseteq G$ we define $|H|\coloneqq |V(H)|$. Let $K$ be a
	maximum clique in $G[M]$ that contains vertices in $L_\eta$ and put $K^1
	\coloneqq (G_1\cap K) - G_\eta$, $K^2 \coloneqq (G_2\cap K) - G_\eta$ and
	$K^\eta \coloneqq K \cap G_\eta$. Thus, $V(K) = V(K^1)\cupdot
	V(K^\eta)\cupdot V(K^2)$. 
	
	Assume, for contradiction, that every maximum clique in $G_1$ does not
	contain vertices in $L_\eta$. Let $K'$ be a maximum clique in $G_1$. Since
	$V(K^1)\cupdot V(K^\eta)\subseteq V(G_1)$ and $V(K^1)\cupdot V(K^\eta)$
	induce a complete graph with vertices in $L_\eta$, we can conclude that
	$|V(K^1)\cupdot V(K^\eta)| = |K^1| +|K^\eta|<|K'|$. Let $K''$ be a maximum
	clique in $G_2$. Since $V(K^\eta)\cupdot V(K^2) \subseteq V(G_2)$ induces a
	complete subgraph in $G_2$, we have $|K^2| +|K^\eta|\leq |K''|$. Assume
	first that $K''$ that does not contain vertices in $L_\eta$. In this case,
	$K''' \coloneqq K'\join K''$ forms a complete graph in $G[M]$ since
	$\lca_N(x,y)=p$ has label $1$ for all $x\in V(K')$ and $y\in V(K'')$. By
	construction, $|K'|+|K''| = |K'''|$. Moreover, since $K$ is a maximum clique
	in $G[M]$ $|K'''|\leq |K|$ must hold. This together with the fact that
	$|K_\eta|\neq 0$ implies \[ |K'|+|K''| = |K'''| \leq |K| = |K^1| + |K^\eta|
	+ |K^2| < |K^1 + 2|K^\eta| + |K^2| < |K'|+|K''|,\] which yields a
	contradiction. Thus, $K''$ must contain vertices in $L_\eta$. Note that
	$\lca_N(x,y)=p$ has label $1$ for all $x\in V(K')$ and $y\in V(K^2)$ and
	thus, $K'''' \coloneqq K'\join K^2$ forms a complete graph in $G[M]$ which
	together with $|K''''| = |K^1| +|K^\eta|<|K'|$ and $|K''''|\leq |K|$ yields
	the following contradiction: \[ |K'|+|K^2| = |K''''|\leq |K| =
	|K^1|+|K^\eta|+|K^2|< |K'|+|K^2|.\] Hence, $G_1$ must contain a maximum
	clique with vertices in $L_\eta$. By similar arguments, $G_2$ must contain a
	maximum clique with vertices in $L_\eta$. 
\end{proof}

\begin{lemma}\label{lem:not-Leta}
	Let $G$ be a \gatex graph that is explained by the pvr-network $(N,t)$ and
	suppose that $G$ contains a prime module $M$ such that $t(\rho_{C_M})=1$. 
	Put $L_\eta \coloneqq  L(N(\eta_C))$ and $G_\eta = G[L_\eta]$. 
	If (at least) one of $G_1(M)$ and $G_2(M)$  does not contain a maximum clique  
	with vertices in	$L_\eta$, then
	$K$ is a maximum clique in $G[M]$ if and only if
	$K\cap G_1(M)$ and  $K\cap G_2(M)$ is a maximum clique
	in $G_1(M)-G_\eta$ and $G_2(M)-G_\eta$, respectively. 
%
%
%
%
\end{lemma}
\begin{proof}
	Let $G$ be a \gatex graph that is explained by the pvr-network $(N,t)$ and
	suppose that $G$ contains a prime module $M$. Put $p=\rho_{C_M}$ and assume
	that $t(p)=1$. Furthermore, put $G_1\coloneqq G_1(M)$, $G_2\coloneqq
	G_2(M)$, $L_{\eta}\coloneqq L(N(\eta_{C_M}))$ and $G_\eta\coloneqq G[L_\eta]$.
	Assume that (at least) one of $G_1$ or $G_2$ does not contain a maximum clique  
	with vertices in $L_\eta$.  By contraposition of Lemma \ref{lem:Leta2}, 
	$G[M]$ cannot contain  a maximum clique with vertices in $L_\eta$. 
	
	Let $K$ be a maximum clique of $G[M]$. Since $V(K)\cap L_\eta = \emptyset$, 
	$K$ consists entirely of vertices in $G_1-G_\eta$ and $G_2-G_\eta$. 
	Hence $V(K) = (V(K)\cap V(G_1))\cupdot (V(K)\cap V(G_2))$. 
	Let $K'$ and $K''$ be the two complete subgraphs in $G_1$ and $G_2$
	that are induced by $V(K)\cap V(G_1)$  and $V(K)\cap V(G_2)$, 
	resectively. Since $K$ contains no vertices in $L_\eta$, 
	$K'$ and $K''$ are subgraphs of $G_1-G_\eta$ and $G_2-G_\eta$. 
	and, in particular, $K = K'\join K''$. 
	Assume, for contradiction, that $K'$ is not a maximum clique in $G_1(M)-G_\eta$. 
	Hence, there is a larger clique $K'''$ in $G_1(M)-G_\eta$. Since
	$K'''$, resp., $K''$ is a subgraph of $G_1-G_\eta$, resp., $G_2-G_\eta$
	we have $\lca(x,y)=p$ for all $x\in V(K''')$ and $y\in V(K'')$. 
	This and  $t(p)=1$ implies that $K'''\join K''$ is a complete
	subgraph in $G[M]$ with more vertices than $K$; a contradiction. 
	Hence, $K'$ must be a maximum clique in $G_1(M)-G_\eta$. 
	Similarity,  $K''$ must be a maximum clique in $G_1(M)-G_\eta$. 
	
	Assume now that $K'$ is a maximum clique in $G_1(M)-G_\eta$ and $K''$ is be
	a maximum clique in $G_1(M)-G_\eta$. By similar arguments as in the previous
	case, $K = K'\join K''$ is a complete subgraph in $G[M]$. Assume, for
	contradiction, that $K$ is not a maximum clique in $G[M]$. In this case,
	there is a larger complete subgraph $\hat K$ in $G[M]$. 
	Let $\hat K'$ and $\hat K''$ be the two complete subgraphs in $G_1$ and $G_2$
	that are induced by $V(\hat K)\cap V(G_1)$  and $V(\hat K)\cap V(G_2)$, 
	As argued at the
	beginning of this proof, $V(\hat K)\cap L_\eta = \emptyset$ must hold
	and therefore, $\hat K = \hat K' \join \hat K''$ and, moreover, 
	$\hat K'$, resp., $\hat K''$	is a subgraph of 
	$G_1-G_\eta$, resp., $G_2-G_\eta$. 
	However, since $|V(\hat K')|+|V(\hat K'')| =|V(\hat K)| >|V(K)| = |V(K')|+|V(K'')|$
	it must hold  $|V(\hat K')|>|V(K')|$ or $|V(\hat K'')|>|V( K'')|$;
	 a contradiction to $K'$ being a maximum clique in $G_1(M)-G_\eta$
	 and $K''$ being a maximum clique in $G_2(M)-G_\eta$.
\end{proof}


\begin{proposition}\label{prop:clique-Number-M-series}
	Let $G$ be a \gatex graph that is explained by the pvr-network $(N,t)$
	and suppose that $G$ contains a prime module $M$. Assume that $t(\rho_{C_M})=1$. 
	Put $L_\eta = L(N(\eta_{C_M}))$, $G_1=G_1(M)$, $G_2=G_2(M)$  and $G_\eta = G[L_\eta]$
	If both, $G_1$ and $G_2$, contain a maximum clique  
	with vertices in $L\eta$, then 
	\[\omega(G[M]) = \omega(G_1)+\omega(G_2)-\omega(G_\eta) \]
	Otherwise, if at least one of $G_1$ or $G_2$ does not contain
	a maximum clique with vertices in $L_\eta$, then
	\[\omega(G[M]) = \omega(G_1-G_\eta)+\omega(G_2-G_\eta) \]		
\end{proposition}
\begin{proof}
	Let $G$ be a \gatex graph that is explained by the pvr-network $(N,t)$ and
	suppose that $G$ contains a prime module $M$. Put $p=\rho_{C_M}$ and assume
	that $t(p)=1$. Furthermore, put $G_1\coloneqq G_1(M)$, $G_2\coloneqq
	G_2(M)$, $L_{\eta}\coloneqq L(N(\eta_{C_M}))$ and $G_\eta\coloneqq G[L_\eta]$.

	Assume first that both, $G_1$ and $G_2$, contain a maximum clique  
	with vertices in $L_\eta$. By Lemma \ref{lem:Leta1}, 
	there are maximum cliques  $K'$ and $K''$ in $G_1$
	and $G_2$, respectively, such that $V(K')\cup V(K'')$ induce a
	maximum clique $K$ in $G[M]$.
	By Lemma \ref{lem:max-clique-in-eta}, we can assume w.l.o.g.\ that
 	$V(K')\cap L_\eta=V(K'')\cap L_\eta$ induce the same maximum clique $K'''$
 	in $G_\eta$. Hence, $\omega(G_\eta)=|V(K''')|$ and 
 	thus, 
 	\begin{itemize}
 	\item[]$\qquad\qquad$   	$\omega(G_1) = |V(K')|  =\ |V(K')\setminus L_\eta| \ + |V(K')\cap L_\eta|  = \ |V(K')\setminus L_\eta| + \omega(G_\eta)  $ and
 	\item[]$\qquad\qquad$ 		$\omega(G_2) = |V(K'')| = |V(K'')\setminus L_\eta| + |V(K'')\cap L_\eta|= |V(K'')\setminus L_\eta| + \omega(G_\eta) $. \smallskip
 	\end{itemize}
 	Taken the latter arguments together, we obtain
 	\[\omega(G[M]) = |V(K)|= |V(K')\cup V(K'')|=   |(V(K')\setminus L_\eta)\cupdot V(K''')\cupdot (V(K'')\setminus L_\eta)|
 	 = \omega(G_1(M))+\omega(G_2(M))-\omega(G_\eta).\]

	Assume now that (at least) one of $G_1$ or $G_2$ does not contain a maximum clique  
	with vertices in $L_\eta$. Let $K'$ and $K''$ be maximum cliques in 
	$G_1-G_\eta$ and $G_2-G_\eta$, respectively. 
	Since $K'$ and $K''$ have no vertices in common and since
	$\lca(x,y)=p$ for all $x\in V(K')$ and $y\in V(K'')$ and
	$t(p)=1$, it follows that $K = K'\join K''$ is a complete subgraph
	in $G[M]$. By Lemma \ref{lem:not-Leta}, $K$ is a maximum clique 
	in $G[M]$. Hence, $\omega(G[M]) = |V(K)| = |V(K')| + |V(K'')|
	=\omega(G_1-G_\eta) + \omega(G_2-G_\eta)$.
 \end{proof}

\begin{figure}[t]
	\begin{center}
			\includegraphics[width=.9\textwidth]{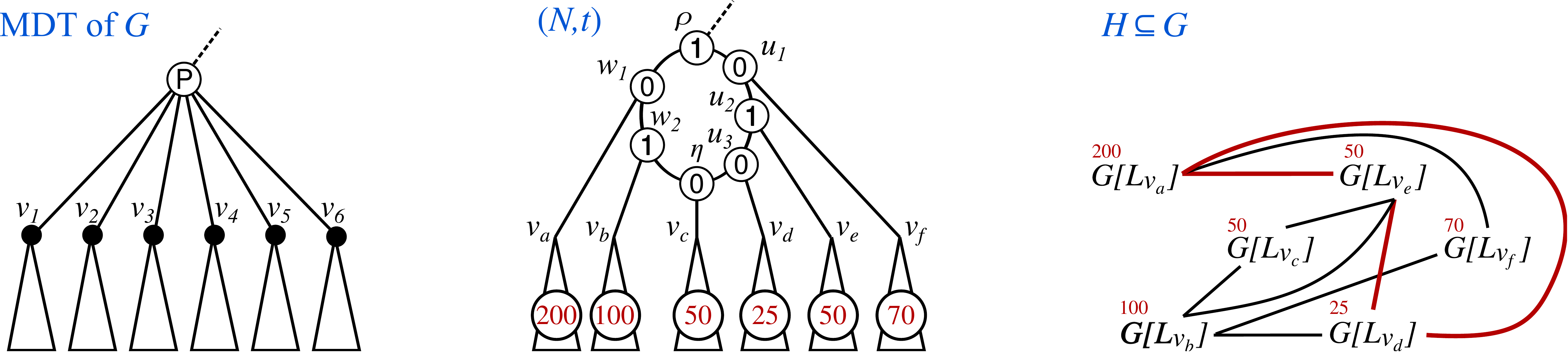}
	\end{center}
	\caption{Shown is a part of the modular decomposition tree $(\MDT_G,t_G)$ (right) and a resulting
			 pvr-network $(N,t)$ (middle) which explains the subgraph $H = G[L(N(\rho))]$ of a \gatex graph $G$. 
			 A schematic drawing of $H$ is shown right. Red numbers indicate
			 the size of a maximum clique in the respective subgraphs. 
			For the indices $a,b,c,d,e,f$ of vertices in $N$ it holds that 
			$\{a,b,c,d,e,f\}  = \{1,2,3,4,5,6\}$. A maximum clique in $H$
			of size $\omega(H)=275$
			(computed with Algorithm~\ref{alg:clique})
			is highlighted in red.
			Note that a maximum clique in $G_1(M) = G[L(N(w_1))]$ has size $200$
			and that a maximum clique in $G_2(M) = G[L(N(u_1))]$ has size $100$
			which sums up to $300$. However, none of the maximum clique in $G_1(M)$
			can contain vertices in $L_\eta$. Therefore, 
			the size of a maximum clique in $G[M]$ with $M = L_{\rho}$
			is composed of the two maximum cliques in $G_1(M) - G_\eta$ and
			$G_2(M) - G_\eta$ where $G_\eta=G[L(\eta)]$. Thus, 
			$\omega(H) = \omegaExclHyb(w_1) + \omegaExclHyb(u_1)=200 + 75=275$, 
			see Example \ref{exmpl:alg} for more details.
			 }
\label{fig:example}
\end{figure}

\begin{proposition}\label{prop:clique-Number-M-parallel}
	Let $G$ be a \gatex graph that is explained by the pvr-network $(N,t)$
	and suppose that $G$ contains a prime module $M$. If $t(\rho_{C_M})=0$, then
	$\omega(G[M]) = \max\{\omega(G_1(M)),\omega(G_2(M))\}$.
\end{proposition}
\begin{proof}
	Let $G$ be a \gatex graph that is explained by the pvr-network $(N,t)$
	and suppose that $G$ contains a prime module $M$. Put $p=\rho_{C_M}$ and assume
	that $t(p)=0$. Furthermore, put $G_1\coloneqq G_1(M)$, $G_2\coloneqq
	G_2(M)$, $L_{\eta}\coloneqq L(N(\eta_{C_M}))$ and $G_\eta\coloneqq G[L_\eta]$.
	
	Let $K$ be a maximum clique in $G[M]$. Note first $K$ cannot contain 
	vertices $x$ and $y$ such that $x\in V(G_1-G_\eta)$ and $y\in V(G_2-G_\eta)$ 
	since, in this case,  $\lca(x,y)=p$ and $t(p)=0$ imply that $\{x,y\}\notin E(G[M])$. 
	Hence, $K$ must be entirely contained in either $G_1$ or $G_2$. 
	Moreover, any maximum clique in $G_1$ and $G_2$ provide a complete
	subgraph of $G[M]$. Taken the latter two arguments together, 
	$\omega(G[M]) = \max\{\omega(G_1(M)),\omega(G_2(M))\}$.
\end{proof}

\begin{example}\label{exmpl:alg}
We exemplify here the main step of Algorithm~\ref{alg:clique} (starting at Line \ref{l:v-prime})
using the example given in  Fig.~\ref{fig:example}. In what follows, we put for a vertex $x$, 
$L_x\coloneqq L(N(x))$ assume that $P^1(C)$ and $P^2(C)$ denote the sides of $C$ that 
contains $w_1$ and $u_1$, respectively. Let $\rho$ and $\eta$ denote the unique root
and hybrid-vertex of $C$, respectively. In the algorithm, $\omega(x)$ always denotes
the size of a maximum clique in subgraph $G[L_x]$. Consider the prime vertex $v$, i.e., 
$M=L(\MDT_G(v))=L_\rho$ is a prime module. 

We start to initialize the parents of $w_2$ and $u_3$ of $\eta$ (Line \ref{l:base-case-start} - \ref{l:base-case-end})
and put $\omegaExclHyb(w_2) \coloneqq \omega(v_b) = 100$ and  	$\omegaExclHyb(u_3) \coloneqq \omega(v_d)=25$.
Since $t(w_2)=1$, we have $G[L_{w_2}] = G[L_{v_b}]\join G[L_{v_c}]$ and put 
$\omega(w_2)\coloneqq \omega(v_b) + \omega(v_c)=150$. 
Here, $\mathit{clique\_incl\_}L_\eta(1)$ remains unchanged and, therefore, \texttt{true}.
Since $t(u_3)=0$, we have $G[L_{u_3}] = G[L_{\eta}]\cupdot G[L_{v_d}]$ and put 
$\omega(u_3)\coloneqq \max\{\omega(\eta),  \omega(v_d)\} = 50$.
It holds that $\omega(v_d) \leq \omega(\eta)$ and thus, 
$\mathit{clique\_incl\_}L_\eta(2)$ remains unchanged and, therefore, \texttt{true}.

In Line \ref{l:init-innerC-start} - \ref{l:init-innerC-end}, we consider 
the vertices in  $P^1(C)$ and $P^2(C)$ that are distinct for the two parents
of $\eta$ and distinct from $\eta$ and $\rho$ in bottom-up fashion. 
Thus, we first consider $P^1(C)$ and 
compute $\omega(w_1)$ and $\omegaExclHyb(w_1)$. 
Since $t(w_1)=0$, we put 
 $\omega(w_1) \coloneqq \max\{\omega(w_2),\omega(v_a)\}=200$ and 
 $\omegaExclHyb(w_1) \coloneqq \max\{\omegaExclHyb(w_2),\omega(v_a)\} = 200$. 
 Since $\omega(v_a)=200 >\omega(w_2)=100$ and 
 $G[L_{w_1}] = G[L_{v_a}]\cupdot G[L_{w_2}]$ and $L_{\eta}\cap L_{v_a}=\emptyset$, none of the maximum cliques
 in  $G[L_{w_1}]$ can contain vertices in $L_\eta$ and we put 
 $\mathit{clique\_incl\_}L_\eta(1)\coloneqq \texttt{false}$. 

We now consider  $P^2(C)$ and vertex $u_2$. 
Since $t(u_2)=1$, we put $\omega(u_2) \coloneqq \omega(u_3) + \omega(v_e)= 100$
and $\omegaExclHyb(u_2) \coloneqq \omegaExclHyb(u_3) + \omega(v_e) = 75$.
We then consider $u_1$. Since $t(u_1)=0$,
 $\omega(u_1) \coloneqq \max\{\omega(u_2),\omega(v_f)\} = 100$ and 
 $\omegaExclHyb(u_1) \coloneqq \max\{\omegaExclHyb(u_2),\omega(v_f)\}=75$. 
It holds that $\omega(v_f) \leq \omega(u_2)$ and thus, 
$\mathit{clique\_incl\_}L_\eta(2)$ remains unchanged and, therefore, \texttt{true}.

We now compute $\omega(v)$ for the subgraph $G[L(\MDT_G(v))] = G[L_\rho]$. 
Since, $\mathit{clique\_incl\_}L_\eta(1) =\texttt{false}$, none of the 
maximum cliques in $G[L_{w_1}]=G_1(M)$ can contain vertices in $L_\eta$. 
However, $\mathit{clique\_incl\_}L_\eta(2) =\texttt{true}$ and thus
there are maximum cliques in $G[L_{u_1}]=G_2(M)$ that contain vertices in $L_\eta$. 
This is, in particular, the reason why we also tracked the value 
$\omegaExclHyb(u_1)$. In the algorithm, we proceed with Line \ref{l:else-rho}
and compute in Line \ref{l:series2}, 
$\omega(v) \coloneqq \omegaExclHyb(w_1) + \omegaExclHyb(u_1) = 200 + 75=275$
to conclude that $H=G[L(\MDT_G(v))]=G[L_\rho]$ contains a maximum clique
of size $275$. 
\end{example}

\begin{algorithm}[H]
  \small 
  \caption{\texttt{Computation of $\omega(G)$} }
  \label{alg:clique}
  \begin{algorithmic}[1]
    \Require  A \gatex graph $G=(V,E)$
    \Ensure   Size $\omega(G)$ of maximum clique in $G$ 

    \State Compute $(\MDT_G,t_G)$ and pvr-network $(N,t)$ of $G$ \label{l:MDT-pvr}

    \State Initialize $\omega(v)\coloneqq 1$ for all leaves $v$ in $\MDT_G$ \label{l:omega-leaves}
    \ForAll{$v\in V(\MDT_G)\setminus{L(\MDT_G)}$ in postorder} \label{l:forV}
    
    	\If{$t_G(v)=0$}
    	 	$\omega(v) \coloneqq \max_{w\in \child_{\MDT_G}(v)} \{\omega(w)\}$ \label{l:v-parallel}
    	\ElsIf{$t_G(v)=1$}	
    		$\omega(v) \coloneqq \sum_{w\in \child_{\MDT_G}(v)} \{\omega(w)\}$ \label{l:v-series}
    	\Else  \Comment{$t_G(v)=\mathrm{prime}$ } \label{l:v-prime}

    	 	\State Let $C\coloneqq C_M$ be the unique cycle in $N$ for which $M \coloneqq L(N(\rho_{C})) = L(\MDT_G(v)) $
    	 	\State Let $\eta$ be the unique hybrid in $C$ and put $\omega(\eta)\coloneqq \omega(u)$, 
    	 			where $u$ is the unique child of $\eta$ in $N$ \label{l:eta}
		
			\State $\mathit{clique\_incl\_}L_\eta(1) \coloneqq$\texttt{true}
			\State $\mathit{clique\_incl\_}L_\eta(2) \coloneqq$\texttt{true} \medskip  \medskip

			\LComment{Init $\omega(w)$ and $\omegaExclHyb(w)$ for the parents $w\in \parent_N(w)$ of $\eta$} 
			\ForAll{$w\in \parent_N(\eta)$} \label{l:base-case-start}
				\State $u'' \coloneqq \child_N(w)\setminus \{\eta\}$ 
				\State $\omegaExclHyb(w) \coloneqq \omega(u'')$ 	
				\If{ $t(w) = 0$ } 
						 $\omega(w) \coloneqq \max\{\omega(\eta),\omega(u'')\}$
						\If{$\omegaExclHyb(w)=\omega(u'')>\omega(\eta)$} 
							 $\mathit{clique\_incl\_}L_\eta(i) \coloneqq$\texttt{false} \label{l:Bool1}
						\EndIf

				\Else \
					 	$\omega(w) \coloneqq \omega(\eta)+\omega(u'')$ 
				\EndIf
												
			\EndFor	\label{l:base-case-end}	\medskip

			\LComment{Init  $\omega(w)$ and $\omegaExclHyb(w)$ for the vertices $w\neq \rho_c,\eta$ and $w\notin \parent_N(\eta)$ along the sides of $C$ bottom-up}
			\State Let $P^1$ and $P^2$ 	   be the two sides of $C$ \label{l:sides}
			\ForAll{$w\in V(P^i)\setminus(\{\eta,\rho_C\}\cup \parent_N(\eta))$ in postorder, $i\in \{1,2\}$}  \label{l:init-innerC-start}
			
				\State $u'\coloneqq \child_N(w)\cap V(C)$ and $u''\coloneqq \child_N(w)\setminus V(C)$  \Comment{Note, $\child_N(w)=\{u',u''\}$}
				
				\If{ $t(w) = 0$ } 
					
					\State $\omega(w) \coloneqq \max\{\omega(u'),\omega(u'')\}$
					\State $\omegaExclHyb(w) \coloneqq \max\{\omegaExclHyb(u'),\omega(u'')\}$ 
					\If{$\omega(u'')>\omega(u')$} 
						 $\mathit{clique\_incl\_}L_\eta(i) \coloneqq$\texttt{false} \label{l:Bool2}
					\EndIf
				
				\Else	
				
					\State $\omega(w) \coloneqq \omega(u') + \omega(u'')$
					\State $\omegaExclHyb(w) \coloneqq \omegaExclHyb(u') + \omega(u'')$
					
				\EndIf

			\EndFor \label{l:init-innerC-end} 

			\LComment{Init $\omega(\rho_C)$. Note, $\rho_C$ corresponds to $v$ in $\MDT_G$}
			\State Let $u'$ and $u''$ be the two children of $\rho_C$	 \label{l:rho1}
			\If{$t(\rho_C) = 0$} 
					 $\omega(v) \coloneqq \max\{\omega(u'), \omega(u'')\}$ \label{l:parallel}
			\Else 
				\If{$\mathit{clique\_incl\_}L_\eta(1) =\mathit{clique\_incl\_}L_\eta(2) =$\texttt{true}} \Comment{$G_1(M)$ and $G_2(M)$ contain max-cliques with vertices in $L_\eta$ } \smallskip
					
				\State		\ $\omega(v) \coloneqq \omega(u') + \omega(u'')-\omega(\eta)$ \label{l:series1}
						
				\Else \label{l:else-rho}		\Comment{$G_1(M)$ or $G_2(M)$ does not 
									 contain any max-clique with vertices in $L_\eta$ }
				\State		 $\omega(v) \coloneqq \omegaExclHyb(u') + \omegaExclHyb(u'')$ \label{l:series2}
				\EndIf	\label{l:rho2}
			\EndIf
%
%
			%
				   
    	\EndIf \label{l:end-v-prime}
    \EndFor \label{l:forV-end}
  \end{algorithmic}
\end{algorithm}

\begin{theorem}
Algorithm \ref{alg:clique} correctly computes the clique number $\omega(G)$
of \gatex graphs $G=(V,E)$ and can be implemented to run in $O(|V|+|E|)$ time. 
\end{theorem}
\begin{proof}
We start with proving the correctness of Algorithm \ref{alg:clique}. Let $G=(V,E)$. 
We first compute $(\MDT_G,t_G)$ and a 
pvr-network $(N,t)$ of $G$ and, for $w\in V(N)$, put $L_w \coloneqq L(N(w))$. 
For a vertex $w\in V(\MDT_G)$, let $M_w\coloneqq L(\MDT_G(w))$ denote 
the module of $G$ ``associated'' with $w$.

We first initialize $\omega(v) = 1$ for all leaves $v$ in $\MDT_G$
and, thus, correctly capture the size $\omega(G[L_v])=\omega(v)$ of a maximum clique in $G[L_v]\simeq K_1$  (Line \ref{l:omega-leaves}). 
We then continue to traverse the remaining vertices in $\MDT_G$ in postorder. This
ensures that whenever we reach a vertex $v$ in $\MDT_G$, all its children have
been processed.  We show now that $\omega(v)\coloneqq \omega(G[M_v])$ is correctly
computed for all $v\in V(\MDT_G)$. 
Let $v$ be the currently processed vertex in Line \ref{l:forV}. 
By induction, we can assume that the children $u$ of $v$ in $\MDT_G$
satisfy  $\omega(u) = \omega(G[M_u])$.

If $t_G(v)=0$, then $M_v$ is a parallel module and,  therefore, 
$G[M_v] = \cupdot_{w\in \child_{\MDT_G}(v)} G[M_w]$. Since 
$G[M_v]$ is the disjoint union of  the graphs
$G[M_w]$ with  $w\in \child_{\MDT_G}(v)$, it follows that every
maximum clique must be located entirely in one of the subgraphs 
$G[M_w]$ of $G[M_v]$. Consequently, 
$\omega(v) \coloneqq \max_{w\in \child_{\MDT_G}(v)} \{\omega(w)\}$
is correctly determined in Line \ref{l:v-parallel}, 
i.e., $\omega(v)=\omega([G[M_v]])$ holds.

If $t_G(v)=1$, then $M_v$ is a series module and,  therefore, 
$G[M_v] = \join_{w\in \child_{\MDT_G}(v)} G[M_w]$. Since 
$M_w\cap M_{w'}=\emptyset$ and all vertices in $M_w$
are adjacent to all vertices in $M_w'$ for distinct children
$w$ and $w'$ of $v$, it follows that a maximum clique in $G[M_v]$
is composed of the maximum cliques in $G[M_w]$, $w\in \child_{\MDT_G}(v)$. 
Consequently, 
$\omega(v) \coloneqq \sum_{w\in \child_{\MDT_G}(v)} \{\omega(w)\}$
is correctly determined in Line \ref{l:v-series}, 
i.e., $\omega(v)=\omega([G[M_v]])$ holds.

Assume now that $t(v)=\mathrm{prime}$ and thus, that $M\coloneqq M_v$ is a prime module of $G$. 
In this case, $v$ is locally replaced by a cycle $C\coloneqq C_{M}$ according to Def.\ \ref{def:pvr}
and  we have $M = L(\MDT_G(v))=L_{\rho_C}$ (cf.\ Obs.\ \ref{ref:properties_C_pvr}).
Out task is now to determined the clique number $\omega(v) \coloneqq \omega(G[M])$ of $G[M]$. 

To this end, we consider the vertice of $C$ in bottom-up fashion
and  start with the unique hybrid-vertex $\eta$ of $C$ in Line \ref{l:eta}. 
By Obs.\ \ref{ref:properties_C_pvr}, $\eta$ has precisely one child $u$ and, therefore, 
$L_\eta=L_u$. Hence, $\omega(\eta)\coloneqq \omega(u) = \omega(G[L_u])$ and, 
since $G[L_u]=G[L_\eta]$, $\omega(\eta) = \omega(G[L_\eta])$ is correctly determined. 

In the following, we will record two values $\omega(w)$ and $\omegaExclHyb(w)$  
for the vertices $w\neq \rho_C,\eta$ in $C$ to capture the size $\omega(w)$ of a maximum
clique in $G[L_w]$ and the size $\omegaExclHyb(w)$ of a maximum
clique in $G[L_w\setminus L_\eta]$. 

In Line \ref{l:base-case-start} - \ref{l:base-case-end} we compute these two
values for the two unique parents of $\eta$, both must be located in $C$ (cf.\
Obs.\ \ref{ref:properties_C_pvr}). Let $w\in \parent_N(\eta)$. By Obs.\
\ref{ref:properties_C_pvr}, $w$ has precisely two children $\eta$ and $u''$ and
it holds that $L_\eta\cap L_{u''}=\emptyset$. 
This, together with the induction assumption and the fact that $u''$ has been
processed, implies that $\omega(u'') = \omega(G[L_u''])$ is correctly computed.
Since $L_w\setminus L_\eta = L_{u''}$, the value $\omegaExclHyb(w) =
\omega(u'')$ correctly determines the size of a largest clique in
$G[L_w\setminus L_\eta]$. To determine $\omega(w)$, we distinguish two cases:
$t(w)=0$ and $t(w)=1$. 
By Obs.\ \ref{ref:properties_C_pvr}, $\lca_N(x,y)=w$ for all $x\in L_{u'}$ and all $y\in L_{u''}$. 
Hence, all $x\in L_{u'}$ and all $y\in L_{u''}$ are either adjacent (case $t(w)=1$) or not adjacent 
(case $t(w)=0$). Now, we can use similar
arguments as for the cases $t_G(v)=1$ and $t_G(v)=0$
to conclude that $\omega(w)$ is correctly computed. 

The Boolean value $\mathit{clique\_incl\_}L_\eta(i)$ is used to record as
whether a maximum clique $G[L_w]$ contains vertices in $L_\eta$ (\texttt{true})
or not (\texttt{false}), $i\in\{1,2\}$. Initially, we put
$\mathit{clique\_incl\_}L_\eta(i) \coloneqq \texttt{true}$ and thus, assume that
if $w\in P^i(C)$, that a maximum clique in $G[L_w]$ can contain vertices in $L_\eta$. We will argue
later that $\mathit{clique\_incl\_}L_\eta(i)$ correctly records as whether there
is a maximum clique in $G_i(M)$ with vertices in $L_\eta$ (\texttt{true}) or if
none of the maximum cliques in $G_i(M)$ contain vertices in $L_\eta$
(\texttt{false}), $i\in\{1,2\}$.

For the parents $w$ of $\eta$, we put $\mathit{clique\_incl\_}L_\eta(i) =
\texttt{false}$ precisely if $w\in P^i(C)$, 
$t(w)=0$ and $\omegaExclHyb(w) = \omega(u'') =
\omega(G[u''])> \omega(\eta) = \omega(G[L_\eta])$. Since $t(w)=0$, there are no
edges between vertices in $L_\eta$ and $L_{u''}$ and, thus, any maximum clique
of $G[L_w]$ must be located entirely in either $G[L_\eta]$ or $G[L_{u''}]$.
Since $\omegaExclHyb(w)> \omega(\eta)$ none of the maximum clique in $G[L_w]$
contains vertices in $L_\eta$ and we correctly put
$\mathit{clique\_incl\_}L_\eta(i) \coloneqq \texttt{false}$. In case $t(w)=1$,
we have $\omega(G[L_w]) = \omega(w) = \omega(\eta) + \omega(u'') =
\omega(G[L_\eta]) + \omega(G[L_{u''}])$. This and the fact that $G[L_w] =
G[L_\eta]\join G[L_{u''}]$ implies that any maximum clique in $G[L_w] $
contains vertices in $L_\eta$. Hence, we keep $\mathit{clique\_incl\_}L_\eta(i)$
as $\texttt{true}$. 

In Line \ref{l:init-innerC-start} - \ref{l:init-innerC-end}, we consider all
vertices $w\in V(C)\setminus\{\eta, \rho_C\}$ and $w\notin\parent_N(\eta)$ in a bottom-up order. Let $w$ be
the currently processed vertex. Again, by Obs.\ \ref{ref:properties_C_pvr}, $w$
has precisely two children $u'$ and $u''$ where $u'$ is located on $C$ while
$u''$ is not. Both vertices $u'$ and $u''$ have already been processed and, by
induction assumption, $\omega(u') = \omega(G[L_{u'}])$, $\omegaExclHyb(u') =
\omega(G[L_{u'}\setminus L_\eta])$, and $\omega(u'') = \omega(G[L_{u''}])$ have
been correctly computed. 

Suppose that $t(w)=0$. Then, we put $\omega(w) = \max\{\omega(u'),\omega(u'')\}$
and by similar argument as used in the previous case ($w$ as a parent of $\eta$), $\omega(w) = \omega(G[L_w])$
is correctly computed. 
 Consider now  $\omegaExclHyb(w)$.
 By Obs.\ \ref{ref:properties_C_pvr}, 
since $\lca_N(x,y) = w$ for all $x\in L_{u'}\setminus L_\eta\subseteq L{u'}$ and $y\in L_{u''}$
and $t(w)=0$, there are no edges between vertices in $G[L_{u'}\setminus L_\eta] $ and $G[L_{u''}]$. 
Hence, $G[L_w\setminus L_\eta] = G[L_{u'}\setminus L_\eta]\cupdot G[L_{u''}]$.
This together with $L_\eta\cap L_{u''}=\emptyset$ implies that
$\omegaExclHyb(w) \coloneqq \omega(G[L_w\setminus L_\eta]) = 
  \max\{\omega(G[L_{u'}\setminus L_\eta]), \omega(G[L_{u''}])\}=
  \max\{\omegaExclHyb(u'), \omega(u'')\}$ 
  is correctly computed. 
  
 Suppose now that  $t(w)=1$. Then, we put $\omega(w) = \omega(u') + \omega(u'')$
and by similar argument as used in the previous case ($w$ as a parent of $\eta$), $\omega(w) = \omega(G[L_w])$
is correctly computed. 
 Consider now  $\omegaExclHyb(w)$.
 Since $\lca_N(x,y) = w$ for all $x\in L_{u'}\setminus L_\eta$ and $y\in L_{u''}$, 
 and $t(w)=1$, all vertices in $G[L_{u'}\setminus L_\eta] $ are adjacent to
 all vertices in $G[L_{u''}]$. Hence, $G[L_w\setminus L_\eta] = G[L_{u'}\setminus L_\eta]\join G[L_{u''}]$
 and therefore, $\omegaExclHyb(w) \coloneqq \omega(G[L_w\setminus L_\eta]) = 
  \omega(G[L_{u'}\setminus L_\eta]) + \omega(G[L_{u''}])=
  \omegaExclHyb(u') + \omega(u'')$ 
  is correctly computed.

 In summary, in Line \ref{l:init-innerC-start} - \ref{l:init-innerC-end} the values
 $\omega(w)=\omega(G[L_w])$ and
 $\omegaExclHyb(w)=\omega(G[L_w\setminus L_\eta])$ have been correctly computed for all
 $w\in V(C)\setminus(\{\eta, \rho_C\}\cup \parent_N(\eta))$ 
  In particular, the two children of 
  $\rho_C$ correspond to the last vertex in $P^1(C)$, resp., $P^2(C)$  
  (in post-order). Let $w$ be one of the children of $\rho_C$. 
  By construction, $G[L_w] = G_i(M)$ for some $i\in \{1,2\}$. 
  Hence, by the previous arguments, $\omega(w) = \omega(G_i(M))$
  and $\omegaExclHyb(w) = \omega(G_i(M)-L_\eta)$ have been correctly
  computed.

 Consider now $G_i(M)$ 
 and let $w$ be the child of $\rho_C$ that is located in $P^i(C)$, $i\in \{1,2\}$. 
 Clearly, if $\omega(w) = \omega(G_i(M)) > \omegaExclHyb(w) = \omega(G_i(M)-L_\eta)$
 then $G_i(M)$ contains a maximum clique with vertices in $L_\eta$. 
 However, $\omega(w) =\omegaExclHyb(w)$ only implies that there is a 
 is a maximum clique with vertices not in $L_\eta$. However, it does not imply
 that there is no maximum clique with vertices in $L_\eta$. 
 To keep track of the latter, we use the $\mathit{clique\_incl\_}L_\eta(i)$. 
 As argued above, $G_1(M)$ is stepwisely composed by the union or join
 of the two children of vertices $w'$ along $P^i(C)$. 
 It remains to show that $\mathit{clique\_incl\_}L_\eta(i)=\texttt{false}$ precisely if 
 $G_1(M)$ does not contain any maximum clique with in $L_\eta$. 
 Let $V(P^i(C))\setminus \{\rho_C,\eta\}= \{w_1,w_2,\dots,w_k\}$
 such that $w_k\prec_N w_{k-1}\prec_N \dots \prec_N w_1 = w$. 
 Moreover, we denote with $u'_{\iota}$ the child of $w_{\iota}$ in $C$ with
  $u''_{\iota}$ the child of $w_{\iota}$ not in $C$.

 Suppose that $\mathit{clique\_incl\_}L_\eta(i)=\texttt{false}$ after we processed
 vertex $w$. In this case, there was a vertex $w_{\iota}$, $1\leq {\iota} \leq k$ such that $t(w_{\iota})=0$ and
 $\omega(u''_{\iota}) > \omega(u'_{\iota})$ (Line \ref{l:Bool1} or \ref{l:Bool2}). As argued above, 
 $G[L_{w_{\iota}}]$ cannot contain a maximum clique with vertices in $L_{u'}$ and since $\eta \prec_N u'$, 
 $G[L_{w_{\iota}}]$ cannot contain a maximum clique with vertices in $L_\eta$.
 Construction of $\omega(w) = \omega(G_i(M))$ involves in each step the values 
 $\omega(w_j)$ and thus, the size of maximum cliques in $G[L_{w_j}]$, $2\leq j\leq k$. 
 Hence, $G_i(M)$ cannot contain a maximum clique with vertices in $L_\eta$. 
 
 Suppose now that none of the maximum cliques in  $G_i(M)$ contain vertices in $L_\eta$. 
 Since $\omega(w) = \omega(G_i(M))$ involves in each step the values 
 $\omega(w_j)$ and thus, the size of maximum cliques in $G[L_{w_j}]$, $2\leq j\leq k$, 
 there must have been one $w_{\iota}$, $1\leq {\iota} \leq k$, such that
 $G[L_{w_{\iota}}]$ does not contain a maximum clique with vertices in $L_\eta$. 
 Let $w_{\iota}$ be the vertex with the largest index ${\iota}$ for which this property
 is satisfied.  Hence, for all $w_j$ with $j<{\iota}$ the subgraph $G[L_{w_j}]$ contain maximum cliques
 with vertices in $L_\eta$. Note that $u'_{\iota}$ corresponds either to $w_{{\iota}+1}$
 or to $\eta$. In the latter case, ${\iota}=k$ and any maximum clique in $G[L_{u'_k}]$ trivially
 contains verices in $L_\eta$. 
  
 If $t(w_{\iota})=1$, then, in all steps, $\omega(w_{\iota}) = \omega(u_{\iota}')+\omega(u_{\iota}'')$
 and, by the latter arguments,  a maximum clique in $G[L_{w_{\iota}}]$
 is composed of maximum cliques in  $G[L_{u'_{\iota}}]$ and $G[L_{u''_{\iota}}]$.
 By choice of $w_{\iota}$ and since $\eta \preceq_N u'_{\iota}$, there is a maximum clique in $G[L_{u'_{\iota}}]$ 
 with vertices in $L_\eta$. 
 Hence, $G[L_{w_{\iota}}]$ contains a maximum clique with vertices in $L_\eta$, a contradiction. 
 Therefore, $t(w_{\iota})=0$ must hold. As argued above, 
 $G[L_{w_{\iota}}] = G[L_{u'_{\iota}}]\cupdot G[L_{u''_{\iota}}]$ and, thus, every maximum clique in
 $G[L_{w_{\iota}}]$ is entirely contained in either $G[L_{u'_{\iota}}]$ or $G[L_{u''_{\iota}}]$. 
 Since, however,  $G[L_{w_{\iota}}]$ does not contain a maximum clique with vertices in $L_\eta$
  and since $\eta \preceq_N u'_{\iota}$ it must hold that
 $\omega(u_{\iota}'')>\omega(u_{\iota}')$ and we correctly set $\mathit{clique\_incl\_}L_\eta(i)=\texttt{false}$. 
  
 In summary,  $\mathit{clique\_incl\_}L_\eta(i)$
 correctly records as whether there
is a maximum clique in $G_i(M)$ with vertices in $L_\eta$ (\texttt{true}) or if
none of the maximum cliques in $G_i(M)$ contain vertices in $L_\eta$
(\texttt{false}), $i\in\{1,2\}$.

 Finally, we compute $\omega(v)$ in Line \ref{l:rho1} - \ref{l:rho2}.
 To recall, $v$ is the vertex in $(\MDT_G,t_G)$ with label $t_G(v)=\mathrm{prime}$. 
 Note that $G[L(\MDT_G(v))] = G[L_{\rho_C}]$ and thus, $\omega(v) = \omega(G[L_{\rho_C}])$. 
 Let $u'$ and $u''$ be the two children of $\rho_C$ (cf.\ Obs.\ \ref{ref:properties_C_pvr})
 and assume w.l.o.g.\ that $u'\in P^1(C)$ and $u''\in P^2(M)$. 
 Note that, by the previous arguments, 
 $\omega(u')=\omega(G_1(M))$, $\omega(u'')=\omega(G_2(M))$,
 $\omegaExclHyb(u') = \omega(G_1(M)-G_\eta)$
 and $\omegaExclHyb(u') = \omega(G_2(M)-G_\eta)$
  have been 
 correctly computed. 
 
 If $t(\rho_C)=0$ we put in Line \ref{l:parallel}, $\omega(v) \coloneqq \max\{\omega(u'), \omega(u'')\}$
 By Prop.\ \ref{prop:clique-Number-M-parallel}, $\omega(v)$ coincides with 
 $\omega(G[L(\MDT_G(v))])$. Otherwise, $t(\rho_C)=1$ must hold. 
 If both Boolean values $\mathit{clique\_incl\_}L_\eta(1)$ and $\mathit{clique\_incl\_}L_\eta(2)$ are 
 \texttt{true}, then there are always maximum cliques in $G_1(M)$ and $G_2(M)$
 with vertices in $L_\eta$. 
 In this case, we put in Line \ref{l:series1}
 $\omega(v) \coloneqq \omega(u') + \omega(u'')-\omega(\eta)$. By Prop \ref{prop:clique-Number-M-series}, 
  $\omega(v)$ coincides with  $\omega(G[L(\MDT_G(v))])$. 
 Otherwise,  at least one $G_1(M)$ and $G_2(M)$ contains no maximum clique with vertices
 in $L_\eta$, i.e., $\mathit{clique\_incl\_}L_\eta(1)$ or $\mathit{clique\_incl\_}L_\eta(2)$ is
 \texttt{false}. 
In this case, we put in Line \ref{l:series2}
 $\omega(v) \coloneqq \omegaExclHyb(u') + \omegaExclHyb(u'')$. 
 By Prop.\ \ref{prop:clique-Number-M-series}, $\omega(v)$ coincides with 
 $\omega(G[L(\MDT_G(v))])$.

 We show now that Algorithm \ref{alg:clique} can be implemented to run in $O(|V|+|E|)$ time 
 for a given  \gatex graph $G=(V,E)$. 
 The modular decomposition tree $(\MDT_G,t_G)$ can be computed in $O(|V|+|E|)$ time \cite{HP:10}. 
 By \cite[Thm.\ 9.4 and Alg.\ 4]{HS:22}, the pvr-network $(N,t)$ of $G$ can be computed within
 the same time complexity. Thus, Line \ref{l:MDT-pvr} takes $O(|V|+|E|)$ time. 
 Initializing  $\omega(v)\coloneqq 1$ for all leaves $v$ (and thus, the vertices of $G$) in 
 Line \ref{l:omega-leaves} can be done in $O(|V|)$ time. 
 
 We then traverse each of the $O(|V|)$ vertices in 
 $(\MDT_G,t_G)$ in postorder. 
 Note first that  the sides $P^1$ and $P^2$ of $C$ can be determined in $O(|V(C)|)$ time
  in Line \ref{l:sides}. 
Moreover, it is easy to verify that all other individual steps
 starting at Line \ref{l:v-prime} that are
 outside and within the two for-loop in Line \ref{l:base-case-start} and \ref{l:init-innerC-start} 
 take constant time. Each of these constant-time steps is computed 
 for every vertex $v\in V(C)$ once. Taken together the latter arguments, for a given prime vertex $v$, 
 Line \ref{l:v-prime} - \ref{l:end-v-prime} have runtime $O(|V(C)|)$. 
 Note that  each cycle $C$ has, by definition of pvr-networks, no vertex in common with 
 every other cycles. For $v\in V(MDT_G)$, let $\lambda_{G}(v)$ denote the size $O(|V_C|)$ of 
 the cycle $C_M$ associated with $M=L(\MDT_G(v))$ in case $M$ is a prime module. 
 Otherwise, i.e., if $M$ is not a prime module, we put  $\lambda_{G}(v)=1$. 
 Hence, the total runtime of Line \ref{l:forV} - \ref{l:forV-end} is 
 $\sum_{v\in V(MDT_G)} \lambda_{G}(v) = O(|V(N)|)$. 
 By \cite[Prop.\ 1]{CRV07}, we have $O(|V(N)|)=O(|V|)$. 
 Hence, the overall time-complexity of Algorithm \ref{alg:clique} is bounded 
 by the time-complexity to compute $(\MDT_G,t_G)$ and $(N,t)$ 
 in  Line \ref{l:MDT-pvr} and is, therefore, $O(|V|+|E|)$ time.
\end{proof}

Since \gatex graphs $G$ are perfect, their chromatic number $\chi(G)$ and the size $\omega(G)$
of a maximum clique coincide, we obtain

\begin{theorem}
The chromatic number $\chi(G)$ of every \gatex graph can be computed in linear-time. 
\end{theorem}

We consider now the problem of determining the independence number $\alpha(G)$
of \gatex graphs $G$, i.e., a maximum subset $W\subseteq V(G)$ of pairwise
non-adjacent vertices (also known as maximum independent set). 
Suppose that a \gatex graph $G$ is explained by the network $(N,t)$ and let 
$\overline t\colon V(N)\to \{0,1,\dot\}$ where 
$\overline t(v)=\odot$ for all leaves $v$ of $N$ and 
$\overline t(v)=1 $ if and only if $t(v)=0$. 
 By \cite[Prop.\ 1]{CRV07}, we have $O(|V(N)|)=O(|V(G)|)$ and thus, 
 this labeling can be computed in $O(|V(G)|)$ time. 
It is easy to verify that
$(N,\overline t)$ explains the complement $\overline G$ of $G$. 
By Obs.\ \ref{obs:complement-galled-tree}, the complement of every \gatex graph is a \gatex graph as well. 
Since maximum cliques in $\overline G$ are precisely the maximum independent sets
in $G$, we obtain

\begin{theorem}
The independence number $\alpha(G)$ of every \gatex graph can be computed in linear-time. 
\end{theorem}

\section{Concluding Remarks}

 In this contribution, we gave two novel characterizations of graphs that
 can be explained by labeled galled-trees (\gatex graphs) in terms of
 induced subgraphs. While Theorem \ref{thm:GatexIFFallPrimitive=PsC} shows
 that a graph is \gatex precisely if all induced primitive subgraphs are
 pseudo-cographs, Theorem \ref{thm:F-free} provides a characterization in terms
 of the set $\forbGT$ of 25 forbidden induced subgraphs. These
 characterizations allowed us to show that \gatex graphs are closely related to
 many other graph classes such as $P_4$-sparse graphs, weakly-chordal graphs,
 perfect graphs with perfect order, comparability and permutation graphs, murky
 graphs as well as interval graphs, Meyniel graphs or very strongly-perfect and
 brittle graphs. The classes of graphs considered here is by-far not exhaustive
 and it would be of interest to investigate the connection of \gatex graphs to
 other graph classes as well.

 A cluster $C(v)$ in a galled-tree $N$ is the set of all leaf-descendants of a
 vertex $v$ and the clustering system $\mathcal C(N)$ of $N$ is the collection of
 all the clusters in $N$. Just recently, clustering systems of galled-trees have
 been characterized \cite{HSS:22cluster}. For cographs $G$ there is a direct
 connection of the clusters in the cotree that explain $G$ and the strong
 modules of $G$ (cf.\ \cite[Sec.\ 3.2]{HSW:16}). In general, the modules of a
 \gatex graph $G$ are not necessarily clusters of the galled-tree that explains
 $G$. Therefore, it is of interest, to understand in more detail, to what extent
 the clusters of a labeled galled-tree that explains $G$ can be directly
 inferred from $G$. In addition, one may ask how the clustering systems of
 galled-trees are related to generalized notions of modules in the graph they
 explain, see e.g.\ \cite{HMEM:21,HMZ:20}.

   Generalizations of galled-trees to more ``refined'' labeled networks that can
 explain a given graph will be an interesting topic for future work to address
 the wide variety of different networks that can be found in phylogenomics
 \cite{HSS:22cluster,huson_rupp_scornavacca_2010}. Such refinements may be
 defined in terms of additional labels to explain edge-colored graphs
 \cite{ER1:90,ER2:90,BD98} or by resolving prime modules in the MDT by more
 ``complicated'' structures than simple cycles. 

\section*{Acknowledgments}
We thank Anna Lindeberg and Nicolas Wieseke for stimulating discussions and comments 
that helped to significantly improve the established results. 
We also would like to thank the organizers of the \textit{18.\ Herbstseminar
der Bioinformatik} hosted in Doubice, Czech Republic,
where the main ideas of this paper were drafted.

\section*{Appendix}

\begin{proof}[Proof of Lemma \ref{lem:PrimWellPC->F}]
 Let $H$ be a primitive graph that is not a pseudo-cograph. Since $H$ is
 primitive, we have $|V(H)|\geq 4$. We proceed now by induction on the number
 $k$ of vertices of $H$. As base-cases, we have $k\in \{4,5,\dots,8\}$ for which
 Obs.\ \ref{obs:exhaustive} implies that $H$ must contain one of the forbidden
 subgraphs in $\forbGT$. Let $|V (H)| \geq 9$ and assume that all primitive
 graphs that are not pseudo-cographs and have less than $|V(H)|$ vertices
 contain an $F\in\forbGT$ as an induced subgraph.

 If $H$ is critical-primitive, then Cor.\ \ref{cor:critical-primitive} implies
 that $H$ is not $\forbGT$-free. Thus, assume that $H$ is not
 critical-primitive. Therefore, there is a vertex $x$ such that $H'= H-x$ is
 primitive. Note that $|V(H')|\geq 8$.
	
 If $H'$ is a not a pseudo-cograph it must, by induction hypothesis, contain a
 forbidden subgraph and so $H$ is not $\forbGT$-free. 
	
 Thus, assume that $H'$ is a pseudo-cograph. Since $H'$ is primitive, Thm.\
 \ref{thm:CharPolCat} implies that $H'$ is a polar-cat. Moreover, since
 $|V(H')|\geq 8$, $H'$ must be well-proportioned, by Obs.\ \ref{obs:well-prop}.
 Hence, by Thm.\ \ref{thm:CharPolCat}, $H'$ has a unique representation as a
 pseudo-cograph $(v,H'_1,H'_2)$. W.l.o.g.\ assume that $H'$ is a slim
 pseudo-cograph (otherwise, take the complement of $H'$; cf.\ Thm.\
 \ref{thm:prop_pc}). Since $H'$ is a slim and primitive pseudo-cograph and
 $H'_1$, $H'_2$ and $v$ are uniquely determined, we can assume that
 $H_Y\coloneqq H[Y] = H'_1$ and $H_Z\coloneqq H[Z]=H'_2$ are induced subgraphs
 of $H-x$ such that $Y\cap Z=\{v\}$, $Y\cup Z = V(H')$ and where the vertices in
 $Y=\{y_1, \ldots, y_{\ell-1}, y_{\ell}=v\}$ and $Z=\{z_1, \ldots, z_{m-1},
 z_{m}=v\}$ are labeled according to Prop.~\ref{prop:polcat-edges} and thus,
 that the edges are of the form as specified in Prop.~\ref{prop:polcat-edges}.
 By definition, $H_Y$ and $H_Z$ are cographs. By assumption, $x$ is adjacent to
 $1,2,3, \dots |V(H')|$ vertices in $H'$ where
 $|V(H')|\geq 8$. For better readability, we denote by $\mathcal N(x)$
 the set of the vertices of $H'$ adjacent to $x$ in $G$. We distinguish now
 the following exhaustive cases: 
	(1) $\mathcal N(x) \subseteq Y$ or $\mathcal N(x) \subseteq Z$, 
	(2) $\mathcal N(x) \not\subseteq Y$, $\mathcal N(x) \not\subseteq Z$ and $v \in \mathcal N(x)$
	and (3) $\mathcal N(x) \not\subseteq Y$, $\mathcal N(x) \not\subseteq Z$ and $v \notin \mathcal N(x)$.

 In what follows, we will make frequent use of Prop~\ref{prop:polcat-edges}(a)
 and the additional argument that there are no edges between vertices in
 $Y-\{v\}$ and $Z-\{v\}$ in $H'$ since $H'$ is a slim pseudo-cograph without
 explicit reference. 

\medskip
\noindent
\emph{Case (1): $\mathcal N(x) \subseteq Y$ or $\mathcal N(x) \subseteq Z$.}

Without loss of generality, we can assume that $\mathcal N(x) \subseteq Y =
\{y_1,\dots,y_\ell=v\}$ and thus, $\mathcal N(x)\cap Z\in \{\emptyset,
\{v\}\}$. Note that $H_Y+x$ cannot be a cograph. To see this, assume for
contradiction that $H_Y+x$ is a cograph. In this case, the fact that $H'$ is a
slim pseudo-cograph and that $x$ is only adjacent to vertices in $Y$ implies
that $H$ is a $(v,H'_Y+x,H_Z)$-pseudo-cograph; a contradiction. Hence, $H_Y+x$
contains an induced $P_4$. Since $H_Y$ is a cograph, this induced $P_4$ must
contain $x$ and is, therefore, either of the form $x-y_i-y_j-y_k$ or
$y_i-x-y_j-y_k$.

\begin{itemize}
\item[(i)] Suppose first that $H$ contains a $P_4$ of the form $x-y_i-y_j-y_k$.
           In this case, $j$ must be odd, since otherwise, both $i$ and $k$ must
           be odd in which case $y_i$ and $y_k$ are adjacent in $H'$, either
           because $i<k$ or $k<i$. Since $j$ is odd and $y_j$ is adjacent to
           $y_i$ and $y_k$ it follows that $j<i,k$ We consider now the subcases
           to distinguish if $v$ is contained in $N(x)$ or not. 

\begin{itemize}
\item[(a)] 
	If $v \in \mathcal N(x)$ and $k$ is odd, then, $i$ must be even and $j < i <
	k$ holds. This and $k\leq \ell$ implies that $v=y_\ell\neq y_i,y_j$.
	Moreover, since $v\in \mathcal N(x)$ but $x$ and $y_k$ are not adjacent, we
	obtain $v\neq y_k$. In this case, $\{x,v,y_i,y_j,y_k,z_1\}$ induces an $F_9
	\sqsubseteq H$.

\item[(b)] 
	If $v \in \mathcal N(x)$ and $k$ is even, then $j < k\leq \ell$ and the fact
	that $j$ is odd while $k$ is even imply that $v,y_j,y_k$ induce a path on
	three vertices in $H$. Consequently, $H$ contains an induced $P_4$
	$x-v-y_j-y_k$. If $|Z| \geq 3$, then $\{x,v,y_j,y_k,z_1,z_2\}$ induces an
	$F_1 \sqsubseteq H$. Suppose now that $|Z|=2$. In this case, $|V(H')|\geq 8$
	and $Y\cap Z=\{v\}$ implies that $|Y| \geq 7$. Then there must exist a
	vertex $y_r \in \mathcal N(x)\setminus \{v\}$ for some $r<|Y|$ as,
	otherwise, $\{x,z_1\}$ would be a module of $H$, contradicting primitivity
	of $H$. If $y_2 \in \mathcal N(x)$, then $\{x,v,y_2, y_j,y_k, z_1\}$ induces
	an $F_1 \sqsubseteq H$ if $j>1$ and an $F_7 \sqsubseteq H$ if $j=1$. If $y_1
	\in \mathcal N(x)$ and $y_2 \notin \mathcal N(x)$, then $y_j\neq y_1$ (since
	$y_j \notin \mathcal N(x)$) and $\{x,v,y_1,y_2,y_j,y_k\}$ induces an $F_{21}
	\sqsubseteq H$. Suppose now that $y_1,y_2 \notin \mathcal N(x)$. In this
	case, $v,y_1,y_2$ induce a path on three vertices in $H$ and thus, $H$
	contains an induced $P_4$ $x-v-y_1-y_2$. If there exists a vertex $y_r \in
	\mathcal N(x)$ for some even $r$, then $\{x,v,y_1,y_2,y_r,z_1\}$ induces an
	$F_7 \sqsubseteq H$. Otherwise, if $y_3 \in \mathcal N(x)$, then
	$\{x,v,y_1,y_2,y_3,y_4,z_1\}$ induces an $F_{19} \sqsubseteq H$. If neither
	of the previous two cases hold then $r\neq 3$ is odd. In particular, $r\neq
	1$ since $y_1 \notin\mathcal N(x)$. In this case,
	$\{x,v,y_1,y_2,y_3,y_4,y_r,z_1\}$ induces an $F_{23} \sqsubseteq H$.

\item[(c)] 
	If $v \notin \mathcal N(x)$ and $i$ is odd, then $\{y_i,y_k\}\notin E(H')$
	implies that $k<i$ and $k$ is even. This and together with $j<i,k$ implies
	$j < k < i$. By similar arguments as in Case (i.a), $v \neq y_i,y_j,y_k$. In
	this case, $\{x,v,y_i,y_j,y_k,z_1\}$ induces an $F_5 \sqsubseteq H$.

\item[(d)] 
	If $v \notin \mathcal N(x)$ and $i$ is even, then $j < i$. By similar
	arguments as in Case (i.b), $H$ contains an induced $P_4$ $x-y_i-y_j-v$. If
	$|Z| \geq 3$, then $\{x,v,y_i,y_j,z_1,z_2\}$ induces an $F_3 \sqsubseteq H$.
	Suppose now that $|Z|=2$ and, therefore, $|Y| \geq 7$. In particular, $Z =
	\{z_1,z_2=v\}$. Suppose first that $y_1 \in \mathcal N(x)$. Since $y_j\notin
	\mathcal N(x)$, we have $y_1\neq y_j$ and $\{x,v,y_1,y_i,y_j,z_1\}$ induces
	an $F_{10} \sqsubseteq H$. If $y_1 \notin \mathcal N(x)$, then $H$ contains
	an induced $P_4$ $x-y_i-y_1-v$ In that case, if there exists an even $r$
	such that $y_r \notin \mathcal N(x)$, then $\{x,v,y_1,y_i,y_r,z_1\}$ induces
	an $F_1 \sqsubseteq H$. If otherwise, $y_r \in \mathcal N(x)$ for all $r$
	even, then if $y_3\notin \mathcal N(x)$, $\{x,v,y_1,y_2,y_3,y_4\}$ induces
	an $F_{12} \sqsubseteq H$, and if $y_3 \in \mathcal N(x)$,
	$\{x,v,y_1,y_3,y_4,z_1\}$ induces an $F_{10} \sqsubseteq H$
\end{itemize}

\item[(ii)] Suppose now that $H$ contains a $P_4$ of the form $y_i-x-y_j-y_k$.
            Assume first, for contradiction, that $i$ is odd. In this case,
            $i>j,k$ must hold, since $y_i$ is neither adjacent to $y_j$ nor to
            $y_k$. However, since $y_j$ and $y_k$ are adjacent, at least one of
            $i$ and $k$ must be odd which together with $i>j,k$ implies that
            $y_i$ is adjacent to $y_j$ or $y_k$; a contradiction. Hence, $i$
            must be even.

\begin{itemize}
\item[(a)] 
	If $v \in \mathcal N(x)$ and $k$ is even, then $j$ must be odd and $i < j <
	k$. The latter and $k\leq \ell$ implies that $v=y_\ell\neq y_i,y_j$.
	Moreover, $v \in \mathcal N(x)$ and $y_k \notin \mathcal N(x)$ implies
	$v\neq y_k$. In this case, $\{x,v,y_i,y_j,y_k,z_1\}$ induces an $F_5
	\sqsubseteq H$.

\item[(b)] 
	If $v \in \mathcal N(x)$ and $k$ is odd, then $i <k$. As in case (ii.a),
	$v\neq y_k$. Since $v=y_\ell$, there is an induced $P_4$ $y_i-x-v-y_k$. If
	$|Z| \geq 3$, then $\{x,v,y_i,y_k,z_1,z_2\}$ induces an $F_1 \sqsubseteq H$.
	Suppose now that $|Z|=2$ in which case $|Y| \geq 7$ must hold. To recap,
	$v=y_\ell\neq y_1$.
	Note that $1\leq i <k$ implies that $y_1\neq y_k$ and, since $i$ is
	even, $y_1\neq y_i$. If $y_1 \notin \mathcal N(x)$, then
	$\{x,v,y_1,y_i,y_k,z_1\}$ induces an $F_9 \sqsubseteq H$. If $y_1 \in
	\mathcal N(x)$, $y_2 \notin \mathcal N(x)$, then $y_2 \neq y_i, y_k$, since
	$y_i \in \mathcal N(x)$ and $k$ is odd, and $\{x,v,y_1,y_2,y_i,y_k,z_1\}$
	induces an $F_{21} \sqsubseteq H$. If $y_1, y_2 \in \mathcal N(x)$, then
	there is, in particular, an induced $P_4$ $y_2-x-v-y_k$ in $H$. In this
	case, if $y_3 \in \mathcal N(x)$, then $\{x,v,y_1,y_2,y_3,y_k,z_1\}$ induces
	an $F_{22} \sqsubseteq H$, and if $y_3 \notin \mathcal N(x)$, then
	$\{x,v,y_2,y_3,y_4,z_1\}$ induces an $F_7 \sqsubseteq H$ if $y_4 \in
	\mathcal N(x)$, and an $F_1 \sqsubseteq H$ if $y_4 \notin \mathcal N(x)$.

\item[(c)] 
	If $v \notin \mathcal N(x)$ and $j$ is even, then $k$ must be odd and, in
	particular, $i < k < j$ and $v \neq y_i,y_j,y_k$ holds by similar arguments
	as in Case (i.a). In this case, $\{x,v,y_i,y_j,y_k,z_1\}$ induces an $F_3
	\sqsubseteq H$.

\item[(d)] 
	If $v \notin \mathcal N(x)$ and $j$ is odd, then $i < j$ and, in particular,
	$j\neq 1$. In this case, there is an induced $P_4$ $y_i-x-y_j-v$ in $H$.
	Since $j\neq 1$, we have $y_j\neq y_1$. Moreover, $v=y_\ell \neq y_1$ holds.
	Therefore, $\{x,v,y_1,y_i,y_j,z_1\}$ induces an $F_{16} \sqsubseteq H$ in
	case $y_1 \notin \mathcal N(x)$, and a $F_{10} \sqsubseteq H$. in case $y_1
	\in \mathcal N(x)$.
\end{itemize}
\end{itemize}

\medskip
\noindent
\emph{Case (2): Case (1) does not hold and $v \in \mathcal N(x)$.} 
	In this case, we have $\mathcal N(x) \cap Y\neq \emptyset$ and $\mathcal
	N(x) \cap Z \neq \emptyset$ since $x$ is adjacent to at least one vertex in
	$Y$ or $Z$ and since Case (1) does not hold.

\begin{itemize}
\item[(i)] Suppose first that one of $|Z|=2$ or $|Y|=2$ holds. W.l.o.g.\ assume
           that $|Y|=2$ in which case $|Z| \geq 7$ must hold. Moreover, since
           $Y=\{y_1,y_2=v=z_m\}$, we must have $y_1 \in \mathcal N(x)$ as,
           otherwise, $\mathcal N(x) \subseteq Z$ and we are in the situation of
           Case (1).

\begin{itemize}
\item[(a)] 
	Suppose that $z_1 \in \mathcal N(x)$. Note, that $\mathcal
	N(x)\setminus\{v,x\} \neq \mathcal{N}(v)\setminus \{x,v\}$ must hold as,
	otherwise, $\{x,v\}$ would be a non-trivial module of $H$, contradicting the
	primitivity of $H$. Since $\mathcal{N}(v)\setminus \{x,v\} = \{y_1\} \cup
	\{z_i\mid 1\leq i<\ell, i \text{ odd}\}$ and $y_1\in \mathcal
	N(x)\setminus\{v,x\}$ and $v=z_m\notin N(x)\setminus\{v,x\}$, there must be
	a vertex $z_i\in \mathcal N(x)$ for some even $i< \ell$ or a vertex
	$z_i\notin \mathcal N(x)$ for some odd $i<\ell$. In particular, such a $z_i$
	satisfies $z_i\neq v$. Let $i$ be the smallest index among all indices in
	$\{1,\dots,m-1\}$ that satisfy $i$ is even and $z_i \in \mathcal N(x)$, or
	$i$ is odd and $z_i \notin \mathcal N(x)$. Assume first that $i$ is odd. In
	this case, $z_1 \in \mathcal N(x)$ implies $i \geq 3$. If $i=3$, then
	$\{x,v,y_1,z_1,z_3,z_4\}$ induces an $F_{8} \sqsubseteq H$ in case $z_4
	\notin \mathcal N(x)$, and an $F_2 \sqsubseteq H$ in case $z_4 \in \mathcal
	N(x)$. If $i>3$ and thus, $z_2\notin \mathcal N(x)$ and $z_3\in \mathcal
	N(x)$, then $\{x,v,y_1,z_1,z_2,z_3,z_i\}$ induces an $F_{22} \sqsubseteq H$.
	Assume now that $i$ is even. In this case, $z_i \in \mathcal N(x)$. If
	$i>2$, then $z_3 \in \mathcal N(x)$ by choice of $i$, and
	$\{x,v,y_1,z_1,z_2,z_3,z_i\}$ induces an $F_{22} \sqsubseteq H$. If $i=2$,
	then let $j \geq 3$ be the smallest index such that $z_j \notin \mathcal
	N(x)$. Note that such an element $j$ must exist, otherwise, $\mathcal
	N(x)=V(H')$ and $V(H')$ would be a non-trivial module of $H$; contradicting
	primitivity of $H$. If $j$ is odd, then $\{x,v,y_1,z_1,z_2,z_j\}$ induces a
	$F_6 \sqsubseteq H$. If $j$ is even, then $j>3$ holds, so $z_3 \in \mathcal
	N(x)$ and $\{x,v,y_1,z_1,z_2,z_3,z_j\}$ induces a $F_{20} \sqsubseteq H$.

\item[(b)] 
	Suppose that $z_1 \notin \mathcal N(x)$. Since $y_1,v\in \mathcal N(x)$ and
	$v=z_m=y_2$, there must be an index $i\in\{1,\dots, \ell-1\}$ such that $z_i
	\in \mathcal N(x)$ as, otherwise, $\mathcal N(x) \subseteq Y$ and we are in
	the situation of Case (1). Let $i$ be the smallest index in $\{1,\dots,
	\ell-1\}$ for which $z_i \in \mathcal N(x)$ holds. Note, $z_i\neq v$. If $i$
	is odd, then $z_2 \notin \mathcal N(x)$, and $\{x,v,y_1,z_1,z_2,z_i\}$
	induces an $F_{10} \sqsubseteq H$. If $i$ is even and $i>2$, then $z_2
	\notin \mathcal N(x)$, and $\{x,v,y_1,z_1,z_2,z_i\}$ induces an $F_{11}
	\sqsubseteq H$. Assume now that $i=2$. If there exists a $j \geq 3$ such
	that $z_j \notin \mathcal N(x)$, then $z_j\neq v$ and
	$\{x,v,y_1,z_1,z_2,z_j\}$ induces an $F_{15} \sqsubseteq H$ if $j$ is odd,
	and an $F_{11} \sqsubseteq H$ if $j$ is even. Finally, if $z_j \in \mathcal
	N(x)$ for all $j \geq 2$, then $\{x,v,y_1,z_1,z_3,z_4\}$ induces an $F_2
	\sqsubseteq H$.
\end{itemize}

\item[(ii)] Suppose now that $|Y|, |Z| \geq 3$. Put $S=\{v,y_1,y_2,z_1,z_2\}$. 
			Since $v\in \mathcal N(x)$, we consider now the $2^4=16$ possible cases 
			that $S \cap \mathcal N(x)$ may satisfy.

\begin{itemize}
\item[(a)] Case $S \cap \mathcal N(x)=\{v\}$. Then $S \cup \{x\}$ induces an
           $F_1 \sqsubseteq H$.

\smallskip
\item[(b)] Case $S \cap \mathcal N(x)=\{v,y_1\}$ or $S \cap \mathcal
           N(x)=\{v,z_1\}$. \smallskip

 W.l.o.g.\ assume that $S \cap \mathcal N(x)=\{v,y_1\}$. In this case, there
 must be an index $i$, $3 \leq i \leq |Z|-1$ such that $z_i \in \mathcal N(x)$,
 since otherwise, we have $\mathcal{N}(x)\subseteq Y$ and we are in the
 situation of Case (1). In this case, $\{x,v,y_1,z_1,z_2, z_i\}$ induces an
 $F_{11} \sqsubseteq H$ if $i$ is even, and an $F_{10} \sqsubseteq H$ if $i$ is
 odd.

\smallskip
\item[(c)] Case $S \cap \mathcal N(x)=\{v,y_2\}$ or $S \cap \mathcal
           N(x)=\{v,z_2\}$. \smallskip

 W.l.o.g.\ assume that $S \cap \mathcal N(x)=\{v,y_2\}$. In this case, there
 must be an index $i$, $3 \leq i \leq |Z|-1$ such that $z_i \in \mathcal N(x)$,
 since otherwise, we have $\mathcal{N}(x)\subseteq Y$ and we are in the
 situation of Case (1). In this case, $\{x,v,y_1,y_2,z_1, z_i\}$ induces an
 $F_{17} \sqsubseteq H$ if $i$ is even, and an $F_{12} \sqsubseteq H$ if $i$ is
 odd.

\smallskip
\item[(d)] Case $S \cap \mathcal N(x)=\{v,y_2,z_2\}$, then $S \cup \{x\}$
           induces an $F_{17} \sqsubseteq H$.

\smallskip
\item[(e)] Case $S \cap \mathcal N(x)=\{v,y_1,z_2\}$ or $S \cap \mathcal
           N(x)=\{v,y_2,z_1\}$, then $S \cup \{x\}$ induces an $F_{16}
           \sqsubseteq H$.

\smallskip
\item[(f)] Case $S \cap \mathcal N(x)=\{v,y_1,y_2\}$ or $S \cap \mathcal
           N(x)=\{v,z_1,z_2\}$. \smallskip

 W.l.o.g.\ assume that $S \cap \mathcal N(x)=\{v,y_1,y_2\}$. In this case, there
 must be an index $i$, $3 \leq i \leq |Z|-1$ such that $z_i \in \mathcal N(x)$,
 since otherwise, we have $\mathcal{N}(x)\subseteq Y$ and we are in the
 situation of case (1). In this case, $\{x,v,y_1,z_1,z_2, z_i\}$ induces an
 $F_{11} \sqsubseteq H$ if $i$ is even, and an $F_{10} \sqsubseteq H$ if $i$ is
 odd. 

\smallskip
\item[(g)] Case  $S \cap \mathcal N(x)=\{v,y_1,z_1\}$. \smallskip

 Since $\{x,v\}$ cannot form a module in $H$, we can reuse the arguments as in
 Case (2.i.a) and conclude that there must be a vertex $z_i\in \mathcal N(x)$
 for some even $i< m$ or a vertex $y_i\in \mathcal N(x)$ for some even $i< \ell$
 or a vertex $z_i\notin \mathcal N(x)$ for some odd $i<m$ or $y_i\notin \mathcal
 N(x)$ for some odd $i<\ell$. If $y_i\in \mathcal N(x)$ and $i<\ell$ is even or
 if $y_i\notin \mathcal N(x)$ and $i<\ell$ is odd, then
 $\{x,v,y_1,y_i,z_1,z_2\}$ induces an $F_{10} \sqsubseteq H$. By similar
 arguments, one finds an $F_{10} \sqsubseteq H$ in case $z_i\in \mathcal N(x)$
 and $i< m$ even or $z_i\notin \mathcal N(x)$ and $i<m$ odd.

\smallskip
\item[(h)] Case $S \cap \mathcal N(x)=\{v,y_1,y_2,z_2\}$ or $S \cap \mathcal
           N(x)=\{y_2,z_1, z_2\}$. Then, $S \cup \{x\}$ induces an $F_{12}
           \sqsubseteq H$.

\smallskip
\item[(i)] Case $S \cap \mathcal N(x)=\{v,y_1,y_2,z_1\}$ or $S \cap \mathcal
           N(x)=\{v,y_1,z_1,z_2\}$. Then, $S \cup \{x\}$ induces an $F_{10}
           \sqsubseteq H$.

\smallskip
\item[(j)] Case $S \cap \mathcal N(x) = S$. \smallskip

	If $\mathcal N(x)=V(H)\setminus\{x\}=V(H')$, then $\mathcal N(x)$ would form
	a non-trivial module in $H$, which is not possible since $H$ is primitive.
	Hence, there must be at least one vertex $w \in V(H')$ such that $w \notin
	\mathcal N(x)$. Without loss of generality, we may assume that $w=y_i$ for
	some $3 \leq i \leq |Y|-1$. Then, $\{x,v,y_1,y_i,z_1,z_2\}$ induces an
	$F_{10} \sqsubseteq H$ if $i$ is even, and an $F_8 \sqsubseteq H$ if $i$ is
	odd. 	
\end{itemize}
\end{itemize}

\medskip
\noindent
\emph{Case (3): Case (1) does not hold and $v \notin \mathcal N(x)$.}

\begin{itemize}
\item[(i)]  Suppose first that one of $|Z|=2$ or $|Y|=2$ holds. W.l.o.g.\ assume that $|Y|=2$. 
			In this case, $|Z| \geq 7$ must hold. As argued in Case (2.i), we have $y_1 \in \mathcal N(x)$.
			Note that $v=y_2=z_m$.

\begin{itemize}
\item[(a)] 
	If $z_1 \in \mathcal N(x)$, we have to distinguish between three cases.
	Suppose first that $z_3 \notin \mathcal N(x)$. In this case,
	$\{x,v,y_1,z_1,z_2,z_3\}$ induces an $F_{15} \sqsubseteq H$ if $z_2 \in
	\mathcal N(x)$, and a $F_9 \sqsubseteq H$ if $z_2 \notin \mathcal N(x)$.
	Thus, assume that $z_3 \in \mathcal N(x)$. If $z_2\in \mathcal N(x)$, then
	$\{x,v,y_1,z_2,z_3,z_4\}$ induces an $F_9 \sqsubseteq H$ if $z_4 \in
	\mathcal N(x)$, and a $F_7 \sqsubseteq H$ if $z_4 \notin \mathcal N(x)$.
	Suppose that $z_2\notin \mathcal N(v)$. Since $\{x,v\}$ cannot form a module
	in $H'$ and, in addition, $v=z_m\notin \mathcal N(x)$ and $y_1,z_1,z_3 \in
	\mathcal N(x)\cap \mathcal N(v)$, there must be a vertex $z_i$ with $i\in
	\{4,5,\dots,m-1\}$ such that $z_i\notin \mathcal N(v)$ but $z_i\in \mathcal
	N(x)$ or $z_i\in \mathcal N(v)$ but $z_i\notin \mathcal N(x)$. In either
	case, $\{x,v,y_1,z_1,z_2,z_i\}$ induces an $F_9 \sqsubseteq H$.

\item[(b)] 
	If $z_1 \notin \mathcal N(x)$, then there exists a vertex $z_i \in \mathcal
	N(x)$ for some $i \in \{2,\dots,m-1\}$, otherwise we are again in situation
	of Case (1). If there exists such an even index $i$ with $z_i \in \mathcal
	N(x)$, then $\{x,v,y_1,z_1,z_i\}$ induces an $F_0 \sqsubseteq H$. If
	otherwise, $z_j \notin \mathcal N(x)$ for all $j$ even, then $z_2 \notin
	\mathcal N(x)$ and there is an odd index $i$ with $z_i \in \mathcal N(x)$
	and $z_i\neq v$. In this case, $\{x,v,y_1,z_1,z_2,z_i\}$ induces an $F_{16}
	\sqsubseteq H$.
\end{itemize}

\item[(ii)] Suppose now that $|Y|, |Z| \geq 3$.

\begin{itemize}
\item[(a)] 
	Suppose that $y_1,z_1 \in \mathcal N(x)$. In this case,
	$\{x,v,y_1,y_2,z_1,z_2\}$ induces an $F_{15} \sqsubseteq H$ if $y_2,z_2 \in
	\mathcal N(x)$, and an $F_{11} \sqsubseteq H$ if exactly one of $y_2,z_2$ is
	in $\mathcal N(x)$. Assume now that $y_2, z_2 \notin \mathcal N(x)$. Similar
	as in Case (2.i.a) we can conclude that there must be a vertex $z_i\in
	\mathcal N(x)$ for some even $i< m$ or a vertex $y_i\in \mathcal N(x)$ for
	some even $i< \ell$ or a vertex $z_i\notin \mathcal N(x)$ for some odd $i<m$
	or $y_i\notin \mathcal N(x)$ for some odd $i<\ell$. If there is an odd index
	$i$ and $y_i \notin \mathcal N(x)$, or an even index $i$ and $y_i \in
	\mathcal N(x)$, then $\{x,v,y_1,y_i,z_1,z_2\}$ induces an $F_{11}
	\sqsubseteq H$. By similar reasoning, there is an induced $F_{11}$ in the
	case for such an index $i$ and vertex $z_i$.

\item[(b)] 
	Suppose that exactly one of $y_1,z_1$ is in $\mathcal N(x)$. W.l.o.g.\
	assume that $y_1\in \mathcal N(x)$ and, thus, $z_1 \notin \mathcal
	N(x)$. Then, there exists an index $i$ with $2 \leq i \leq |Z|-1$ such
	that $z_i \in \mathcal N(x)$, otherwise we are in the situation of Case (1).
	If there exists such an even index $i$, then $\{x,v,y_1,z_1,z_i\}$ induces
	an $F_0 \sqsubseteq H$. If otherwise, $z_j \notin \mathcal N(x)$ for all $j$
	even, then $z_2 \notin \mathcal N(x)$ and there is an odd index $i\neq
	1,|Z|$ with $z_i \in \mathcal N(x)$. In this case, 
	$\{x,v,y_1,z_1,z_2,z_i\}$ induces an $F_{16} \sqsubseteq H$.

\item[(c)] 
	Suppose that $y_1,z_1 \notin \mathcal N(x)$. Then there must exists $2 \leq
	i \leq |Y|-1$ such that $y_i \in \mathcal N(x)$, otherwise we are in the
	situation of case (1). If there exists such an even index $i$, then
	$\{x,v,y_1,y_i,z_1,z_2\}$ induces an $F_3 \sqsubseteq H$ if $z_2 \notin
	\mathcal N(x)$, and an $F_{13} \sqsubseteq H$ if $z_2 \in \mathcal N(x)$. If
	otherwise, $y_j \notin \mathcal N(x)$ for all $j$ even, then there exists an
	odd index $i\neq 1$ such that $y_i \in \mathcal N(x)$,
	and $\{x,v,y_1,y_2,y_i,z_1\}$ induces an $F_5 \sqsubseteq H$.
\end{itemize}
\end{itemize}
\end{proof}

\bibliographystyle{plain}
\bibliography{pc}

\begin{thebibliography}{10}

\bibitem{BD98}
Sebastian B{\"o}cker and Andreas W.~M. Dress.
\newblock Recovering symbolically dated, rooted trees from symbolic
  ultrametrics.
\newblock {\em Adv. Math.}, 138:105--125, 1998.

\bibitem{bonnet_et_al:LIPIcs.ICALP.2021.35}
\'{E}douard Bonnet, Colin Geniet, Eun~Jung Kim, St\'{e}phan Thomass\'{e}, and
  R\'{e}mi Watrigant.
\newblock {Twin-width III: Max Independent Set, Min Dominating Set, and
  Coloring}.
\newblock In Nikhil Bansal, Emanuela Merelli, and James Worrell, editors, {\em
  48th International Colloquium on Automata, Languages, and Programming (ICALP
  2021)}, volume 198 of {\em Leibniz International Proceedings in Informatics
  (LIPIcs)}, pages 35:1--35:20, Dagstuhl, Germany, 2021. Schloss Dagstuhl --
  Leibniz-Zentrum f{\"u}r Informatik.

\bibitem{BKTW:21}
\'{E}douard Bonnet, Eun~Jung Kim, St\'{e}phan Thomass\'{e}, and R\'{e}mi
  Watrigant.
\newblock Twin-width i: Tractable fo model checking.
\newblock {\em J. ACM}, 69(1), 2021.

\bibitem{Brandstadt1999}
Andreas Brandst{\"a}dt, Van~Bang Le, and Jeremy~P Spinrad.
\newblock {\em Graph Classes: A Survey}.
\newblock Society for Industrial and Applied Mathematics, Philadelphia, PA,
  USA, 1999.

\bibitem{CRV07}
G.~Cardona, F.~Rossell\'o, and G.~Valiente.
\newblock Comparison of tree-child phylogenetic networks.
\newblock {\em {IEEE}/{ACM} {T}ransactions on {C}omputational {B}iology and
  {B}ioinformatics}, 6:552--569, 2007.

\bibitem{CHVATAL1987349}
V~Chvátal.
\newblock On the p4-structure of perfect graphs iii. partner decompositions.
\newblock {\em Journal of Combinatorial Theory, Series B}, 43(3):349--353,
  1987.

\bibitem{CHVATAL1977145}
Václav Chvátal and Peter~L. Hammer.
\newblock Aggregation of inequalities in integer programming.
\newblock In P.L. Hammer, E.L. Johnson, B.H. Korte, and G.L. Nemhauser,
  editors, {\em Studies in Integer Programming}, volume~1 of {\em Annals of
  Discrete Mathematics}, pages 145--162. Elsevier, 1977.

\bibitem{Corneil:81}
D.~G. Corneil, H.~Lerchs, and L~K Stewart~Burlingham.
\newblock Complement reducible graphs.
\newblock {\em Discr. Appl. Math.}, 3:163--174, 1981.

\bibitem{DGC:01}
Elias Dahlhaus, Jens Gustedt, and Ross~M McConnell.
\newblock Efficient and practical algorithms for sequential modular
  decomposition.
\newblock {\em Journal of Algorithms}, 41(2):360 -- 387, 2001.

\bibitem{DiS:12}
Gabriele {Di Stefano}.
\newblock Distance-hereditary comparability graphs.
\newblock {\em Discrete Applied Mathematics}, 160(18):2669--2680, 2012.
\newblock V Latin American Algorithms, Graphs, and Optimization Symposium —
  Gramado, Brazil, 2009.

\bibitem{Golumbic-book:04}
Martin Charles~Golumbic (Eds.).
\newblock {\em Algorithmic Graph Theory and Perfect Graphs}.
\newblock Annals of Discrete Mathematics 57. North Holland, 2 edition, 2004.

\bibitem{ER1:90}
A~Ehrenfeucht and G~Rozenberg.
\newblock Theory of 2-structures, part {I}: Clans, basic subclasses, and
  morphisms.
\newblock {\em Theor. Comp. Sci.}, 70:277--303, 1990.

\bibitem{ER2:90}
A~Ehrenfeucht and G~Rozenberg.
\newblock Theory of 2-structures, part {II}: Representation through labeled
  tree families.
\newblock {\em Theor. Comp. Sci.}, 70:305--342, 1990.

\bibitem{EHMS:94}
Andrzej Ehrenfeucht, Harold~N. Gabow, Ross~M. Mcconnell, and Stephen~J.
  Sullivan.
\newblock {An O($n^2$) Divide-and-Conquer Algorithm for the Prime Tree
  Decomposition of Two-Structures and Modular Decomposition of Graphs}.
\newblock {\em Journal of Algorithms}, 16(2):283--294, 1994.

\bibitem{gallai1967transitiv}
Tibor Gallai.
\newblock Transitiv orientierbare graphen.
\newblock {\em Acta Mathematica Hungarica}, 18(1-2):25--66, 1967.

\bibitem{GBP12}
P.~Gambette, V.~Berry, and C.~Paul.
\newblock Quartets and unrooted phylogenetic networks.
\newblock {\em Journal of Bioinformatics and Computational Biology},
  10(4):1250004.1--1250004.23, 2012.

\bibitem{gambette2017challenge}
Philippe Gambette, KT~Huber, and Steven Kelk.
\newblock On the challenge of reconstructing level-1 phylogenetic networks from
  triplets and clusters.
\newblock {\em Journal of mathematical biology}, 74(7):1729--1751, 2017.

\bibitem{GAO20132763}
Yong Gao, Donovan~R. Hare, and James Nastos.
\newblock The cluster deletion problem for cographs.
\newblock {\em Discrete Mathematics}, 313(23):2763--2771, 2013.

\bibitem{garey1979computers}
Michael~R Garey and David~S Johnson.
\newblock {\em Computers and intractability}, volume 174.
\newblock freeman San Francisco, 1979.

\bibitem{Geiss:18a}
Manuela Gei{\ss}, John Anders, Peter~F. Stadler, Nicolas Wieseke, and Marc
  Hellmuth.
\newblock Reconstructing gene trees from {Fitch}'s xenology relation.
\newblock {\em J. Math. Biol.}, 77:1459--1491, 2018.

\bibitem{geiss2020best}
Manuela Gei{\ss}, Marcos E~Gonz{\'a}lez Laffitte, Alitzel~L{\'o}pez
  S{\'a}nchez, Dulce~I Valdivia, Marc Hellmuth, Maribel~Hern{\'a}ndez Rosales,
  and Peter~F Stadler.
\newblock Best match graphs and reconciliation of gene trees with species
  trees.
\newblock {\em Journal of mathematical biology}, 80(5):1459--1495, 2020.

\bibitem{GV:97}
Vassilis Giakoumakis and Jean-Marie Vanherpe.
\newblock On extended p4-reducible and extended p4-sparse graphs.
\newblock {\em Theoretical Computer Science}, 180(1):269--286, 1997.

\bibitem{GH:64}
P.~C. Gilmore and A.~J. Hoffman.
\newblock A characterization of comparability graphs and of interval graphs.
\newblock {\em Canadian Journal of Mathematics}, 16:539–548, 1964.

\bibitem{grotschel2012geometric}
Martin Gr{\"o}tschel, L{\'a}szl{\'o} Lov{\'a}sz, and Alexander Schrijver.
\newblock {\em Geometric algorithms and combinatorial optimization}, volume~2.
\newblock Springer Science \& Business Media, 2012.

\bibitem{Gusfield:03}
Dan Gusfield, Satish Eddhu, and Charles Langley.
\newblock Efficient reconstruction of phylogenetic networks with constrained
  recombination.
\newblock In {\em CSB '03: Proceedings of the IEEE Computer Society Conference
  on Bioinformatics}, pages 363--374, Washington DC, US, 2003. IEEE Computer
  Society.

\bibitem{HP:10}
M.~Habib and C.~Paul.
\newblock A survey of the algorithmic aspects of modular decomposition.
\newblock {\em Computer Science Review}, 4(1):41 -- 59, 2010.

\bibitem{HMEM:21}
Michel Habib, Lalla Mouatadid, {\'{E}}ric Sopena, and Mengchuan Zou.
\newblock ({\(\alpha\)}, {\(\beta\)})-modules in graphs.
\newblock {\em CoRR}, abs/2101.08881, 2021.

\bibitem{HMZ:20}
Michel Habib, Lalla Mouatadid, and Mengchuan Zou.
\newblock Approximating modular decomposition is hard.
\newblock In {\em Algorithms and Discrete Applied Mathematics}, pages 53--66,
  Cham, 2020. Springer International Publishing.

\bibitem{HM:90}
Peter~L. Hammer and Frédéric Maffray.
\newblock Completely separable graphs.
\newblock {\em Discrete Applied Mathematics}, 27(1):85--99, 1990.

\bibitem{hayward1985weakly}
Ryan~B Hayward.
\newblock Weakly triangulated graphs.
\newblock {\em Journal of Combinatorial Theory, Series B}, 39(3):200--208,
  1985.

\bibitem{HAYWARD:90}
Ryan~B Hayward.
\newblock Murky graphs.
\newblock {\em Journal of Combinatorial Theory, Series B}, 49(2):200--235,
  1990.

\bibitem{HHH+13}
M.~Hellmuth, M.~Hernandez-Rosales, K.~T. Huber, V.~Moulton, P.~F. Stadler, and
  N.~Wieseke.
\newblock Orthology relations, symbolic ultrametrics, and cographs.
\newblock {\em J. Math. Biology}, 66(1-2):399--420, 2013.

\bibitem{HW:16}
M.~Hellmuth and N.~Wieseke.
\newblock {\em From Sequence Data Including Orthologs, Paralogs, and Xenologs
  to Gene and Species Trees}, chapter~21, pages 373--392.
\newblock Springer, Cham, 2016.

\bibitem{github-MH}
Marc Hellmuth.
\newblock \url{https://github.com/marc-hellmuth/ForbiddenSubgraphs-GaTEx}
  (accessed Nov 3, 2022).

\bibitem{Hellmuth:17}
Marc Hellmuth.
\newblock Biologically feasible gene trees, reconciliation maps and informative
  triples.
\newblock {\em Alg Mol Biol}, 12:23, 2017.

\bibitem{HSS:22cluster}
Marc Hellmuth, David Schaller, and Peter~F. Stadler.
\newblock Clustering systems of phylogenetic networks.
\newblock {\em Theory in Biosciences}, 2022.
\newblock (to appear) arXiv:2204.13466v1 [q-bio.PE].

\bibitem{HS:22}
Marc Hellmuth and Guillaume~E. Scholz.
\newblock From modular decomposition trees to level-1 networks:
  Pseudo-cographs, polar-cats and prime polar-cats.
\newblock {\em Discrete Applied Mathematics}, 321:179--219, 2022.

\bibitem{HSW:16}
Marc Hellmuth, Peter~F Stadler, and Nicolas Wieseke.
\newblock The mathematics of xenology: Di-cographs, symbolic ultrametrics,
  2-structures and tree-representable systems of binary relations.
\newblock {\em Journal of Mathematical Biology}, 75(1):199--237, 2017.

\bibitem{Hellmuth:15a}
Marc Hellmuth, Nicolas Wieseke, Marcus Lechner, Hans-Peter Lenhof, Martin
  Middendorf, and Peter~F. Stadler.
\newblock Phylogenomics with paralogs.
\newblock {\em Proceedings of the National Academy of Sciences},
  112(7):2058--2063, 2015.

\bibitem{hoang1985perfect}
CT~Ho{\`a}ng.
\newblock Perfect graphs (ph. d. thesis), 1985.

\bibitem{HK:88}
C.~T. Hoàng and N.~Khouzam.
\newblock On brittle graphs.
\newblock {\em Journal of Graph Theory}, 12(3):391--404, 1988.

\bibitem{HOANG1987302}
C.T Hoàng.
\newblock On a conjecture of meyniel.
\newblock {\em Journal of Combinatorial Theory, Series B}, 42(3):302--312,
  1987.

\bibitem{HS18}
K.~T. Huber and G.~E. Scholz.
\newblock Beyond representing orthology relations with trees.
\newblock {\em Algorithmica}, 80(1):73--103, 2018.

\bibitem{huson_rupp_scornavacca_2010}
Daniel~H. Huson, Regula Rupp, and Celine Scornavacca.
\newblock {\em Phylogenetic Networks: Concepts, Algorithms and Applications}.
\newblock Cambridge University Press, Cambridge, UK, 2010.

\bibitem{JO:89}
B.~Jamison and S.~Olariu.
\newblock P4-reducible graphs, a class of uniquely tree-representable graphs.
\newblock {\em Studies in Applied Mathematics}, 81(1):79--87, 1989.

\bibitem{JO:91}
B.~Jamison and S.~Olariu.
\newblock On a unique tree representation for p4-extendible graphs.
\newblock {\em Discrete Applied Mathematics}, 34(1):151--164, 1991.

\bibitem{JO:92}
B.~Jamison and S.~Olariu.
\newblock A tree representation for p4-sparse graphs.
\newblock {\em Discrete Applied Mathematics}, 35(2):115--129, 1992.

\bibitem{JO:93}
Beverly Jamison and Stephan Olariu.
\newblock On the homogeneous decomposition of graphs.
\newblock In Ernst~W. Mayr, editor, {\em Graph-Theoretic Concepts in Computer
  Science}, pages 170--183, Berlin, Heidelberg, 1993. Springer Berlin
  Heidelberg.

\bibitem{JO:95}
Beverly Jamison and Stephan Olariu.
\newblock P-components and the homogeneous decomposition of graphs.
\newblock {\em SIAM Journal on Discrete Mathematics}, 8(3):448--463, 1995.

\bibitem{KSS:08}
Dieter Kratsch, Jeremy~P. Spinrad, and R.~Sritharan.
\newblock A new characterization of hh-free graphs.
\newblock {\em Discrete Mathematics}, 308(20):4833--4835, 2008.

\bibitem{LDEM:16}
Manuel Lafond, Riccardo Dondi, and Nadia El-Mabrouk.
\newblock The link between orthology relations and gene trees: a correction
  perspective.
\newblock {\em Algorithms for Molecular Biology}, 11(1):1, 2016.

\bibitem{Lafond2014}
Manuel Lafond and Nadia El-Mabrouk.
\newblock Orthology and paralogy constraints: satisfiability and consistency.
\newblock {\em BMC Genomics}, 15(6):S12, 2014.

\bibitem{Lekkeikerker1962}
C.~Lekkeikerker and J.~Boland.
\newblock Representation of a finite graph by a set of intervals on the real
  line.
\newblock {\em Fundamenta Mathematicae}, 51(1):45--64, 1962.

\bibitem{LS:09}
Min~Chih Lin and Jayme~L. Szwarcfiter.
\newblock Characterizations and recognition of circular-arc graphs and
  subclasses: A survey.
\newblock {\em Discrete Mathematics}, 309(18):5618--5635, 2009.

\bibitem{LIU201245}
Yunlong Liu, Jianxin Wang, Jiong Guo, and Jianer Chen.
\newblock Complexity and parameterized algorithms for cograph editing.
\newblock {\em Theoretical Computer Science}, 461:45--54, 2012.
\newblock 17th International Computing and Combinatorics Conference (COCOON
  2011).

\bibitem{Maffray2003}
F.~Maffray.
\newblock {\em On the coloration of perfect graphs}, pages 65--84.
\newblock Springer New York, New York, NY, 2003.

\bibitem{CS:94}
Ross~M. McConnell and Jeremy~P. Spinrad.
\newblock Linear-time modular decomposition and efficient transitive
  orientation of comparability graphs.
\newblock In {\em Proceedings of the Fifth Annual ACM-SIAM Symposium on
  Discrete Algorithms}, SODA '94, page 536–545, USA, 1994. Society for
  Industrial and Applied Mathematics.

\bibitem{CS:97}
Ross~M McConnell and Jeremy~P Spinrad.
\newblock Linear-time transitive orientation.
\newblock In {\em Proceedings of the eighth annual ACM-SIAM symposium on
  Discrete algorithms}, pages 19--25, 1997.

\bibitem{CS:99}
Ross~M. McConnell and Jeremy~P. Spinrad.
\newblock Modular decomposition and transitive orientation.
\newblock {\em Discrete Mathematics}, 201(1-3):189 -- 241, 1999.

\bibitem{graph-list}
B~McKay.
\newblock Graphs.
\newblock \url{https://users.cecs.anu.edu.au/~bdm/data/graphs.html} (accessed
  Nov 3, 2022).

\bibitem{MEYNIEL:84}
H.~Meyniel.
\newblock The graphs whose odd cycles have at least two chords.
\newblock In C.~Berge and V.~Chv\'atal, editors, {\em Topics on Perfect
  Graphs}, volume~88 of {\em North-Holland Mathematics Studies}, pages
  115--119. North-Holland, 1984.

\bibitem{pnueli_lempel_even_1971}
A.~Pnueli, A.~Lempel, and S.~Even.
\newblock Transitive orientation of graphs and identification of permutation
  graphs.
\newblock {\em Canadian Journal of Mathematics}, 23(1):160–175, 1971.

\bibitem{tralda}
David Schaller, Marc Hellmuth, and Peter~F. Stadler.
\newblock A simpler linear-time algorithm for the common refinement of rooted
  phylogenetic trees on a common leaf set.
\newblock {\em Alg. Mol. Biol.}, 16(1):23, 2021.

\bibitem{TWW:sat}
André Schidler and Stefan Szeider.
\newblock {\em A SAT Approach to Twin-Width}, pages 67--77.
\newblock 2022.

\bibitem{SCHMERL1993191}
James~H. Schmerl and William~T. Trotter.
\newblock Critically indecomposable partially ordered sets, graphs, tournaments
  and other binary relational structures.
\newblock {\em Discrete Mathematics}, 113(1):191--205, 1993.

\bibitem{Seinsche:74}
D~Seinsche.
\newblock On a property of the class of n-colorable graphs.
\newblock {\em Journal of Combinatorial Theory, Series B}, 16(2):191--193,
  1974.

\bibitem{Spinrad:85}
Jeremy Spinrad.
\newblock On comparability and permutation graphs.
\newblock {\em SIAM Journal on Computing}, 14(3):658--670, 1985.

\bibitem{Sumner74}
David~P. Sumner.
\newblock Dacey graphs.
\newblock {\em Journal of the Australian Mathematical Society},
  18(4):492–502, 1974.

\bibitem{TCHP:08}
Marc Tedder, Derek Corneil, Michel Habib, and Christophe Paul.
\newblock Simpler linear-time modular decomposition via recursive factorizing
  permutations.
\newblock In {\em Automata, Languages and Programming}, volume 5125 of {\em
  Lecture Notes in Computer Science}, pages 634--645. Springer Berlin
  Heidelberg, 2008.

\bibitem{TurauWeyer:2015}
Volker Turau and Christoph Weyer.
\newblock {\em Algorithmische Graphentheorie}.
\newblock De Gruyter, Berlin, München, Boston, 2015.

\end{thebibliography}

\end{document}